\documentclass[reqno]{amsart}
\usepackage{geometry}
\usepackage{hyperref}
\usepackage{xcolor}
\usepackage{tikz}
\usetikzlibrary{decorations.pathreplacing}
\usepackage[utf8]{inputenc}
\usepackage[english]{babel}
\usepackage{upgreek}
\usepackage{tipa}
\usepackage{lastpage}
\usepackage{float}
\usepackage[shortlabels]{enumitem}
\usepackage{subcaption}
\usepackage{multirow}
\usepackage{graphicx} 
\graphicspath{ {c:/Users/Owner/Desktop} }
\usepackage{amsmath}
\usepackage{amssymb}
\usepackage{dsfont}
\usepackage{amsthm}
\usepackage{thmtools}
\usepackage{thm-restate}
\allowdisplaybreaks
\theoremstyle{plain}
\newtheorem{thm}{Theorem}[section]
\newtheorem{lem}[thm]{Lemma}
\newtheorem{cor}[thm]{Corollary}

\theoremstyle{definition}

\newtheorem{example}[thm]{Example}
\newtheoremstyle{iremark}
  {\topsep}   
  {\topsep}   
  {\upshape}  
  {0pt}       
  {\itshape}  
  {.}         
  {5pt plus 1pt minus 1pt} 
  {\thmname{#1}\thmnumber{ \itshape#2}\thmnote{ (#3)}} 
\theoremstyle{iremark}
\newtheorem{remark}[thm]{Remark}
\usepackage{enumitem}

\makeatletter
\newcommand*{\defeq}{\mathrel{\rlap{%
                     \raisebox{0.3ex}{$\m@th\cdot$}}%
                     \raisebox{-0.3ex}{$\m@th\cdot$}}%
                    \, =}
\makeatother
\newcommand*{\comma}{,}
\newcommand*{\equal}{=}

\title{Second Class Particle Behaviour in Blocking ASEP}

\author{Daniel Adams} 
\address[Daniel Adams]{The Innovation Game}

\author{M\'arton Bal\'azs}
\address[M\'arton Bal\'azs]{School of Mathematics, University of Bristol, Woodland Road, Bristol, BS8 1UG, UK}
\email{m.balazs@bristol.ac.uk}

\author{Jessica Jay} 
\address[Jessica Jay] {School of Mathematical Sciences, Lancaster University, LA1 4YW, UK.}
\email{j.jay@lancaster.ac.uk}

\date{}

\begin{document}

\begin{abstract}
We consider any fixed $d\in\mathbb{Z}_{>0}$ number of second class particles in the asymmetric simple exclusion process (ASEP), constructed via a basic coupling of two ASEPs. We give the joint distribution of the positions of the second class particles and also the probability of there being a second class particle at a given site, under the natural blocking measure for ASEP. In order to find these distributions we use results about the number of particles in half-infinite and finite site ranges of ASEP. Our investigations also lead to probabilistic proofs of well-known combinatorial identities; the Durfee rectangles identity, Euler's identity, and the $q$-Binomial Theorem.
\end{abstract}

\maketitle
\section{Introduction}\label{intro}
The Asymmetric Simple Exclusion Process (ASEP) is a classic nearest-neighbour interacting particle system on $\mathbb{Z}$, introduced in 1970 by Spitzer \cite{spitzer}. In 1976, Liggett \cite{ligg_coupling} showed that ASEP has reversible stationary measures known as blocking measures. Blocking measures concentrate on a countable set of states where informally speaking an ASEP state can be split into blocks, an infinite block of empty sites to the left, an infinite block of full sites to the right and some block in the middle of empty and full sites. The infinite block of full sites  to the right means that any given particle can only drift so far to the right, i.e.\ it's motion is blocked, hence the name blocking measure. The simplest form of blocking measures are of product structure, however these are not ergodic as they possess a conserved quantity. Conditioning on this quantity gives rise to ergodic blocking measures.

These measures were further investigated for a family of interacting particle systems by Bal\'azs and Bowen in \cite{blocking}; here we return to ASEP only and extend the work of \cite{blocking} by considering the distribution of second class particles. Whilst second class particles have been well studied under non-reversible scenarios, the blocking case has been investigated much less. 

 Second class particles are natural coupling objects that arise from attractive interacting particle systems, like ASEP. The coupled particle systems are referred to as a multi-species particle system. Stationary measures of such systems have drawn lots of interest, for examples in the translation invariant case see Angel \cite{angel_2_species}, Ferrari and Martin \cite{ferr_mart_multi_species} and also, in terms of last passage percolation, Fan and Sepp\"al\"ainen \cite{fan_sepp_joint_bus}. In \cite{belitski_schutz_n_species_asep}, Belitsky and Sch\"utz constructed stationary measures for the $n$-species priority ASEP on a finite integer lattice, they also studied the hydrodynamic limit for this process. 
\par Informally, a two-species ASEP can be described in the following way. Both species of particles attempt to evolve as the single species ASEP would, i.e.\ obeying the exclusion rule within their own species. The two species can interact as follows, a first class particle may jump into the space of a second and hence swap position with the second class particle. However if a second class particle attempts to jump into the space of a first class particle the jump is blocked.

In this paper we give explicit formulas for the distribution of a fixed number of second class particles under both the product and ergodic blocking measures for ASEP. Along the way we identify the distribution for the number of first class particles in finite and semi-infinite intervals in both the product and ergodic blocking scenarios. The well known particle-hole symmetry of ASEP takes an interesting form in this context which we also formulate.  

A.\ Bufetov and K.\ Chen, \cite{bufetov_chen}, about the same time as our manuscript, achieved distributional results on second and higher class particles via arguments involving the Mallows measure. Instead, our second class particle arguments rely on an interpretation of the two species ASEP as a single species ASEP with a second ASEP on the particle labels using the first as a dynamic background. It turns out that both these ASEPs are in their stationary blocking distributions, which allows us to perform calculations. Both for the distribution of the number of particles and positions of second class particles, we use probabilistic arguments that lead to various recursions. Finding the solutions of these gives rise to our proofs. 

Blocking measures have a special algebraic structure which allowed Bal\'azs and Bowen \cite{blocking} to derive a probabilistic proof of the Jacobi triple product identity. In that paper the authors considered ASEP and its equivalent description in terms of the Asymmetric Zero-Range Process. Comparing the blocking measures for each process led to their result. In 2022, this method was generalised by Bal\'azs, Fretwell and Jay \cite{MDJ}, where new identities of combinatorial significance were proved by considering blocking measures of more general particle systems. Other combinatorial identities were found by Jay and Lees \cite{jay_lees_ising}, by considering a family of inhomogeneous Ising chains and equivalent family of nearest neighbour particle systems.

As a byproduct of our methods we give purely probabilistic proofs to further well-known and non-trivial combinatorial identities. The ones we cover are the Durfee rectangles identity, Euler's identity, the $q$-Binomial Theorem and an identity relating to Dyson's crank, each fundamental to the theory of integer partitions. Probabilistic methods are certainly not the common way of proving such identities, which gives a unique perspective to this paper.

\subsection{Notation}\label{subsect: notation}~
\par Throughout the paper we will use the $q$-Pochhammer symbol as well as the $q$-Binomial coefficient. For some $a\in\mathbb{R}$ and $n \in \mathbb{Z}_{>0}$, the $q$-Pochhammer symbol is given by
$(a;q)_n\defeq \prod\limits_{i=0}^{n-1}(1-aq^i)$.
Also we set, $(a;q)_0\defeq 1$ and $(a;q)_\infty\defeq\prod\limits_{i=0}^{\infty}(1-aq^i)$.
The $q$-Binomial coefficient is then, 
$$\begin{bmatrix}
n\\
m
\end{bmatrix}_q\defeq \frac{(q;q)_n}{(q;q)_m(q;q)_{n-m}}=\prod\limits_{i=0}^{m-1}\frac{(1-q^{n-i})}{(1-q^{i+1})}.$$
Sometimes we will use the shorthand notation $(x)^+$ to mean $\max\{x,0\}$.
\subsection{Results}\label{subsect: intro- results}~
\subsubsection{Distributional results for single species ASEP }\label{subsubsect: intro - dist results}~
\vspace{3mm}\par There is a one-parameter ($c\in\mathbb{R}$) family of stationary reversible measures for ASEP, called blocking measures. These measures are product measures where the occupation at each site is distributed as an independent Bernoulli with parameter dependent on the asymmetry parameter $q\in(0,1)$ and the site $i$.  In particular, the probability a site $i$ is occupied is given by $\frac{q^{c-i}}{1+q^{c-i}}$.  We denote these blocking measures by $\underline{\mu}^c$ for any $c\in\mathbb R$. We also let
$N^{\overset{\leftarrow}{p}}_{m+\frac{1}{2}}(\underline{z})$ denote the number of particles at and to the left of site $m$ in the state $\underline{z}$. We will derive the distribution under the blocking measures of a given number of particles on a half infinite volume, i.e. on the sites $(-\infty,m]$ for some $m\in\mathbb{Z}$.
\begin{restatable}{thm}{partL}\label{thm: half infinite asep particles}
For any $m \in \mathbb{Z}$, $k\in\mathbb{Z}_{\geq 0}$ and $c \in \mathbb{R}$, 
$$\underline{\mu}^c(\{N^{\overset{\leftarrow}{p}}_{m+\frac{1}{2}}(\underline{z})=k\})=\frac{q^{k(c-m)+\frac{k(k-1)}{2}}}{(q;q)_k(-q^{c-m};q)_\infty}.$$
\end{restatable}
\vspace{3mm}\par It is well known that the following quantity is conserved by ASEP dynamics and is finite $\underline{\mu}^c$-a.s.,
$$N(\underline{z})=\sum\limits_{i=1}^\infty (1-z_i) -\sum\limits_{i=-\infty}^0 z_i.$$ This gives an ergodic decomposition of the state space (into states of a given conserved quantity value), and we denote the unique ergodic stationary measure on the set of states with conserved quantity $n$ by $\underline{\nu}^n(\cdot)=\underline{\mu}^c(\cdot | \{N=n\})$. We will also derive the distribution of the number of particles on a half-infinite volume under the ergodic measures $\underline{\nu}^n$ for any $n\in\mathbb{Z}$.

\begin{restatable}{thm}{nupartL}\label{thm: half infinite ASEP particles under nu n}~
    \par \noindent For any $n,m,k \in \mathbb{Z}$, such that $k\geq (m-n)^+$,
    
    $$\underline{\nu}^n(\{N^{\overset{\leftarrow}{p}}_{m+\frac{1}{2}}(\underline{z})=k\})=\frac{ q^{k(n-m+k)}(q;q)_\infty}{(q;q)_k(q;q)_{n-m+k}}.$$    
\end{restatable}
\par  We also consider the distribution under the measures $\underline{\mu}^c$ and $\underline{\nu}^n$ of a given number of particles in a finite volume, i.e.\ in the site range $[m_1+1,m_2-1]$ for any $m_1+1<m_2 \in \mathbb{Z}$. 
\begin{restatable}{lem}{finparts}\label{lem: finite asep particles}
For any $m_1+1<m_2 \in \mathbb{Z},c \in \mathbb{R}$ and $k \in \{0,1,...,\hat{m}_2\}$ where $\hat{m}_2:=m_2-m_1-1$, 
$$\underline{\mu}^c(\{N^{\overset{\leftarrow}{p}}_{m_2-1+\frac{1}{2}}(\underline{z})-N^{\overset{\leftarrow}{p}}_{m_1+\frac{1}{2}}(\underline{z})=k\})=\frac{q^{k(c+1-m_2)+\frac{k(k-1)}{2}}}{(-q^{c-m_2+1};q)_{\hat{m}_2}}\begin{bmatrix}
    \hat{m}_2 \\
    k
    \end{bmatrix}_{q}.$$
    
\end{restatable}

\begin{restatable}{lem}{nufinparts}\label{lem: finite ASEP particles under nu n}
    For any $n\in\mathbb{Z}$, any $m_1+1<m_2 \in \mathbb{Z},c \in \mathbb{R}$ and $k \in \{0,1,...,\hat{m}_2\}$ where $\hat{m}_2:=m_2-m_1-1$, 
    \begin{align*}
        \underline{\nu}^n(\{N^{\overset{\leftarrow}{p}}_{m_2-1+\frac{1}{2}}(\underline{z})&-N^{\overset{\leftarrow}{p}}_{m_1+\frac{1}{2}}(\underline{z})=k\})\\
        &=q^{k(n+1+k-m_2)}(q;q)_\infty\cdot\begin{bmatrix}
        \hat{m}_2 \\
        k
        \end{bmatrix}_{q}\cdot \sum\limits_{j=(m_2-1-n-k)^+}^\infty\frac{q^{j(n+1-m_2+k+j)+\hat{m}_2j}}{(q;q)_j(q;q)_{n+1-m_2+k+j}}.
    \end{align*}
    
\end{restatable}

\subsubsection{Distributional results for two species ASEP}\label{subsubsect: intro - second cp results}~
\vspace{3mm}\par We define an ASEP with some $d>0$ second class particles by considering a basic coupling of two ASEPs, $\{\underline{\eta},\underline{\xi}\}$.
At any time, $\underline{\eta}$ and $\underline{\xi}$ differ in exactly $d$ places (at these sites $\xi_i>\eta_i$), these differences define the $d$ second class particles. In this way we think of the process $\underline{\xi}$ as an ASEP with $d$ second class particles. Using this coupling we can give an alternative proof of Theorem \ref{thm: half infinite asep particles} (see Section \ref{subsect: asep label process}). This proof follows by studying the stationary distribution of the label process for the $d$ second class particles (given in Theorem \ref{thrm: label process stationary dist}). The label of a particle is defined to be the number of particles to its left, which by the blocking nature of the process will be finite. In order to compute the stationary distribution of the label process we use probabilistic and dynamic arguments. First we show that the label process can be viewed as the positions of particles in an ASEP on $\mathbb{Z}_{\geq 0}$ with exactly $d$ particles (meaning closed left boundary). The dynamics of this ASEP depends on the state of $\underline{\xi}$, jumps are suppressed if the particles with the corresponding labels in $\underline{\xi}$ are not nearest neighbours. We then use the fact that if a Markov chain has a reversible stationary measure, the cut Markov chain (where some edges in its transition diagram are removed) is also reversible stationary with respect to that distribution.
\par Using the basic coupling, we find the probability that a second class particle is at a given site, $m$. 
\begin{restatable}{thm}{scpm}\label{thrm: second class particle at a given site}
Assume that the coupled system is stationary and $\underline{\xi}\sim\underline{\mu}^c$. For a given $m\in\mathbb{Z}$, we find a second class particle at site $m$ with probability, 
$$\mathbb{P}(\xi_m>\eta_m)=\frac{(1-q^d)q^{c-m}}{(1+q^{c-m})(1+q^{c+d-m})}.$$
\end{restatable}
We also give this distribution under the ergodic measure $\underline{\nu}^n$.
\begin{restatable}{prop}{nuscpm}\label{prop: second class particle at a given site nu^n}
    Assume that the coupled system is stationary and $\underline{\xi}\sim\underline{\nu}^n$. For a given $m\in\mathbb{Z}$, we find a second class particle at site $m$ with probability, 
$$\mathbb{P}_{\underline{\nu}^n}(\xi_m>\eta_m)=(q;q)_\infty\left\{\sum\limits_{k=(m-n-1)^+}^\infty\frac{q^{(k+1)(n-m+k+1)}}{(q;q)_k(q;q)_{n-m+k+1}}-\sum\limits_{k=(m-n-d-1)^+}^\infty\frac{q^{(k+1)(n+d-m+k+1)}}{(q;q)_k(q;q)_{n+d-m+k+1}}\right\}.$$
\end{restatable}
We consider the case when there is a single second class particle ($d=1$). By conditioning on the number of particles to the left of the second class particle (i.e.\ it's label) we find another formula for the distribution of the second class particles position under the ergodic measure $\underline{\nu}^n$.
\begin{restatable}{prop}{nuscpmother}\label{prop: position of single 2cp under nu}
    Consider the coupling $(\underline{\eta},\underline{\xi})$ when $d=1$ (i.e.\ a single second class particle in ASEP). Assume the coupled system is stationary with $\underline{\xi} \sim \underline{\nu}^n$. Then for any $m\in\mathbb{Z}$ the probability that the second particle is at site $m$ is, 
    $$\mathbb{P}_{\underline{\nu}^n}(X=m)=(1-q)(q;q)_\infty \sum\limits_{k=(m-n-1)^+}^\infty \frac{q^{k(n-m+k+1)+k+(n-m+k+1)}}{(q;q)_k(q;q)_{n-m+k+1}}.$$
\end{restatable}

 We also consider the stationary distribution of the $d$ second class particle positions process, for any $d>0$. When the ASEP with $d$ second class particles is stationary under the blocking measure $\underline{\mu}^c$ we find the following.
\begin{restatable}{thm}{scpX}\label{thrm: second class positions dist}
Assuming that the coupled system is stationary and $\underline{\xi}\sim \underline{\mu}^c$, the stationary distribution of $\underline{X}$, the position of $d>0$ second class particles, is given by, 
$$\mathbb{P}(\underline{X}=\underline{m})=\frac{\prod\limits_{i=1}^d(1-q^i)\cdot q^{dc-\sum\limits_{j=1}^dm_j}}{\prod\limits_{j=1}^d(1+q^{c+d-j-m_j})(1+q^{c+d+1-j-m_j})}, \hspace{5mm} \text{for all } m_1<m_2<...<m_d\in\mathbb{Z}.$$
\end{restatable}
\par \noindent Of course, when $d=1$ Theorems \ref{thrm: second class particle at a given site} and \ref{thrm: second class positions dist} are the same. In this case, the position of a single second class particle under $\underline{\mu}^c$ is distributed according to $-Y$ where $Y$ is a discrete logistic with parameters $q$ and $-c$ (introduced by Chakraborty and Chakravarty in 2016 \cite{discrete_logistic}). 
\par In order to prove Theorem \ref{thrm: second class positions dist}, we first condition on the label process for the second class particles. Then we split states up into the sites to the left of site $m_1$ and for $j\in\{1,...,d-1\}$ the sites between $m_j$ and $m_{j+1}$. From here we use the distributional results for particles in half infinite and finite volumes to complete the proof.

\par By following similar reasoning we also find the joint distribution of the position of any $d>0$ given second class particles in ASEP under the ergodic measures $\underline{\nu}^n$ (for any $n\in\mathbb{Z}$); that is conditioned on the of the conserved quantity $N(\underline{\xi})=n$.
\begin{restatable}{prop}{dundernu}\label{prop: d 2cps under nu}
    Assume that the coupled system is stationary and $\underline{\xi}\sim \underline{\nu}^n$. For any $\underline{m}=(m_1,\cdots,m_d)$ such that $m_1<m_2<\cdots m_d\in\mathbb{Z}$, the probability that the positions of the second class particles, $\underline{X}$, is $\underline{m}$ is given by,
\begin{align*}
    \mathbb{P}_{\underline{\nu}^n}(\underline{X}=\underline{m})=q^{-\sum\limits_{i=1}^dm_i}(q;q)_\infty\prod\limits_{i=1}^d(1-q^i)\sum\limits_{\substack{\underline{k}\in Z_+^d:\\\forall j\in\{1,2,\cdots,d\},\\ k_j\geq m_j-n-1}} &\frac{q^{\frac{k_1(k_1+1)}{2}+k_1(d-m_1)+\frac{k_d(k_d+1)}{2}+(k_d+1)(n+1)}}{(q;q)_{n-m_d+k_d+1}(q;q)_{k_1}}\\
    &\hspace{10mm} \cdot\prod\limits_{j=2}^dq^{\hat{k}_j(d+1-j-m_j)+\frac{\hat{k}_j(\hat{k}_j+1)}{2}}\cdot \begin{bmatrix}
        \hat{m}_j\\
        \hat{k}_j
    \end{bmatrix}_q,
\end{align*}
where $\hat{m}_j\defeq m_j-m_{j-1}-1$ and $\hat{k}_j\defeq k_j-k_{j-1}-1$ for each $j\in\{2,\cdots,d\}$.
\end{restatable}

\subsubsection{Combinatorial Identities}\label{subsubsect: intro - identities}~
\par \vspace{3mm}  By considering very natural distributional questions for single species ASEP (including those from Section \ref{subsubsect: intro - dist results}), under its natural blocking measures, we give probabilistic proofs to classical combinatorial identities (Section \ref{subsect: identities}). For completeness, we also discuss these identities' combinatorial meanings in Section \ref{sect: combinatorics}. 
\par In Section \ref{subsect: ASEP symmetry}, we look at the symmetry between particles and holes in ASEP under its blocking measure. Using this symmetry we prove a well-known combinatorial identity, the Durfee rectangles identity.
\begin{restatable}{thm}{durfee}\emph{\textbf{Durfee Rectangles Identity }(for example see equation (4) in Gessel \cite{Ges_durfee_rectangle})}
\label{thrm: identity from symmetry}
\par \noindent For $q\in(0,1)$ and any fixed $n\in\mathbb{Z}$,
    $$\frac{1}{(q;q)_\infty}=\sum\limits_{k=\max\{-n,0\}}^\infty\frac{q^{k(n+k)}}{(q;q)_{n+k}(q;q)_k}.$$
\end{restatable}
\par \noindent Combinatorially this identity identity says that for a fixed integer $n$, any integer partition has a maximal rectangle of side lengths $n+k$ and $k$ (for some $k\geq \max\{-n,0\}$) in its Ferrers diagram known as its Durfee rectangle. When $n=0$, this identity is the well-known Durfee square identity (for example see Andrews and Eriksson \cite{integer_partitions}).
\par In Section \ref{subsect: ASEP dist results} we give the distribution of a single site under the ergodic measure $\underline{\nu}^n$ for any $n$. Using this distribution we prove the following combinatorial identity that looks similar to the Durfee rectangle identity.
\newpage\begin{restatable}{prop}{crankiden}\label{prop: identitiy related to durfee}~
    \par \noindent For any $n\in\mathbb{Z}$, $q\in(0,1)$ we have that, 
    $$\sum\limits_{k=\max\{-n,0\}}^\infty\frac{q^{(k+1)(n+k)}}{(q;q)_k(q;q)_{n+k}}+\sum\limits_{k=\max\{1-n,0\}}^\infty\frac{q^{k(n+k)}}{(q;q)_k(q;q)_{n-1+k}}=\frac{1}{(q;q)_\infty}.$$
\end{restatable}
In combinatorics this identity is an equivalence of generating functions for integer partitions by splitting them depending on the value of a certain associated quantity called the crank (this is discussed more in Section \ref{sect: combinatorics}). 
\par\noindent  We also prove the following results. 
\par \begin{restatable}{thm}{euler}\emph{\textbf{Euler's Identity }(for example see equation E1 in Andrews \cite{Andrews_jacobi})}\label{thrm: Euler}
\par \noindent For $q\in(0,1)$ and $z\in\mathbb{R}_{>0}$, 
$$\sum\limits_{k=0}^\infty \frac{q^{\frac{k(k-1)}{2}}z^k}{(q;q)_k}=(-z;q)_\infty.$$
\end{restatable}
\par\vspace{2mm} \noindent In combinatorics, this identity gives equivalent ways of writing the generating function for integer partitions into distinct parts. In a general term of the expansion, the power of $z$ gives the number of parts in the partition.

\begin{restatable}{thm}{qbin}\emph{\textbf{$q$-Binomial Theorem} (Heine \cite{heine}, also Andrews \cite{andrews_q_bin})}\label{thrm: q-Bin}
\par \noindent For $q\in(0,1)$, $z\in\mathbb{R}_{>0}$ and any $m \in \mathbb{Z}_{\geq 0}$, 
$$\sum\limits_{k=0}^mq^{\frac{k(k-1)}{2}}z^k\begin{bmatrix}
m\\
k
\end{bmatrix}_{q}=(-z;q)_m.$$
\end{restatable}
\par\vspace{2mm} \noindent This identity gives equivalent forms of the generating function for integer partitions into distinct parts of size at most $m$. Again here in a general term in the expansion, the power of $z$ is the number of parts in the partition. Euler's identity can be thought of as the limiting identity when we take $m\rightarrow\infty$ in the $q$-Binomial theorem. 
\par The proofs we give here are purely probabilistic using the blocking measures for ASEP and give natural probabilistic interpretations to these combinatorial identities. It is the probabilistic nature that gives the restriction on the values of $p$ and $z$ in the above identities; we note that these can be extended by analytic continuation.

\section*{Acknowledgements}
\par \vspace{1mm} \noindent The authors would like to thank Dan Fretwell for helpful discussions regarding viewing particle states as integer partitions and the combinatorial identities we have found. Also many thanks to Alexey Bufetov for discussions about the multi-species ASEP. Thank you to Edward Crane for highlighting the link between our results (Corollaries \ref{cor: 1 displacement} and \ref{cor: d displacements}) and the results of Gnedin and Olshanski \cite{gnedin_olshanski}. We would also like to thank an anonymous referee for helpful comments and suggestions. 
\par \noindent M.\ Bal\'azs was partially supported by the EPSRC EP/R021449/1 and EP/W032112/1 Standard Grants of the UK. J.\ Jay was partially funded by the Heilbronn Institute for Mathematical Research.
\par\noindent This study did not involve any underlying data.

\section{The asymmetric simple exclusion process (ASEP)}\label{sect: ASEP}
The Asymmetric Simple Exclusion Process (ASEP) is a well known nearest neighbour particles system, introduced by Spitzer \cite{spitzer}. In this paper we will consider ASEP on some state space $\Omega\subset\{0,1\}^\mathbb{Z}$.
\par To define $\Omega$ we consider the following functions $N_{m+\frac{1}{2}}^{\overset{\leftarrow}{p}},N_{m+\frac{1}{2}}^{\overset{\rightarrow}{h}}:\{0,1\}^\mathbb{Z} \rightarrow \mathbb{Z}_{\geq 0}\cup \{\infty\}$  for any $m \in \mathbb{Z}$ defined by,
$$N_{m+\frac{1}{2}}^{\overset{\leftarrow}{p}}(\underline{\eta})\defeq \sum\limits_{i=\ell}^m \eta_i\hspace{5mm} \textrm{ and } \hspace{5mm} N_{m+\frac{1}{2}}^{\overset{\rightarrow}{h}}(\underline{\eta})\defeq\sum\limits_{i=m+1}^\mathfrak{r}(1-\eta_i),$$
\par \noindent that is the number of `particles' to the left and `holes' to right of $m+\frac{1}{2}$ respectively. When $m=0$ we simplify notation to $N^{\overset{\leftarrow}{p}}$ and $N^{\overset{\rightarrow}{h}}$. Then $\Omega$ is given by, 
$$\Omega\defeq \{\underline{\eta}\in\{0,1,\}^\mathbb{Z}:N^{\overset{\leftarrow}{p}}(\underline{\eta})<\infty \text{ and }^{\overset{\rightarrow}{h}}<\infty\}.$$
\par The process evolves by particles attempting to make nearest neighbour jumps. Below we give the Markov generator for this process. First we define $\underline{\eta}^{(i,j)}$ denote the state reached from $\underline{\eta}\in \Omega$ by a particle jumping from site $i$ to site $j$ (when this is possible), that is, $$\left(\underline{\eta}^{(i,j)}\right)_k=\begin{cases}
\eta_k &\textrm{if } k \neq i,j\\
\eta_i-1 &\textrm{if } k=i \\
\eta_j+1 &\textrm{if }k=j.
\end{cases}$$
Then the generator is of the form,
$$\left(L\varphi\right)(\underline{\eta})=\sum\limits_{i=\ell}^{\mathfrak{r}-1}\left\{\eta_i(1-\eta_{i+1})\left(\varphi(\underline{\eta}^{(i,i+1)})-\varphi(\underline{\eta})\right)+q\eta_{i+1}(1-\eta_i)\left(\varphi(\underline{\eta}^{(i+1,i)})-\varphi(\underline{\eta})\right)\right\}$$
\par \noindent for some cylinder function $\varphi:\Omega \rightarrow \mathbb{R}$ and some asymmetry parameter $q\in(0,1)$.
\par 
 ASEP has a one parameter family of reversible stationary measures, known as blocking measures, which were first found by Liggett in 1976 \cite{ligg_coupling} (Theorem 1.4). The natural product blocking measure for any $c\in\mathbb{R}$ is given by,
$$\underline{\mu}^c(\underline{\eta})=\prod\limits_{i=-\infty}^\infty\frac{q^{-(i-c)\eta_i}}{1+q^{-(i-c)}}=\prod\limits_{i=-\infty}^0\frac{q^{-(i-c)\eta_i}}{1+q^{-(i-c)}}\prod\limits_{i=1}^\infty \frac{q^{(i-c)(1-\eta_i)}}{1+q^{(i-c)}}.$$ 
\par The state space is not irreducible but decomposes into irreducible components according to a conserved quantity. For any $m\in\mathbb{Z}$ the quantity, 
$$N_{m+\frac{1}{2}}(\underline{\eta})\defeq N^{\overset{\rightarrow}{h}}_{m+\frac{1}{2}}(\underline{\eta})-N^{\overset{\leftarrow}{p}}_{m+\frac{1}{2}}(\underline{\eta})=\sum\limits_{i=m+1}^\infty (1-\eta_i)-\sum\limits_{i=-\infty}^m \eta_i$$
is finite by definition and conserved by the dynamics of the process. As before when $m=0$ we simplify notation to $N$. Then, $$\Omega=\bigcup_{n \in \mathbb{Z}}\Omega^n \hspace{2mm} \text{ where } \hspace{2mm}\Omega^n\defeq\{\underline{\eta} \in \Omega: N(\underline{\eta})=n\}.$$
We note that the left shift operator, $\tau$ such that $(\tau\underline{\eta})_i=\eta_{i+1}$, defines a bijection $\Omega^n \xrightarrow[]{\tau} \Omega^{n-1}$ (i.e if $\underline{\eta} \in \Omega^n$ then, $N(\tau \underline{\eta})=n-1$). 
\par We can then find the unique stationary distribution on $\Omega^n$,
$$\underline{\nu}^n(\cdot)\defeq\underline{\mu}^c(\cdot|N(\cdot)=n)=\frac{\underline{\mu}^c(\cdot)\mathbb{I}\{\cdot\in\Omega^n\}}{\underline{\mu}^c(\{N=n\})}.$$
The results of \cite{ligg_coupling} (Theorem 1.4) and \cite{blocking} (Corollary 6.2 and Equation 6.2) give us that, 
\begin{equation}\label{eq: mu of N}
 \underline{\mu}^c(\{N=n\})=\frac{q^{\frac{n(n+1)}{2}-nc}}{\sum\limits_{\ell=-\infty}^\infty q^{\frac{\ell(\ell+1)}{2}-\ell c}},   
\end{equation}
\vspace{-1mm} and so, 
$$\underline{\nu}^n(\underline{\eta})=\frac{\sum\limits_{\ell=-\infty}^\infty q^{\frac{\ell(\ell+1)}{2}-\ell c}}{q^{\frac{n(n+1)}{2}-nc}}\prod\limits_{i=-\infty}^0\frac{q^{-(i-c)\eta_i}}{1+q^{-(i-c)}}\prod\limits_{i=1}^\infty \frac{q^{(i-c)(1-\eta_i)}}{1+q^{i-c}}\cdot \mathbb{I}\{N(\underline{\eta})=n\}.$$
Note that this is in fact independent of the parameter $c$.

\vspace{2mm}\par It is natural to ask about the distribution of the number of particles or holes in a certain region of sites under both $\mu^c$ and $\nu^n$. In Section \ref{subsect: ASEP dist results} we will answer this for any half-infinite or finite range of sites. Before this we will prove a particle-hole symmetry for ASEP which will allow us to compare the number of particles to the left of a site with the number of holes to the right (Corollary \ref{cor: c change between holes and particles in ASEP}). 

\subsection{Symmetry of ASEP}\label{subsect: ASEP symmetry}~
\par The following lemma gives a precise symmetry between the particles and holes in ASEP.
\begin{lem}\label{lem: symmetry holes to particles}
    Let $\underline{\eta}$ be an ASEP with right drift and blocking measure (for any $c\in\mathbb{R}$),
$$\underline{\mu}^c(\underline{\eta})=\prod\limits_{i=-\infty}^\infty\frac{q^{-(i-c)\eta_i}}{1+q^{-(i-c)}}.$$
Consider the process given by $\hat{\underline{\eta}}\defeq \underline{1}-\underline{\eta}$. This process is an ASEP with left drift on the state space 
$$\hat{\Omega}\defeq\{\underline{z}\in\{0,1\}^\mathbb{Z}: \sum\limits_{i=1}^\infty z_i, \sum\limits_{i=-\infty}^0 (1-z_i)<\infty\}$$
and has stationary measure (for any $c\in\mathbb{R}$), 
$$\hat{\underline{\mu}^c}(\hat{\underline{\eta}})=\prod\limits_{i=-\infty}^\infty\frac{q^{(i-c)\hat{\eta}_i}}{1+q^{i-c}}.$$
Let $N^{\overset{\rightarrow}{p}}_{m+\frac{1}{2}}(\underline{z})\defeq \sum\limits_{i=m+1}^\infty z_i$. Then for any $m\in\mathbb{Z}$, $k\in\mathbb{Z}_{\geq 0}$, and $c\in\mathbb{R}$, 
$$\underline{\mu}^c(\{N^{\overset{\rightarrow}{h}}_{m+\frac{1}{2}}(\underline{z})=k\})=\hat{\underline{\mu}}^c(\{N^{\overset{\rightarrow}{p}}_{m+\frac{1}{2}}(\underline{z})=k\}).$$
\end{lem}
\begin{proof}
By definition of $\hat{\underline{\eta}}$, particles in $\hat{\underline{\eta}}$ move like holes in $\underline{\eta}$. That is, particles jump to the left at rate $1$ and to the right at rate $q$, unless the jump is blocked. Thus, the probability there is a particle at site $i$, is exactly the probability that there isn't in $\underline{\eta}$ and similarly for when there isn't a particle in $\hat{\underline{\eta}}$. So we have that $\hat{\underline{\mu}}^c=\bigotimes\limits_{i\in\mathbb{Z}}\hat{\mu}_i^c$ given by, 
\begin{align*}
    \hat{\mu}^c_i(1)&=\mu^c_i(0)=\frac{1}{1+q^{-(i-c)}}=\frac{q^{i-c}}{1+q^{i-c}}\\
    \hat{\mu}_i^c(0)&=\mu^c_i(1)=\frac{q^{-(i-c)}}{1+q^{-(i-c)}}=\frac{1}{1+q^{i-c}}
\end{align*}
is reversible stationary for this ASEP with left drift.
Hence if we define, $N^{\overset{\rightarrow}{p}}_{m+\frac{1}{2}}(\underline{z})\defeq \sum\limits_{i=m+1}^\infty z_i$, it immediately follows that  for any $m\in\mathbb{Z}$, $k\in\mathbb{Z}_{\geq 0}$, and $c\in\mathbb{R}$, 
$$\underline{\mu}^c(\{N^{\overset{\rightarrow}{h}}_{m+\frac{1}{2}}(\underline{z})=k\})=\hat{\underline{\mu}}^c(\{N^{\overset{\rightarrow}{p}}_{m+\frac{1}{2}}(\underline{z})=k\}).$$ 

\end{proof}
We will now see that considering particles to the right of say $m+\frac{1}{2}$ in $\hat{\underline{\eta}}$ is similar to considering particles to the left of $m+\frac{1}{2}$ in $\underline{\eta}$.
\begin{lem}\label{lem: symmetry particles to left and right}
  For any $m \in\mathbb{Z}$, $k\in\mathbb{Z}_{\geq 0}$, and $c\in\mathbb{R}$, 
  $$\hat{\underline{\mu}}^c(\{N^{\overset{\rightarrow}{p}}_{m+\frac{1}{2}}(\underline{z})=k\})=\underline{\mu}^{2m+1-c}(\{N^{\overset{\leftarrow}{p}}_{m+\frac{1}{2}}(\underline{z})=k\}).$$
\end{lem}
\begin{proof}
We have that, 
\begin{align*}
    &\hat{\underline{\mu}}^c(\{N^{\overset{\rightarrow}{p}}_{m+\frac{1}{2}}(\underline{z})=k\})=\sum\limits_{\substack{\underline{z}\in\hat{\Omega}:\\ \sum\limits_{i=m+1}^\infty z_i=k}}\prod\limits_{i=m+1}^\infty\hat{\mu}_i^c(z_i)=\sum\limits_{\substack{\underline{z}\in\hat{\Omega}:\\ \sum\limits_{i=m+1}^\infty z_i=k}}\prod\limits_{i=0}^\infty\hat{\mu}_{m+1+i}^c(z_{m+1+i})\\
    \text{and}&\\
    &\underline{\mu}^{2m+1-c}(\{N^{\overset{\leftarrow}{p}}_{m+\frac{1}{2}}(\underline{z})=k\})=\sum\limits_{\substack{\underline{z}\in\Omega:\\ \sum\limits_{i=-\infty}^m z_i=k}}\prod\limits_{i=-\infty}^ m\mu_i^{2m+1-c}(z_i)=\sum\limits_{\substack{\underline{z}\in\Omega:\\ \sum\limits_{i=-\infty}^m z_i=k}}\prod\limits_{i= 0}^\infty\mu_{m-i}^{2m+1-c}(z_{m-i}).
\end{align*}
We see that for $i\geq 0$ the marginals $\mu^{2m+1-c}_{m-i}$ and $\hat{\mu}^c_{m+1+i}$ agree, for $z\in\{0,1\}$
$$\mu^{2m+1-c}_{m-i}(z)=\frac{q^{-(m-i-(2m+1-c))z}}{1+q^{-(m-i-(2m+1-c))}}=\frac{q^{(m+1+i-c)z}}{1+q^{m+1+i-c}}=\hat{\mu}^c_{m+1+i}(z).$$
Thus for any $m \in\mathbb{Z}$, $k\in\mathbb{Z}_{\geq 0}$ and $c\in\mathbb{R}$, 
$$\hat{\underline{\mu}}^c(\{N^{\overset{\rightarrow}{p}}_{m+\frac{1}{2}}(\underline{z})=k\})=\underline{\mu}^{2m+1-c}(\{N^{\overset{\leftarrow}{p}}_{m+\frac{1}{2}}(\underline{z})=k\}).$$
\end{proof}
By combining Lemmas \ref{lem: symmetry holes to particles} and \ref{lem: symmetry particles to left and right} we have the following result.
\begin{cor}\label{cor: c change between holes and particles in ASEP}
For any $m\in\mathbb{Z}$, $k \in\mathbb{Z}_{\geq 0}$, and $c\in\mathbb{R}$,
$$\underline{\mu}^c(\{N^{\overset{\rightarrow}{h}}_{m+\frac{1}{2}}(\underline{z})=k\})=\underline{\mu}^{2m+1-c}(\{N^{\overset{\leftarrow}{p}}_{m+\frac{1}{2}}(\underline{z})=k\}).$$
\end{cor}
\subsection{Distributional Results}\label{subsect: ASEP dist results}~
\par The blocking measures for ASEP are a one parameter family with parameter $c\in\mathbb{R}$. We will now see that this parameter is closely linked to shifting states. Firstly if we shift the entire system this corresponds to an integer shift in the parameter $c$.
\begin{lem} \label{lem: shift by tau is a shift in c}
For any $c \in \mathbb{R}$ and $\underline{\eta}\in \Omega$
$$\underline{\mu}^c(\tau\underline{\eta})=\underline{\mu}^{c+1}(\underline{\eta}).$$
\end{lem}
\begin{proof}
Writing the blocking measure compactly we have, $$\underline{\mu}^c(\tau\underline{\eta})=\prod\limits_{i=-\infty}^\infty\frac{q^{(c-i)(\tau\underline{\eta})_i}}{1+q^{c-i}}=\prod\limits_{i=-\infty}^\infty\frac{q^{(c-i)\eta_{i+1}}}{1+q^{c-i}}=\prod\limits_{j=-\infty}^\infty\frac{q^{(c-(j-1))\eta_j}}{1+q^{-c-(j-1)}}=\prod\limits_{j=-\infty}^\infty\frac{q^{((c+1)-j)\eta_j}}{1+q^{c+1-j}}=\underline{\mu}^{c+1}(\underline{\eta}).$$
\end{proof}
Using this fact we see that shifting the site at which we are interested the distribution of the quantities $N^{\overset{\leftarrow}{p}}_{\bullet}$, $N^{\overset{\rightarrow}{h}}_{\bullet}$ any $N_\bullet$ also corresponds to a shift in the parameter $c$.
\begin{cor}\label{cor: shift in c shifts the limits in N}
For any $c \in \mathbb{R}$ and $m,k \in \mathbb{Z}$,
\begin{enumerate}[a)]
    \item $\underline{\mu}^c(\{N^{\overset{\leftarrow}{p}}_{m+\frac{1}{2}}(\underline{z})=k\})=\underline{\mu}^{c+1}(\{N^{\overset{\leftarrow}{p}}_{m+1+\frac{1}{2}}(\underline{z})=k\})$
    \item $\underline{\mu}^c(\{N^{\overset{\rightarrow}{h}}_{m+\frac{1}{2}}(\underline{z})=k\})=\underline{\mu}^{c+1}(\{N^{\overset{\rightarrow}{h}}_{m+1+\frac{1}{2}}(\underline{z})=k\})$
    \item $\underline{\mu}^c(\{N_{m+\frac{1}{2}}(\underline{z})=k\})=\underline{\mu}^{c+1}(\{N_{m+1+\frac{1}{2}}(\underline{z})=k\})$
\end{enumerate}
\end{cor}
\begin{proof}
Firstly we note that for any $r \in \mathbb{Z}$, 
\begin{align*}
N^{\overset{\rightarrow}{p}}_{r+\frac{1}{2}}(\tau\underline{\eta})&=\sum\limits_{i=-\infty}^r \eta_{i+1}=\sum\limits_{i=-\infty}^{r+1}\eta_i=N^{\overset{\rightarrow}{p}}_{r+1+\frac{1}{2}}(\underline{\eta})\\
N^{\overset{\rightarrow}{h}}_{r+\frac{1}{2}}(\tau\underline{\eta})&=\sum\limits_{i=r+1}^\infty(1-\eta_{i+1})=\sum\limits_{i=r+2}^\infty(1-\eta_i)=N^{\overset{\rightarrow}{h}}_{r+1+\frac{1}{2}}(\underline{\eta}) \\
N_{r+\frac{1}{2}}(\tau\underline{\eta})&=N^{\overset{\rightarrow}{h}}_{r+\frac{1}{2}}(\tau\underline{\eta})-N^{\overset{\rightarrow}{p}}_{r+\frac{1}{2}}(\tau\underline{\eta})
=N^{\overset{\rightarrow}{h}}_{r+1+\frac{1}{2}}(\underline{\eta})-N^{\overset{\rightarrow}{p}}_{r+1+\frac{1}{2}}(\underline{\eta})=N_{r+1+\frac{1}{2}}(\underline{\eta}).
\end{align*}
\par \noindent  So we have that, 
\begin{align*}
  \underline{\mu}^c(\{N^{\overset{\leftarrow}{p}}_{m+\frac{1}{2}}(\underline{z})=k\})=\underline{\mu}^c&(\{\underline{z}:N^{\overset{\leftarrow}{p}}_{m+1+\frac{1}{2}}(\tau^{-1}\underline{z})=k\}) \\
  &=\sum\limits_{\underline{y}:N^{\overset{\leftarrow}{p}}_{m+1+\frac{1}{2}}(\underline{y})=k}\underline{\mu}^c(\tau\underline{y})=\sum\limits_{\underline{y}:N^{\overset{\leftarrow}{p}}_{m+1+\frac{1}{2}}(\underline{y})=k}\underline{\mu}^{c+1}(\underline{y})=\underline{\mu}^{c+1}(\{N^{\overset{\leftarrow}{p}}_{m+1+\frac{1}{2}}(\underline{z})=k\}).
\end{align*}
Similar arguments hold for $\underline{\mu}^c(\{N^{\overset{\rightarrow}{h}}_{m+\frac{1}{2}}(\underline{z})=k\})$ and $\underline{\mu}^c(\{N_{m+\frac{1}{2}}(\underline{z})=k\})$.
\end{proof}
We can also give results where we only change the parameter $c$ and not the quantity we wish to understand the probability of.
\begin{lem}\label{lem: c shift in mu and N p,h}
For any $c\in\mathbb{R}$, $m,k\in\mathbb{Z}$ we have that,
\begin{enumerate}[(a)]
    \item $\underline{\mu}^{c-1}(\{N^{\overset{\leftarrow}{p}}_{m+\frac{1}{2}}(\underline{z})=k\})=\frac{q^{-k}}{1+q^{c-m-1}}\cdot \underline{\mu}^c(\{N^{\overset{\leftarrow}{p}}_{m+\frac{1}{2}}(\underline{z})=k\})$
    \item $\underline{\mu}^{c-1}(\{N^{\overset{\rightarrow}{h}}_{m+\frac{1}{2}}(\underline{z})=k\})=q^{k+m+1-c}(1+q^{c-m-1})\cdot \underline{\mu}^c(\{N^{\overset{\rightarrow}{h}}_{m+\frac{1}{2}}(\underline{z})=k\})$
    \item $\underline{\mu}^{c-1}(\{N_{m+\frac{1}{2}}(\underline{z})=k\})=q^{k+m+1-c}\cdot \underline{\mu}^c(\{N_{m+\frac{1}{2}}(\underline{z})=k\}).$
\end{enumerate}

\end{lem}
\begin{proof}~
\par \begin{enumerate}[(a)]
    \item We have that, \begin{align*}
    \underline{\mu}^{c-1}(\{N^{\overset{\leftarrow}{p}}_{m+\frac{1}{2}}(\underline{z})=k\})&=\sum\limits_{\underline{z}:~N^{\overset{\leftarrow}{p}}_{m+\frac{1}{2}}(\underline{z})=k}\underline{\mu}^{c-1}(\underline{z})=\sum\limits_{\underline{z}:~N^{\overset{\leftarrow}{p}}_{m+\frac{1}{2}}(\underline{z})=k}\prod\limits_{j=-\infty}^m\frac{q^{(c-1-j)z_j}}{(1+q^{c-1-j})}\\
    &=\sum\limits_{\underline{z}:~N^{\overset{\leftarrow}{p}}_{m+\frac{1}{2}}(\underline{z})=k}\frac{q^{-\sum\limits_{j=-\infty}^mz_j}}{(1+q^{c-1-m})}\prod\limits_{j=-\infty}^m\frac{q^{(c-j)z_j}}{1+q^{c-j}}\\
    &=\frac{q^{-k}}{1+q^{c-m-1}}\cdot\underline{\mu}^c(\{N^{\overset{\leftarrow}{p}}_{m+\frac{1}{2}}(\underline{z})=k\}).
\end{align*}
\item We have that, \begin{align*}
    \underline{\mu}^{c-1}(\{N^{\overset{\rightarrow}{h}}_{m+\frac{1}{2}}(\underline{z})=k\})&=\sum\limits_{\underline{z}:~N^{\overset{\rightarrow}{h}}_{m+\frac{1}{2}}(\underline{z})=k}\underline{\mu}^{c-1}(\underline{z})\\
    &=\sum\limits_{\underline{z}:~N^{\overset{\rightarrow}{h}}_{m+\frac{1}{2}}(\underline{z})=k}\prod\limits_{j=m+1}^\infty\frac{q^{(c-1-j)z_j}}{1+q^{c-1-j}}\\
    &=\sum\limits_{\underline{z}:~N^{\overset{\rightarrow}{h}}_{m+\frac{1}{2}}(\underline{z})=k}\prod\limits_{j=m+1}^\infty\frac{q^{(1-z_j)(j-c+1)}}{{q}^{(j-c+1)}(1+q^{c-1-j})}\\
    &=\sum\limits_{\underline{z}:~N^{\overset{\rightarrow}{h}}_{m+\frac{1}{2}}(\underline{z})=k}q^{\sum\limits_{j= m+1}^\infty(1-z_j)}\cdot q^{m+1-c}(1+q^{c-m-1})\prod\limits_{j= m+1}^\infty\frac{q^{(1-z_j)(j-c)}}{q^{(j-c)}(1+q^{c-j})}\\
    &=q^{k+m+1-c}(1+q^{c-m-1})\cdot\underline{\mu}^c(\{N^{\overset{\rightarrow}{h}}_{m+\frac{1}{2}}(\underline{z})=k\}).
\end{align*}
\item We notice that, 
\begin{align*}\underline{\mu}^{c-1}(\{N_{m+\frac{1}{2}}(\underline{z})=k\})&=\sum\limits_{k_1,k_2\in \mathbb{Z}}\underline{\mu}^{c-1}(\{N_{m+\frac{1}{2}}(\underline{z})=k\}\cap\{N^{\overset{\leftarrow}{p}}_{m+\frac{1}{2}}(\underline{z})=k_1\}\cap\{N^{\overset{\rightarrow}{h}}_{m+\frac{1}{2}}(\underline{z})=k_2\})\\
&=\sum\limits_{\substack{k_1,k_2\in\mathbb{Z}:\\k_2-k_1=k}}\underline{\mu}^{c-1}(\{N^{\overset{\leftarrow}{p}}_{m+\frac{1}{2}}(\underline{z})=k_1\}\cap\{N^{\overset{\rightarrow}{h}}_{m+\frac{1}{2}}(\underline{z})=k_2\})\\
&=\sum\limits_{\substack{k_1,k_2\in\mathbb{Z}:\\k_2-k_1=k}}\underline{\mu}^{c-1}(\{N^{\overset{\leftarrow}{p}}_{m+\frac{1}{2}}(\underline{z})=k_1\}) \cdot\underline{\mu}^{c-1}(\{N^{\overset{\rightarrow}{h}}_{m+\frac{1}{2}}(\underline{z})=k_2\}).
\end{align*}
Using (a) and (b) we have that, 
\begin{align*}
    &\sum\limits_{\substack{k_1,k_2: \\ k_2-k_1=k}} \underline{\mu}^{c-1}(\{N^{\overset{\leftarrow}{p}}_{m+\frac{1}{2}}(\underline{z})=k_1\}) \cdot\underline{\mu}^{c-1}(\{N^{\overset{\rightarrow}{h}}_{m+\frac{1}{2}}(\underline{z})=k_2\}) \\
    &= \sum\limits_{\substack{k_1,k_2: \\ k_2-k_1=k}} \frac{q^{-k_1}}{(1+q^{c-m-1})}\cdot\underline{\mu}^c(\{N^{\overset{\leftarrow}{p}}_{m+\frac{1}{2}}(\underline{z})=k_1\}) \cdot q^{k_2+m+1-c}(1+q^{c-m-1})\cdot\underline{\mu}^c(\{N^{\overset{\rightarrow}{h}}_{m+\frac{1}{2}}(\underline{z})=k_2\})\\
    &= q^{k+m+1-c}\sum\limits_{\substack{k_1,k_2: \\ k_2-k_1=k}} \underline{\mu}^{c}(\{N^{\overset{\leftarrow}{p}}_{m+\frac{1}{2}}(\underline{z})=k_1\}) \cdot\underline{\mu}^{c}(\{N^{\overset{\rightarrow}{h}}_{m+\frac{1}{2}}(\underline{z})=k_2\}) \\
    &=q^{k+m+1-c}\cdot \underline{\mu}^c(\{N_{m+\frac{1}{2}}(\underline{z})=k\}).
\end{align*}
\end{enumerate}
\end{proof}
\begin{remark} We note that the proofs given in this section can be easily extended to give similar results for all processes in the blocking family as defined by Bal\'azs and Bowen, \cite{blocking}.
\end{remark}
\par Using the above results we will now find explicit formulae for the distribution of $N^{\overset{\leftarrow}{p}}_{\bullet}$, $N^{\overset{\rightarrow}{h}}_{\bullet}$ any $N_\bullet$ under the natural product blocking measure $\underline{\mu}^c$. 
\par Equation \eqref{eq: mu of N} gave the distribution under $\underline{\mu}^c$ of the conserved quantity about the site $0$, that is, $\underline{\mu}^c(\{N=n\})$. By considering shifts of ASEP states we can find the distribution of the conserved quantity about any site $m\in\mathbb{Z}$.
\begin{lem}\label{lem: mu^c(N_m=n)}
For any $n,m\in\mathbb{Z}$ and $c\in\mathbb{R}$,
$$\underline{\mu}^c(\{N_{m+\frac{1}{2}}=n\})=\frac{q^{\frac{(n+m)(n+m+1)}{2}-(n+m)c}}{\sum\limits_{\ell=-\infty}^\infty q^{\frac{\ell(\ell+1)}{2}-\ell c}}$$
\end{lem}
\begin{proof} As we saw in the proof of Corollary \ref{cor: shift in c shifts the limits in N}, $N_{m+\frac{1}{2}}(\underline{z})=N_{m-1+\frac{1}{2}}(\tau\underline{z})$. Thus we have that, $N_{m+\frac{1}{2}}(\underline{z})=N(\tau^m\underline{z})$. And as we have seen before, a left shift decreases the conserved quantity by 1 and so, $N(\tau^m\underline{z})=N(\underline{z})-m$. Then, by \eqref{eq: mu of N} we have that, 
    $$\underline{\mu}^c(\{N_{m+\frac{1}{2}}=n\})=\underline{\mu}^c(\{N=n+m\})=\frac{q^{\frac{(n+m)(n+m+1)}{2}-(n+m)c}}{\sum\limits_{\ell=-\infty}^\infty q^{\frac{\ell(\ell+1)}{2}-\ell c}}.$$
\end{proof}
We now consider the probability under the blocking measure that there is some $k\in\mathbb{Z}_{\geq 0}$ particles to left of a site $m\in\mathbb{Z}$ (inclusive).
\partL*
\begin{proof}
\par Conditioning on site $m$ we have that (for $k\geq 1$), 
\begin{align*}
    \underline{\mu}^c(\{N^{\overset{\leftarrow}{p}}_{m+\frac{1}{2}}(\underline{z})=k\})&=\underline{\mu}^c(\{N^{\overset{\leftarrow}{p}}_{m+\frac{1}{2}}(\underline{z})=k|z_m=0\})\mu^c_m(0)+\underline{\mu}^c(\{N^{\overset{\leftarrow}{p}}_{m+\frac{1}{2}}(\underline{z})=k|z_m=1\})\mu_m^c(1) \\ 
    &=\underline{\mu}^c(\{N^{\overset{\leftarrow}{p}}_{m-1+\frac{1}{2}}(\underline{z})=k\})\frac{1}{1+q^{c-m}}+\underline{\mu}^c(\{N^{\overset{\leftarrow}{p}}_{m-1+\frac{1}{2}}(\underline{z})=k-1\})\frac{q^{c-m}}{1+q^{c-m}}.
\end{align*}
We want to look at $N^{\overset{\leftarrow}{p}}_{m+\frac{1}{2}}(\underline{z})$ not $N^{\overset{\leftarrow}{p}}_{m-1+\frac{1}{2}}(\underline{z})$ so we use the result of Corollary \ref{cor: shift in c shifts the limits in N}:
$$\underline{\mu}^c(\{N^{\overset{\leftarrow}{p}}_{m-1+\frac{1}{2}}(\underline{z})=k\})=\underline{\mu}^{c+1}(\{N^{\overset{\leftarrow}{p}}_{m+\frac{1}{2}}(\underline{z})=k\}).$$
\par \noindent Hence, 
$$\underline{\mu}^c(\{N^{\overset{\leftarrow}{p}}_{m+\frac{1}{2}}(\underline{z})=k\})=\underline{\mu}^{c+1}(\{N^{\overset{\leftarrow}{p}}_{m+\frac{1}{2}}(\underline{z})=k\})\frac{1}{1+q^{c-m}}+\underline{\mu}^{c+1}(\{N^{\overset{\leftarrow}{p}}_{m+\frac{1}{2}}(\underline{z})=k-1\})\frac{q^{c-m}}{1+q^{c-m}}.$$
By Lemma \ref{lem: c shift in mu and N p,h}, 
$$\underline{\mu}^{c+1}(\{N^{\overset{\leftarrow}{p}}_{m+\frac{1}{2}}(\underline{z})=k\})=q^k(1+q^{c-m})\underline{\mu}^{c}(\{N^{\overset{\leftarrow}{p}}_{m+\frac{1}{2}}(\underline{z})=k\}).$$
And so we have, 
$$\underline{\mu}^c(\{N^{\overset{\leftarrow}{p}}_{m+\frac{1}{2}}(\underline{z})=k\})=\underline{\mu}^{c}(\{N^{\overset{\leftarrow}{p}}_{m+\frac{1}{2}}(\underline{z})=k\})q^k+\underline{\mu}^{c}(\{N^{\overset{\leftarrow}{p}}_{m+\frac{1}{2}}(\underline{z})=k-1\})q^{c-m+k-1}.$$
Giving the recurrence relation, 
$$\underline{\mu}^c(\{N^{\overset{\leftarrow}{p}}_{m+\frac{1}{2}}(\underline{z})=k\})=\frac{q^{k-1+c-m}}{1-q^k}\underline{\mu}^{c}(\{N^{\overset{\leftarrow}{p}}_{m+\frac{1}{2}}(\underline{z})=k-1\})$$
which we solve by, 
\begin{align*}
  \underline{\mu}^c(\{N^{\overset{\leftarrow}{p}}_{m+\frac{1}{2}}(\underline{z})=k\})&=\left(\prod\limits_{i=1}^k \frac{q^{i-1+c-m}}{1-q^i} \right)\underline{\mu}^c(\{N^{\overset{\leftarrow}{p}}_{m+\frac{1}{2}}(\underline{z})=0\})\\
  &=\prod\limits_{i=1}^k \frac{q^{i-1+c-m}}{1-q^i}\cdot \prod\limits_{j=-\infty}^ m\frac{1}{1+q^{c-j}}\\
  &=\frac{q^{k(c-m-1)+\frac{k(k+1)}{2}}}{(q;q)_k(-q^{c-m};q)_\infty}\\
  &=\frac{q^{k(c-m)+\frac{k(k-1)}{2}}}{(q;q)_k(-q^{c-m};q)_\infty}.
\end{align*}
\end{proof}
\begin{remark}
    An alternative proof of this result by considering two coupled ASEPs can be given. We will see this in Section \ref{subsect: asep label process}.
\end{remark}
It is also natural to ask what is the distribution of this quantity given we consider a fixed conserved quantity $n\in\mathbb{Z}$, i.e.\ under $\underline{\nu}^n$ the unique distribution on the irreducible component $\Omega^n$. First we give the distribution of a single site under the conditional measure $\underline{\nu}^n$.

\begin{lem}\label{lem: site under nu n}
    For any $n,m\in\mathbb{Z}$ we have that, 
    \begin{align*}
        \underline{\nu}^n(\{z_m=1\})&=(q;q)_\infty\sum\limits_{k=(m-n-1)^+}^\infty\frac{q^{(k+1)(n-m+k+1)}}{(q;q)_k(q;q)_{n-m+k+1}}\\
        \underline{\nu}^n(\{z_m=0\})&=(q;q)_\infty\sum\limits_{k=(m-n)^+}^\infty\frac{q^{k(n-m+k+1)}}{(q;q)_k(q;q)_{n-m+k}}.
    \end{align*}
\end{lem}

\begin{proof}
By the product structure of $\underline{\mu}^c$, Corollary \ref{cor: c change between holes and particles in ASEP} and Theorem \ref{thm: half infinite asep particles}  we have that,
    \begin{align*}
    \underline{\nu}^n(&\{z_m=1\})=\underline{\mu}^c(z_m=1|N(\underline{z})=n)\\
    \\
    &=\frac{\underline{\mu}^c(\{z_m=1\}\bigcap\{N(\underline{z})=n\})}{\underline{\mu}^c(\{N(\underline{z})=n\})}\\
    \\
    &=\frac{\underline{\mu}^c(\{z_m=1\}\bigcap\{N_{m+\frac{1}{2}}(\underline{z})=n-m\})}{\underline{\mu}^c(\{N(\underline{z})=n\})}\\
    \\
    &=\sum\limits_{k=(m-n-1)^+}^\infty\frac{\underline{\mu}^c(\{z_m=1\}\bigcap\{N_{m-1+\frac{1}{2}}^{\overset{\leftarrow}{p}}(\underline{z})=k\}\bigcap\{N_{m+\frac{1}{2}}^{\overset{\rightarrow}{h}}(\underline{z})=n-m+k+1\})}{\underline{\mu}^c(\{N(\underline{z})=n\})}\\
    \\
    &=\sum\limits_{k=(m-n-1)^+}^\infty\frac{\mu_m^c(1)\underline{\mu}^c(\{N_{m-1+\frac{1}{2}}^{\overset{\leftarrow}{p}}(\underline{z})=k\})\underline{\mu}^c(\{N_{m+\frac{1}{2}}^{\overset{\rightarrow}{h}}(\underline{z})=n-m+k+1\})}{\underline{\mu}^c(\{N(\underline{z})=n\})}\\
    \\
    &=\sum\limits_{k=(m-n-1)^+}^\infty\frac{\mu_m^c(1)\underline{\mu}^c(\{N_{m-1+\frac{1}{2}}^{\overset{\leftarrow}{p}}(\underline{z})=k\})\underline{\mu}^{2m+1-c}(\{N_{m+\frac{1}{2}}^{\overset{\leftarrow}{p}}(\underline{z})=n-m+k+1\})}{\underline{\mu}^c(\{N(\underline{z})=n\})}\\
    \\
    &=\frac{\sum\limits_{\ell\in\mathbb{Z}}q^{\frac{\ell(\ell+1)}{2}-\ell c}}{q^{\frac{n(n+1)}{2}-nc}}\cdot \frac{q^{c-m}}{(1+q^{c-m})}\cdot \sum\limits_{k=(m-n-1)^+}^\infty\frac{q^{k(c+1-m)+\frac{k(k-1)}{2}}}{(q;q)_k(-q^{c+1-m};q)_\infty}\cdot \frac{q^{(n-m+k+1)(m+1-c)+\frac{(n-m+k+1)(n-m+k)}{2}}}{(q;q)_{n-m+k+1}(-q^{m+1-c};q)_\infty}\\
    \\
    &=\sum\limits_{\ell\in\mathbb{Z}}q^{\frac{\ell(\ell+1)}{2}-\ell c} \cdot \frac{q^{mc -\frac{m(m+1)}{2}}}{(-q^{c-m};q)_\infty(-q^{m+1-c};q)_\infty}\sum\limits_{k=(m-n-1)^+}^\infty\frac{q^{(k+1)(n-m+k+1)}}{(q;q)_{n-m+k+1}(q;q)_k}.
\end{align*}
Lemma \ref{lem: pochammer relation for nu^n calcs} gives us that,
$$\frac{q^{\frac{mc-m(m+1)}{2}}}{(-q^{c-m};q)_\infty(-q^{m+1-c};q)_\infty}=\frac{1}{(-q^{1-c};q)_\infty(-q^{c};q)_\infty}.$$
The Jacobi Triple Product identity states that, for $q,z\in\mathbb{R}$  such that $|q|<1$ and $z\neq 0$,
$$\sum\limits_{l=-\infty}^\infty q^{\frac{\ell(\ell+1)}{2}}z^{\ell}=(q;q)_\infty(-qz;q)_\infty(-z^{-1};q)_\infty.$$
Then by setting $z\defeq q^{-c}$ we have that, 
\begin{equation}\label{eq: JTP}
    \sum\limits_{l=-\infty}^\infty q^{\frac{\ell(\ell+1)}{2}-\ell c}=(q;q)_\infty(-q^{1-c};q)_\infty(-q^c;q)_\infty.
\end{equation}
Putting this together we find that, 
$$\underline{\nu}^n(\{z_m=1\})=(q;q)_\infty\sum\limits_{k=(m-n-1)^+}^\infty\frac{q^{(k+1)(n-m+k+1)}}{(q;q)_k(q;q)_{n-m+k+1}}.$$
By similar reasoning we have that, 
\begin{align*}
    \underline{\nu}^n(&\{z_m=0\})=\sum\limits_{k=(m-n)^+}^\infty\frac{\mu^c_m(0)\underline{\mu}^c(\{N^{\overset{\leftarrow}{p}}_{m-1+\frac{1}{2}}(\underline{z})=k\})\underline{\mu}^c(\{N^{\overset{\rightarrow}{h}}_{m+\frac{1}{2}}(\underline{z})=n-m+k\})}{\underline{\mu}^c(\{N(\underline{z})=n\})}\\
    &=\sum\limits_{k=(m-n)^+}^\infty\frac{\mu^c_m(0)\underline{\mu}^c(\{N^{\overset{\leftarrow}{p}}_{m-1+\frac{1}{2}}(\underline{z})=k\})\underline{\mu}^{2m+1-c}(\{N^{\overset{\leftarrow}{p}}_{m+\frac{1}{2}}(\underline{z})=n-m+k\})}{\underline{\mu}^c(\{N(\underline{z})=n\})}\\
    \\
    &=\frac{\sum\limits_{\ell\in\mathbb{Z}}q^{\frac{\ell(\ell+1)}{2}-\ell c}}{q^{\frac{n(n+1)}{2}-nc}}\cdot \frac{1}{(1+q^{c-m})}\cdot \sum\limits_{k=(m-n)^+}^\infty\frac{q^{k(c+1-m)+\frac{k(k-1)}{2}}}{(q;q)_k(-q^{c+1-m};q)_\infty}\cdot \frac{q^{(n-m+k)(m+1-c)+\frac{(n-m+k)(n-m+k-1)}{2}}}{(q;q)_{n-m+k}(-q^{m+1-c};q)_\infty}\\
    \\
    &=\sum\limits_{\ell\in\mathbb{Z}}q^{\frac{\ell(ell+1)}{2}-\ell c}\cdot \frac{q^{mc-\frac{m(m+1)}{2}}}{(-q^{c-m};q)_\infty(-q^{m+1-c};q)_\infty}\cdot \sum\limits_{k=(m-n)^+}^\infty\frac{q^{k(n-m+k+1)}}{(q;q)_k(q;q)_{n-m+k}}.
\end{align*}
By Lemma \ref{lem: pochammer relation for nu^n calcs} and the Jacobi triple product identity, equation \eqref{eq: JTP}, we find that, 
$$\underline{\nu}^n(\{z_m=0\})=(q;q)_\infty\sum\limits_{k=(m-n)^+}^\infty\frac{q^{k(n-m+k+1)}}{(q;q)_k(q;q)_{n-m+k}}.$$
\end{proof}
 Now we find the distribution of ASEP particles in a half infinite range under $\underline{\nu}^n$.
\nupartL*
\newpage\begin{proof}
    By definition (for any $c\in\mathbb{R})$, 
    $$\underline{\nu}^n(\{N^{\overset{\leftarrow}{p}}_{m+\frac{1}{2}}(\underline{z})=k\})=\underline{\mu}^c(\{N^{\overset{\leftarrow}{p}}_{m+\frac{1}{2}}(\underline{z})=k\}|\{N(\underline{z})=n\})=\frac{\underline{\mu}^c(\{N^{\overset{\leftarrow}{p}}_{m+\frac{1}{2}}(\underline{z})=k\}\cap\{N(\underline{z})=n\})}{\underline{\mu}^c(\{N(\underline{z})=n\})}.$$
    We notice that, 
    \begin{align*}
    \{N^{\overset{\leftarrow}{p}}_{m+\frac{1}{2}}(\underline{z})=k\}\cap\{N(\underline{z})=n\}&=\{N^{\overset{\leftarrow}{p}}_{m+\frac{1}{2}}(\underline{z})=k\}\cap\{N_{m+\frac{1}{2}}(\underline{z})=n-m\}\\
    &=\{N^{\overset{\leftarrow}{p}}_{m+\frac{1}{2}}(\underline{z})=k\}\cap\{N^{\overset{\rightarrow}{h}}_{m+\frac{1}{2}}(\underline{z})=n-m+k\}.
    \end{align*}
    Since both the number of particles to the left and holes to the right of $(m+\frac{1}{2})$ must be positive we see that $\underline{\nu}^n(\{N^{\overset{\leftarrow}{p}}_{m+\frac{1}{2}}(\underline{z})=k\})$ is only well defined for $k\geq\max\{m-n,0\}$.
    \par Since $\underline{\mu}^c$ is a product measure we have that, 
    $$\underline{\nu}^n(\{N^{\overset{\leftarrow}{p}}_{m+\frac{1}{2}}(\underline{z})=k\})=\frac{\underline{\mu}^c(\{N^{\overset{\leftarrow}{p}}_{m+\frac{1}{2}}(\underline{z})=k\})\cdot\underline{\mu}^c(\{N^{\overset{\rightarrow}{h}}_{m+\frac{1}{2}}(\underline{z})=n-m+k\})}{\underline{\mu}^c(\{N(\underline{z})=n\})}.$$
    By Corollary \ref{cor: c change between holes and particles in ASEP} that is,
     $$\underline{\nu}^n(\{N^{\overset{\leftarrow}{p}}_{m+\frac{1}{2}}(\underline{z})=k\})=\frac{\underline{\mu}^c(\{N^{\overset{\leftarrow}{p}}_{m+\frac{1}{2}}(\underline{z})=k\})\cdot\underline{\mu}^{2m+1-c}(\{N^{\overset{\leftarrow}{p}}_{m+\frac{1}{2}}(\underline{z})=n-m+k\})}{\underline{\mu}^c(\{N(\underline{z})=n\})}.$$
     Theorem \ref{thm: half infinite asep particles} and equation \eqref{eq: mu of N} give, 
     $$\underline{\nu}^n(\{N^{\overset{\leftarrow}{p}}_{m+\frac{1}{2}}(\underline{z})=k\})=\frac{\sum\limits_{\ell\in\mathbb{Z}}q^{\frac{\ell(\ell+1)}{2}-\ell c}\cdot q^{k(n-m+k)-\frac{m(m+1)}{2}+mc}}{(q;q)_k(-q^{c-m};q)_\infty(q;q)_{n-m+k}(-q^{m+1-c};q)_\infty}.$$
     Now we can use Lemma \ref{lem: pochammer relation for nu^n calcs} and find that, 
     $$\underline{\nu}^n(\{N^{\overset{\leftarrow}{p}}_{m+\frac{1}{2}}(\underline{z})=k\})=\frac{\sum\limits_{\ell \in \mathbb{Z}}q^{\frac{\ell(\ell+1)}{2}-\ell c}\cdot q^{k(n-m+k)}}{(q;q)_k(q;q)_{n-m+k}(-q^{1-c};q)_\infty (-q^c;q)_\infty}.$$
     Finally the Jacobi triple product identity, equation \eqref{eq: JTP} gives us that, 
     $$\underline{\nu}^n(\{N^{\overset{\leftarrow}{p}}_{m+\frac{1}{2}}(\underline{z})=k\})=\frac{ q^{k(n-m+k)}(q;q)_\infty}{(q;q)_k(q;q)_{n-m+k}}.$$
\end{proof}
\vspace{3mm} We also consider the probability under the blocking measure that there is some $k$ particles in between two sites. 
\finparts*
\begin{proof}
Firstly, let us consider the cases where $k=0$ or $k=\hat{m}_2$. These cases are directly computable since there is only 1 state in each case for which $N^{\overset{\leftarrow}{p}}_{m_2-1+\frac{1}{2}}(\underline{z})-N^{\overset{\leftarrow}{p}}_{m_1+\frac{1}{2}}(\underline{z})=k$. If $k=0$ then the state is such that all the sites in $[m_1+1,m_2-1]$ are empty and so, 
\begin{equation}\label{eq: 0 particles}
\underline{\mu}^c(\{N^{\overset{\leftarrow}{p}}_{m_2-1+\frac{1}{2}}(\underline{z})-N^{\overset{\leftarrow}{p}}_{m_1+\frac{1}{2}}(\underline{z})=0\})=\frac{1}{\prod\limits_{i=m_1+1}^{m_2-1}(1+q^{c-i})}=\frac{1}{\prod\limits_{i=1}^{\hat{m}_2}(1+q^{c-m_2+i})}=\frac{1}{(-q^{c-m_2+1};q)_{\hat{m}_2}}.
\end{equation}
If $k=\hat{m}_2$ then the state is such that all the sites in $[m_1+1,m_2-1]$ are full and so, 
\begin{equation}\label{eq: full particles}
\underline{\mu}^c(\{N^{\overset{\leftarrow}{p}}_{m_2-1+\frac{1}{2}}(\underline{z})-N^{\overset{\leftarrow}{p}}_{m_1+\frac{1}{2}}(\underline{z})=\hat{m}_2\})=\prod\limits_{i=m_1+1}^{m_2-1}\frac{q^{c-i}}{(1+q^{c-i})}=\frac{q^{\hat{m}_2(c+1-m_2)+\frac{\hat{m}_2(\hat{m}_2-1)}{2}}}{(-q^{c-m_2+1};q)_{\hat{m}_2}}.\end{equation}
Now, let us consider $k\in\{1,...,\hat{m}_2-1\}$, we see that
\begin{align}\underline{\mu}^c(\{N^{\overset{\leftarrow}{p}}_{m_2-1+\frac{1}{2}}(\underline{z})-&N^{\overset{\leftarrow}{p}}_{m_1+\frac{1}{2}}(\underline{z})=k\})\notag\\
=& \underline{\mu}^c(\{N^{\overset{\leftarrow}{p}}_{m_2-1+\frac{1}{2}}(\underline{z})-N^{\overset{\leftarrow}{p}}_{m_1+\frac{1}{2}}(\underline{z})=k|z_{m_2-1}=0\})\cdot \mu^c_{m_2-1}(0)\notag\\
&+\underline{\mu}^c(\{N^{\overset{\leftarrow}{p}}_{m_2-1+\frac{1}{2}}(\underline{z})-N^{\overset{\leftarrow}{p}}_{m_1+\frac{1}{2}}(\underline{z})=k|z_{m_2-1}=1\})\cdot\mu^c_{m_2-1}(1)\notag\\
=&\underline{\mu}^c(\{N^{\overset{\leftarrow}{p}}_{m_2-2+\frac{1}{2}}(\underline{z})-N^{\overset{\leftarrow}{p}}_{m_1+\frac{1}{2}}(\underline{z})=k\})\cdot\frac{1}{1+q^{c-m_2+1}}\notag\\
&+\underline{\mu}^c(\{N^{\overset{\leftarrow}{p}}_{m_2-2+\frac{1}{2}}(\underline{z})-N^{\overset{\leftarrow}{p}}_{m_1+\frac{1}{2}}(\underline{z})=k-1\})\cdot \frac{q^{c-m_2+1}}{1+q^{c-m_2+1}}.\label{eq: finite asep particles}
\end{align}
We now show that $\underline{\mu}^c(\{N^{\overset{\leftarrow}{p}}_{m_2-1+\frac{1}{2}}(\underline{z})-N^{\overset{\leftarrow}{p}}_{m_1+\frac{1}{2}}(\underline{z})=k\})=\frac{q^{k(c+1-m_2)+\frac{k(k-1)}{2}}}{(-q^{c-m_2+1};q)_{\hat{m}_2}}\begin{bmatrix}
    \hat{m}_2 \\
    k
    \end{bmatrix}_{q}$ satisfies (\ref{eq: finite asep particles}). This formula is inspired by a combinatorial argument using integer partitions. As additional material this argument is given in Appendix \ref{appendix: combinatorics} but those parts are not needed for any of the proofs in this article, since these are self-contained and purely probabilistic.
    \par Start with the RHS of (\ref{eq: finite asep particles}), 
\begin{align*}
&\frac{1}{1+q^{c-m_2+1}}\left(\underline{\mu}^c(\{N^{\overset{\leftarrow}{p}}_{m_2-2+\frac{1}{2}}(\underline{z})-N^{\overset{\leftarrow}{p}}_{m_1+\frac{1}{2}}(\underline{z})=k\})+q^{c-m_2+1}\cdot\underline{\mu}^c(\{N^{\overset{\leftarrow}{p}}_{m_2-2+\frac{1}{2}}(\underline{z})-N^{\overset{\leftarrow}{p}}_{m_1+\frac{1}{2}}(\underline{z})=k-1\})\right)\\
&\hspace{8mm}=\frac{1}{1+q^{c-m_2+1}}\left( \frac{q^{k(c+2-m_2)+\frac{k(k-1)}{2}}}{(-q^{c-m_2+2};q)_{\hat{m}_2-1}}\begin{bmatrix}
    \hat{m}_2-1 \\
    k
    \end{bmatrix}_{q} + q^{c-m_2+1}\frac{q^{(k-1)(c+2-m_2)+\frac{(k-1)(k-2)}{2}}}{(-q^{c-m_2+2};q)_{\hat{m}_2-1}}\begin{bmatrix}
    \hat{m}_2-1 \\
    k-1
    \end{bmatrix}_{q} \right)\\
&\hspace{8mm}=\frac{q^{k(c+1-m_2)+\frac{k(k-1)}{2}}}{(-q^{c-m_2+1};q)_{\hat{m}_2}}\left(q^k\begin{bmatrix}
    \hat{m}_2-1 \\
    k
    \end{bmatrix}_{q}+\begin{bmatrix}
    \hat{m}_2-1 \\
    k-1
    \end{bmatrix}_{q}\right).
\end{align*}
Using the $q$-binomial analogue of Pascal's identity (Lemma \ref{lem: q-pascal}) we have that,
$$q^k\begin{bmatrix}
    \hat{m}_2-1 \\
    k
    \end{bmatrix}_{q}+\begin{bmatrix}
    \hat{m}_2-1 \\
    k-1
    \end{bmatrix}_{q}=\begin{bmatrix}
    \hat{m}_2 \\
    k
    \end{bmatrix}_{q}.$$
Thus (\ref{eq: finite asep particles}) holds with $\underline{\mu}^c(\{N^{\overset{\leftarrow}{p}}_{m_2-1+\frac{1}{2}}(\underline{z})-N^{\overset{\leftarrow}{p}}_{m_1+\frac{1}{2}}(\underline{z})=k\})=\frac{q^{k(c+1-m_2)+\frac{k(k-1)}{2}}}{(-q^{c-m_2+1};q)_{\hat{m}_2}}\begin{bmatrix}
    \hat{m}_2 \\
    k
    \end{bmatrix}_{q}$.
    \par\vspace{2mm}\noindent To complete the proof we first notice that the boundary conditions given by \eqref{eq: 0 particles} and \eqref{eq: full particles} conform with the solution above and show that with the recursion \eqref{eq: finite asep particles} the distribution is completely fixed. To find the value of $\underline{\mu}^c(\{N^{\overset{\leftarrow}{p}}_{m_2-1+\frac{1}{2}}(\underline{z})-N^{\overset{\leftarrow}{p}}_{m_1+\frac{1}{2}}(\underline{z})=k\})$ at a point $(\hat{m}_2,k)$ the recursion \eqref{eq: finite asep particles} requires values at the points $(\hat{m}_2-1,k)$ and $(\hat{m}_2-1,k-1)$. Using the recursion iteratively we arrive to the boundary points $(x,x)$ and $(x,0)$ for any $x\in\{1,\dots,k\}$. The former case is handled by \eqref{eq: full particles} and the latter by \eqref{eq: 0 particles}.  
\end{proof}
\begin{remark}
    Notice that Theorem \ref{thm: half infinite asep particles} can be seen by taking $m\defeq m_2-1$  and the limit as $m_1 \to -\infty$ in Lemma \ref{lem: finite asep particles}. This gives it yet another alternative proof.
\end{remark}
Again it is natural to ask what is the distribution of this quantity under $\underline{\nu}^n$?
\nufinparts*

\begin{proof}
    By definition (for any $c\in\mathbb{R}$), 
    \begin{align*}\underline{\nu}^n(\{N^{\overset{\leftarrow}{p}}_{m_2-1+\frac{1}{2}}(\underline{z})-N^{\overset{\leftarrow}{p}}_{m_1+\frac{1}{2}}(\underline{z})=k\})&=\underline{\mu}^c(\{N^{\overset{\leftarrow}{p}}_{m_2-1+\frac{1}{2}}(\underline{z})-N^{\overset{\leftarrow}{p}}_{m_1+\frac{1}{2}}(\underline{z})=k\}|\{N(\underline{z})=n\})\\
    &=\frac{\underline{\mu}^c(\{N^{\overset{\leftarrow}{p}}_{m_2-1+\frac{1}{2}}(\underline{z})-N^{\overset{\leftarrow}{p}}_{m_1+\frac{1}{2}}(\underline{z})=k\}\cap\{N(\underline{z})=n\})}{\underline{\mu}^c(\{N(\underline{z})=n\})}.
    \end{align*}
    We notice that any state $\underline{z}$ that satisfies the event, 
    $$\{N^{\overset{\leftarrow}{p}}_{m_2-1+\frac{1}{2}}(\underline{z})-N^{\overset{\leftarrow}{p}}_{m_1+\frac{1}{2}}(\underline{z})=k\}\cap\{N(\underline{z})=n\},$$
    is such that for some $j$,
    $$\{N^{\overset{\rightarrow}{h}}_{m_2-1+\frac{1}{2}}=n+1+k+j-m_2\}\cap\{N^{\overset{\leftarrow}{p}}_{m_2-1+\frac{1}{2}}(\underline{z})-N^{\overset{\leftarrow}{p}}_{m_1+\frac{1}{2}}(\underline{z})=k\}\cap\{N^{\overset{\leftarrow}{p}}_{m_1+\frac{1}{2}}(\underline{z})=j\}.$$
    Since the number of particles to the left of $m_1+\frac{1}{2}$ and the number of holes to the right of $m_2-1+\frac{1}{2}$ must both be positive we see that if $N^{\overset{\leftarrow}{p}}_{m_2-1+\frac{1}{2}}(\underline{z})-N^{\overset{\leftarrow}{p}}_{m_1+\frac{1}{2}}(\underline{z})=k$ then $j\geq (m_2-n-1-k)^+$.
    \par Since $\underline{\mu}^c$ is a product measure we have that, 
    \begin{align*}
        \underline{\nu}^n(\{&N^{\overset{\leftarrow}{p}}_{m_2-1+\frac{1}{2}}(\underline{z})-N^{\overset{\leftarrow}{p}}_{m_1+\frac{1}{2}}(\underline{z})=k\})\\
        &=\frac{1}{\underline{\mu}^c(\{N(\underline{z})=n\})}\sum\limits_{j=(m_2-n-1-k)^+}^\infty \underline{\mu}^c(\{N^{\overset{\rightarrow}{h}}_{m_2-1+\frac{1}{2}}=n+1+k+j-m_2\})\\
        &\hspace{55mm}\cdot \underline{\mu}^c(\{N^{\overset{\leftarrow}{p}}_{m_2-1+\frac{1}{2}}(\underline{z})-N^{\overset{\leftarrow}{p}}_{m_1+\frac{1}{2}}(\underline{z})=k\})\cdot \underline{\mu}^c(\{N^{\overset{\leftarrow}{p}}_{m_1+\frac{1}{2}}(\underline{z})=j\}).
    \end{align*}
     By Corollary \ref{cor: c change between holes and particles in ASEP} that is,
     \begin{align*}
        \underline{\nu}^n(\{&N^{\overset{\leftarrow}{p}}_{m_2-1+\frac{1}{2}}(\underline{z})-N^{\overset{\leftarrow}{p}}_{m_1+\frac{1}{2}}(\underline{z})=k\})\\
        &=\frac{1}{\underline{\mu}^c(\{N(\underline{z})=n\})}\sum\limits_{j=(m_2-n-1-k)^+}^\infty \underline{\mu}^{2m_2-1-c}(\{N^{\overset{\leftarrow}{p}}_{m_2-1+\frac{1}{2}}=n+1+k+j-m_2\})\\
        &\hspace{55mm}\cdot \underline{\mu}^c(\{N^{\overset{\leftarrow}{p}}_{m_2-1+\frac{1}{2}}(\underline{z})-N^{\overset{\leftarrow}{p}}_{m_1+\frac{1}{2}}(\underline{z})=k\})\cdot \underline{\mu}^c(\{N^{\overset{\leftarrow}{p}}_{m_1+\frac{1}{2}}(\underline{z})=j\}).
    \end{align*}
     Theorem \ref{thm: half infinite asep particles}, Lemma \ref{lem: finite asep particles} and equation \eqref{eq: mu of N} give, 
     \begin{align*}
        \underline{\nu}^n(\{&N^{\overset{\leftarrow}{p}}_{m_2-1+\frac{1}{2}}(\underline{z})-N^{\overset{\leftarrow}{p}}_{m_1+\frac{1}{2}}(\underline{z})=k\})\\&=\frac{\sum\limits_{\ell\in\mathbb{Z}}q^{\frac{\ell(\ell+1)}{2}-\ell c}}{q^{\frac{n(n+1)}{2}-nc}}\cdot \sum\limits_{j=(m_2-n-1-k)^+}^\infty \frac{q^{(n+1+k+j-m_2)(m_2-c)+\frac{(n+1+k+j-m_2)(n+k+j-m_2)}{2}}}{(q;q)_{n+1+k+j-m_2}(-q^{m_2-c};q)_\infty}\\
        \\
        &\hspace{55mm}\cdot \frac{q^{j(c-m_1)+\frac{j(j-1)}{2}}}{(q;q)_j(-q^{c-m_1};q)_\infty}\cdot \frac{q^{k(c+1-m_2)+\frac{k(k-1)}{2}}}{(-q^{c-m_2+1};q)_{\hat{m}_2}}\cdot \begin{bmatrix}
        \hat{m}_2 \\
        k
        \end{bmatrix}_{q}\\
        \\
        &=\frac{q^{k(n+1+k-m_2)+(m_2-1)c-\frac{m_2(m_2-1)}{2}}\sum\limits_{\ell\in\mathbb{Z}}q^{\frac{\ell(\ell+1)}{2}-\ell c}}{(-q^{c-m_2+1};q)_{\hat{m}_2}(-q^{c-m_1};q)_\infty(-q^{m_2-c};q)_\infty}\cdot \begin{bmatrix}
        \hat{m}_2 \\
        k
        \end{bmatrix}_{q}\cdot \sum\limits_{j=(m_2-1-n-k)^+}^\infty\frac{q^{j(n-m_1+k+j)}}{(q;q)_j(q;q)_{n+1-m_2+k+j}}. 
     \end{align*}
     We notice that,
     $$(-q^{c-m_2+1};q)_{\hat{m}_2}(-q^{c-m_1};q)_\infty=(-q^{c-m_2+1};q)_\infty.$$
     So by Lemma \ref{lem: pochammer relation for nu^n calcs} with $m\defeq m_2-1$ we find that, 
     \begin{align*}
        \underline{\nu}^n(\{N^{\overset{\leftarrow}{p}}_{m_2-1+\frac{1}{2}}(\underline{z})-&N^{\overset{\leftarrow}{p}}_{m_1+\frac{1}{2}}(\underline{z})=k\})\\
        &=\frac{q^{k(n+1+k-m_2)}\sum\limits_{\ell\in\mathbb{Z}}q^{\frac{\ell(\ell+1)}{2}-\ell c}}{(-q^{1-c};q)_\infty(-q^c;q)_\infty}\cdot\begin{bmatrix}
    \hat{m}_2 \\
    k
    \end{bmatrix}_{q}\cdot \sum\limits_{j=(m_2-1-n-k)^+}^\infty\frac{q^{j(n-m_1+k+j)}}{(q;q)_j(q;q)_{n+1-m_2+k+j}}\\
    \\
    &=\frac{q^{k(n+1+k-m_2)}\sum\limits_{\ell\in\mathbb{Z}}q^{\frac{\ell(\ell+1)}{2}-\ell c}}{(-q^{1-c};q)_\infty(-q^c;q)_\infty}\cdot\begin{bmatrix}
    \hat{m}_2 \\
    k
    \end{bmatrix}_{q}\cdot \sum\limits_{j=(m_2-1-n-k)^+}^\infty\frac{q^{j(n+1-m_2+k+j)+\hat{m}_2j}}{(q;q)_j(q;q)_{n+1-m_2+k+j}}.
    \end{align*}
    \par Finally the Jacobi triple product identity, equation \eqref{eq: JTP}, gives, 
    $$\frac{\sum\limits_{\ell\in\mathbb{Z}}q^{\ell(\ell+1)-\ell c}}{(-q^{1-c};q)_\infty(-q^c;q)_\infty}=(q;q)_\infty.$$
    Thus, 
    \begin{align*}
        \underline{\nu}^n(\{N^{\overset{\leftarrow}{p}}_{m_2-1+\frac{1}{2}}(\underline{z})&-N^{\overset{\leftarrow}{p}}_{m_1+\frac{1}{2}}(\underline{z})=k\})\\
        &=q^{k(n+1+k-m_2)}(q;q)_\infty\cdot\begin{bmatrix}
        \hat{m}_2 \\
        k
        \end{bmatrix}_{q}\cdot \sum\limits_{j=(m_2-1-n-k)^+}^\infty\frac{q^{j(n+1-m_2+k+j)+\hat{m}_2j}}{(q;q)_j(q;q)_{n+1-m_2+k+j}}.
    \end{align*}
\end{proof}
\begin{remark}
    It seems that this can not be simplified further. The sum, 
    $$\sum\limits_{j=(m_2-1-n-k)^+}^\infty\frac{q^{j(n+1-m_2+k+j)+\hat{m}_2j}}{(q;q)_j(q;q)_{n+1-m_2+k+j}}$$
    looks very similar to the sum in the Durfee rectangles identity, in fact if the $q^{j\hat{m}_2}$ factor was not there it could be simplified using the Durfee rectangles identity (see Theorem \ref{thrm: identity from symmetry}).
\end{remark}
\subsection{Combinatorial Identities coming from ASEP}\label{subsect: identities}~
\par Certain probabilistic questions for blocking family interacting particle systems, such as stationarity and equivalences between systems, lead to proofs of identities of combinatorial significance. For example, in 2018 Bal\'azs and Bowen \cite{blocking} gave a probabilistic proof of the Jacobi triple product identity using the Exclusion - Zero-range correspondence. Also, in 2022 Bal\'azs, Fretwell and Jay \cite{MDJ} found new 3-variable combinatorial identities as well as known identities by considering the family of $0$-$1$-$2$ blocking particle systems and also the $k$-Exclusion process. Here we see that other well-known combinatorial identities with interpretations in terms of integer partitions (explained in Section \ref{sect: combinatorics}) can be explained probabilistically by considering certain quantities for ASEP.  
\par By considering the symmetry between particles and holes for the asymmetric simple exclusion process we can give a probabilistic interpretation to a well-known identity from the theory of integer partitions, the Durfee rectangles identity. 
\durfee*
\begin{proof}
It is clear by definition of $N(\cdot)$ that, 
$$\underline{\mu}^c(\{N(\underline{z})=n\})=\sum\limits_{k=\max\{-n,0\}}^\infty\underline{\mu}^c(\{N^{\overset{\rightarrow}{h}}(\underline{z})=n+k\})\cdot \underline{\mu}^c(\{N^{\overset{\leftarrow}{p}}(\underline{z})=k\}).$$
By \eqref{eq: mu of N} as well as Theorem \ref{thm: half infinite asep particles} and Corollary \ref{cor: c change between holes and particles in ASEP} (for $m=0$) we have that, 
\begin{align*}
\frac{q^{\frac{n(n+1)}{2}-nc}}{\sum\limits_{\ell=-\infty}^\infty q^{\frac{\ell(\ell+1)}{2}-\ell c}}&=\sum\limits_{k=\max\{-n,0\}}^\infty \frac{q^{-(n+k)c+\frac{(n+k)(n+k+1)}{2}}}{(-q^{1-c};q)_\infty(q;q)_{n+k}}\cdot \frac{q^{k(c-1)+\frac{k(k+1)}{2}}}{(-q^c;q)_\infty(q;q)_k}\\
&=\sum\limits_{k=\max\{-n,0\}}^\infty \frac{q^{\frac{n(n+1)}{2}+k(n+k)-nc}}{(-q^{1-c};q)_\infty(q;q)_{n+k}(-q^c;q)_\infty(q;q)_k}.
\end{align*}

By using the Jacobi triple product identity, equation \ref{eq: JTP}, we find that, 
$$\frac{1}{(q;q)_\infty}=\sum\limits_{k=\max\{-n,0\}}^\infty\frac{q^{k(n+k)}}{(q;q)_{n+k}(q;q)_k}.$$
\end{proof}
\begin{remark}
Notice that if we instead consider $\underline{\mu}^c(\{N_{m+\frac{1}{2}}(\underline{z})=n\})$ for any $m\in\mathbb{Z}$ we prove the same identity. This is since $m$ is just a shift for the the state and the factors involving $m$ cancel out. Also since the particle-hole symmetry was used to find $\underline{\nu}^n(\{N^{\overset{\leftarrow}{p}}_{m+\frac{1}{2}}(\underline{z})=k\})$ in Theorem \ref{thm: half infinite ASEP particles under nu n}, if we sum this distribution over $k$ we again recover the Durfee rectangles identity.
\end{remark}
\par By considering the distribution of a single site under $\underline{\nu}^n$ (Lemma \ref{lem: site under nu n}) we find a combinatorial identity that seems to be related to the Durfee rectangle identity.
\crankiden*
\vspace{2mm}\begin{proof}
    Clearly for any $n\in\mathbb{Z}$, 
    $$\underline{\nu}^{n}(\{z_1=1\})+\underline{\nu}^{n}(\{z_1=0\})=1.$$
    By Lemma \ref{lem: site under nu n} and rearranging we find that, 
    $$\sum\limits_{k=\max\{-n,0\}}\frac{q^{(k+1)(n+k)}}{(q;q)_k(q;q)_{n+k}}+\sum\limits_{k=\max\{1-n,0\}}\frac{q^{k(n+k)}}{(q;q)_k(q;q)_{n-1+k}}=\frac{1}{(q;q)_\infty}.$$
\end{proof}
\par We now demonstrate that two other well-known combinatorial identities (Euler's identity and the $q$-Binomial Theorem) arise as consequences of Theorem \ref{thm: half infinite asep particles} and Lemma \ref{lem: finite asep particles}. 
\euler*
\vspace{2mm}\begin{proof} For any $m \in\mathbb{Z}$ and $c \in \mathbb{R}$ it is clear that, 
$\sum\limits_{k=0}^\infty \underline{\mu}^c (\{N_m^{\overset{\leftarrow}{p}}(\underline{z})=k\})=1.$ By Theorem \ref{thm: half infinite asep particles}, that is 
$$\sum\limits_{k=0}^\infty \frac{q^{k(c-m)}q^{\frac{k(k-1)}{2}}}{(q;q)_{k}(-q^{c-m};q)_\infty}=1.$$
If we let $z:=q^{c-m}$ and rearrange we find Euler's identity, 
$$\sum\limits_{k=0}^\infty \frac{q^{\frac{k(k-1)}{2}}z^k}{(q;q)_k}=(-z;q)_\infty.$$
\end{proof}

\par We also see the $q$-Binomial Theorem:
\qbin*
\vspace{2mm}\begin{proof} For any $m_1<m_2\in\mathbb{Z}$ and $c \in \mathbb{R}$ it is clear that, $\sum\limits_{k=0}^{m_2-m_1-1}\underline{\mu}^c(\{N^{\overset{\leftarrow}{p}}_{m_2-1+\frac{1}{2}}(\underline{z})-N^{\overset{\leftarrow}{p}}_{m_{1}+\frac{1}{2}}(\underline{z})=k\})=1.$ By Lemma \ref{lem: finite asep particles}, that is 
$$\sum\limits_{k=0}^{m_2-m_1-1} \frac{q^{k(c+1-m_2)}q^{\frac{k(k-1)}{2}}}{(-q^{c+1-m_2};q)_{m_2-m_1-1}}\begin{bmatrix}
m_2-m_1-1\\
k
\end{bmatrix}_q=1.$$
Now if we let $z:=q^{c+1-m_2}$, $m:=m_2-m_1-1$ and rearrange we have the $q$-binomial theorem, 
$$\sum\limits_{k=0}^mq^{\frac{k(k-1)}{2}}z^k\begin{bmatrix}
m\\
k
\end{bmatrix}_{q}=(-z;q)_m.$$
\end{proof}
All of these proofs are purely probabilistic. In Appendix \ref{appendix: combinatorics} we give alternative, combinatorial proofs to Theorem \ref{thm: half infinite asep particles} and Lemma \ref{lem: finite asep particles} which shed light on why we should expect combinatorial identities from these probabilistic questions.

\section{The asymmetric simple exclusion process with second class particles}
Now we consider the asymmetric simple exclusion process with 2 species of particles, first and second class particles. Both species of particles evolve according to ASEP dynamics obeying the exclusion rule within their own class, with first class particles viewing second class as `holes'. That is if,
\begin{itemize}
    \item a second class particle tries to jump into a space occupied by a first class particle the jump is blocked and nothing changes
    \item a first class particle tries to jump into a space occupied by a second class particle it takes that space and the two particles swap positions. 
\end{itemize}
To construct an ASEP with first class particles and some fixed $d>0$ second class particles we use the basic coupling method. This idea of defining second class particles via a basic coupling is similar to the work of Ferrari, Kipnis and Saada \cite{fer_kip_saa_coupling} and later work of Tracy and Widom \cite{tracy-widom} as well as Bal\'azs and Sepp\"al\"ainen \cite{bal_sepp_coupling} (also work of Liggett \cite{ligg_coupling} in 1976 considered coupling for simple exclusion processes).
\par We will study the behaviour of these second class particles under the blocking measure. 
\subsection{Constructing a Basic Coupling of Two ASEPs} \label{subsect: asep coupling}~
\par To construct a basic coupling we start with an ASEP on $\mathbb{Z}$, $\{\underline{\xi}(t)\}$, and also a process on the particle labels. For some given $d>0$, this label process is defined by an ASEP on the half infinite line $\mathbb{Z}_{\geq 0}$ with exactly $d$ particles moving with a left drift, $\{\underline{\gamma}(t)\}$. The dynamics of $\{\underline{\gamma}(t)\}$ is restricted by the state of $\{\underline{\xi}(t)\}$, such that a particle jump over the edge $(i,i+1)$ in $\underline{\gamma}(t)$ can only occur if the particles $i$ and $i+1$ in $\underline{\xi}(t)$ are nearest neighbours. Lastly we construct a second ASEP on $\mathbb{Z}$, $\{\underline{\eta}(t)\}$,  by at each time removing the particles corresponding to the label process from $\underline{\xi}(t)$. We see that this defines a basic coupling of two ASEPs with the second class particles being defined as the particles corresponding to the label process. We make this procedure precise below.
\subsubsection{ASEP on $\mathbb{Z}$}\label{subsubsect: xi}~
\par We consider an asymmetric simple exclusion process on , 
$\Omega\defeq\{\underline{z}\in\{0,1\}^\mathbb{Z}:N^{\overset{\leftarrow}{p}}(\underline{\eta}), N^{\overset{\rightarrow}{h}}(\underline{\eta})<\infty\}$, under the blocking measure $\underline{\mu}^c$ (as in Section \ref{sect: ASEP}) At time $t\in\mathbb{R}_{\geq 0}$ we will denote the state of the process by $\underline{\xi}(t)\in\Omega$.
\par At any time $t \in \mathbb{R}_{\geq 0}$ we attach to the state $\underline{\xi}(t)$ a labelling of the particles. That is, we enumerate the particles in $\underline{\xi}(t)$ from left to right starting from $0$ for the left most particle and increasing as we go to the right. So a particle's label is simply the number of particles to its left. When we consider second class particles we will be interested in the label of only $d>0$ certain particles.
\subsubsection{ASEP on $\mathbb{Z}_{\geq 0}$ with exactly $d$ particles (Label process)}\label{subsubsect: label process}~
\par Alongside the asymmetric simple exclusion process $\{\underline{\xi}(t)\}_{t\in\mathbb{R}_{\geq 0}}$ we also consider an ASEP on the half infinite integer line, with exactly $d>0$ particles moving with a left drift. That is, an interacting particle system on the state space
$\hat{\Omega}\defeq\{\underline{z}\in\{0,1\}^{\mathbb{Z}_{\geq0}}: \sum\limits_{i=0}^\infty z_i=d\}$,
which evolves according to nearest neighbour particle jumps with jump rates, (for $i\geq 0$)
$$\hat{p}(z_i,z_{i+1})=q\mathbb{I}\{z_i \neq 0\}\mathbb{I}\{z_{i+1}\neq 1\} \hspace{5mm}\text{ and }\hspace{5mm} \hat{q}(z_i,z_{i+1})=\mathbb{I}\{z_{i+1} \neq 0\}\mathbb{I}\{z_{i}\neq 1\}.$$
Note that by the definition of the state space $\hat{\Omega}$ any state has exactly $d$ particles and so there is a closed boundary to the left of site $0$. At time $t\in\mathbb{R}_{\geq 0}$ we will denote the state of this process by $\underline{\gamma}(t)$.
\vspace{1mm}\par Let $\underline{y}(t)=(y_1(t),...,y_d(t))$, with $0\leq y_1(t)<y_2(t)<...<y_d(t)$, be the vector of sites where the particles are in $\underline{\gamma}(t)$, i.e. $\forall j \in \{1,..,d\}$,
$$y_j(t)=i \text{ if } \gamma_i=1 \text{ and } \sum\limits_{k=0}^i\gamma_k=j.$$
\par In fact, we will consider a restriction of the process $\{\underline{\gamma}(t)\}_{t\in\mathbb{R}_{\geq0}}$ by conditioning on the process $\{\underline{\xi}(t)\}_{t\in\mathbb{R}_{\geq0}}$. We say that a particle jump which is possible over the edge $(i,i+1)$ in the half-infinite ASEP can only happen if the particles with labels $i$ and $i+1$ are nearest neighbours at that time in the ASEP on $\Omega$. So, we adjust the jump rates $\hat{p}$ and $\hat{q}$ to, 
\small\begin{align*}
&\hat{p}^*(\gamma_i(t),\gamma_{i+1}(t))=q\mathbb{I}\{\gamma_i(t) \neq 0\}\mathbb{I}\{\gamma_{i+1}(t)\neq 1\}\left(\sum\limits_{j\in\mathbb{Z}}\mathbb{I}\left\{\sum\limits_{k=-\infty}^{j-1}\xi_k(t)=i\right\}\mathbb{I}\{\xi_j(t)=1\}\mathbb{I}\{\xi_{j+1}(t)=1\}\right)\\
\text{ and, }&\hat{q}^*(\gamma_i(t),\gamma_{i+1}(t))=\mathbb{I}\{\gamma_{i+1}(t) \neq 0\}\mathbb{I}\{\gamma_{i}(t)\neq 1\}\left(\sum\limits_{j\in\mathbb{Z}}\mathbb{I}\left\{\sum\limits_{k=-\infty}^{j-1}\xi_k(t)=i\right\}\mathbb{I}\{\xi_j(t)=1\}\mathbb{I}\{\xi_{j+1}(t)=1\}\right).
\end{align*}
\normalsize For this restricted process, let the vector of sites where particles are in $\underline{\gamma}(t)$ be $\underline{x}(t)$. Once we define a coupling and thus define $d$ second class particles, $\{\underline{x}(t)\}_{t\in\mathbb{R}_{\geq 0}}$ will be process that gives the labels for the second class particles.
\par So we have an ASEP $\{\underline{\xi}(t)\}_{t\in\mathbb{R}_{\geq0}}$ with some label process sitting on top where the $d$ particles of interest can swap only when they are neighbouring another particle. For example see Figure \ref{fig: label process} below (the red particles in $\underline{\xi}(t)$ are those which correspond to the process $\{\underline{x}(t)\}_{t\in\mathbb{R}_{\geq 0}}$). In this example the particle jump at sites 2 and 3 is allowed but not the jump at sites 6 and 7 in $\underline{\gamma}(t)$. 
\vspace{10mm}
\begin{figure}[H]
    \centering
    \begin{tikzpicture}[scale=0.6]
    \draw[thick, ->] (-11,0)--(0,0);
    \foreach \x in {-11,-10,-9,-8,-7,-6,-5,-4,-3,-2,-1} \draw[thick, -](\x cm, -0.03)--(\x cm, 0.1);
    \node (a) at (-11,-1) [label= 0]{};
    \node (a) at (-10,-1) [label= 1]{};
    \node (a) at (-9,-1) [label= 2]{};
    \node (a) at (-8,-1) [label= 3]{};
    \node (a) at (-7,-1) [label= 4]{};
    \node (a) at (-6,-1) [label= 5]{};
    \node (a) at (-5,-1) [label= 6]{};
    \node (a) at (-4,-1) [label= 7]{};
    \node (a) at (-3,-1) [label= 8]{};
    \node (a) at (-2,-1) [label= 9]{};
    \node (a) at (-1,-1) [label= 10]{};
    \filldraw[black] (-11,0.5) circle (4pt);
    \filldraw[black] (-10,0.5) circle (4pt);
    \filldraw[black] (-8,0.5) circle (4pt);
    \filldraw[black] (-6,0.5) circle (4pt);
    \filldraw[black] (-5,0.5) circle (4pt);
    \node (a) at (1, -0.5) [label=$\underline{\gamma}(t)$]{};
    \node (a) at (7, -0.5) [label=and so $\underline{x}(t)\equal (0\comma 1\comma 3\comma 5\comma 6)$]{};

    \draw[thick, <->] (-11,-3)--(11,-3);
    \foreach \x in {-10,-9,-8,-7,-6,-5,-4,-3,-2,-1,0,1,2,3,4,5,6,7,8,9,10} \draw[thick, -](\x cm, -3.03)--(\x cm, -2.9)node[anchor=north]{$\x$};
    \filldraw[black] (-7,-2.5) circle (4pt);
    \node(a) at (-7.3,-3) [label=\textcolor{black}{\tiny{2}}]{};
    \filldraw[red] (-8,-2.5) circle (4pt);
     \node(a) at (-8.3,-3) [label=\textcolor{red}{\tiny{1}}]{};
    \filldraw[red] (-9,-2.5) circle (4pt);
     \node(a) at (-9.3,-3) [label=\textcolor{red}{\tiny{0}}]{};
    \filldraw[red] (-6,-2.5) circle (4pt);
    \node(a) at (-6.3,-3) [label=\textcolor{red}{\tiny{3}}]{};
    \filldraw[black] (8,-2.5) circle (4pt);
    \node(a) at (7.7,-3) [label=\textcolor{black}{\tiny{8}}]{};
    \filldraw[black] (9,-2.5) circle (4pt);
    \node(a) at (8.7,-3) [label=\textcolor{black}{\tiny{9}}]{};
    \filldraw[black] (10,-2.5) circle (4pt);
    \node(a) at (9.7,-3) [label=\textcolor{black}{\tiny{10}}]{};
    \filldraw[black] (10.25,-2.75) circle (1pt);
    \filldraw[black] (10.5,-2.75) circle (1pt);
    \filldraw[black] (10.75,-2.75) circle (1pt);
    \filldraw[black] (-4,-2.5) circle (4pt);
    \node(a) at (-4.3,-3) [label=\textcolor{black}{\tiny{4}}]{};
    \filldraw[red] (-1,-2.5) circle (4pt);
    \node(a) at (-1.3,-3) [label=\textcolor{red}{\tiny{5}}]{};
    \filldraw[red] (3,-2.5) circle (4pt);
    \node(a) at (2.7,-3) [label=\textcolor{red}{\tiny{6}}]{};
    \filldraw[black] (5,-2.5) circle (4pt);
    \node(a) at (4.7,-3) [label=\textcolor{black}{\tiny{7}}]{};
    \filldraw[black] (-10.25,-2.75) circle (1pt);
    \filldraw[black] (-10.5,-2.75) circle (1pt);
    \filldraw[black] (-10.75,-2.75) circle (1pt);
    \node (a) at (11.5, -3.5) [label=$\underline{\xi}(t)$]{};
    \draw[->,thick] (-5,1) arc
    [
        start angle=180,
        end angle=0,
        x radius=0.5cm,
        y radius =0.5cm
    ] ;
    \draw[->,thick, red] (-10,1) arc
    [
        start angle=180,
        end angle=0,
        x radius=0.5cm,
        y radius =0.5cm
    ] ;
\node (a) at (-4.5,0.85) [label={\textbf{X}}]{};
\node (a) at (-9.5,1.3) [label=\textcolor{red}{$q$}]{};
\node (a) at (-7.5,-1.7) [label=\textcolor{red}{$q$}]{};
    \draw[->,thick, red] (-7,-2) arc
    [
        start angle=0,
        end angle=180,
        x radius=0.5cm,
        y radius =0.5cm
    ] ;
    \end{tikzpicture}
    \caption{An example of the pair $(\underline{\xi}(t),\underline{x}(t))$ when $d=5$. The red arrows correspond to a possible second class particle label jump.}
    \label{fig: label process}
\end{figure}
\vspace{5mm}
\par We will be interested in where the second class particles are at time $t$. As discussed the red particles in the above figure for example will eventually be the second class particles. We denote the positions of the particles in $\underline{\xi}(t)$ corresponding to the label process $\underline{x}(t)$ by, $\underline{X}(t)=(X_1(t),...,X_d(t))$, i.e. $X_j(t)$ is the position of the particle with label $x_j(t)$ in $\underline{\xi}(t)$. So for each $j\in\{1,...,d\}$,
$$X_j(t)=i \text{ if } \xi_i(t)=1 \text{ and } \sum\limits_{k=0}^{i-1}\xi_k(t)=x_j(t).$$
It is clear to see by construction that at any time $t\geq 0$ we have that $X_1(t)< X_2(t) < ... <X_d(t)$.
\subsubsection{Basic Coupling of Two ASEPs}\label{subsubsect: eta}~
\par We now construct another process $\{\underline{\eta}(t)\}_{t\in\mathbb{R}_{\geq 0}}$, by taking the state of the ASEP at time $t$, $\underline{\xi}(t)$, and removing the particles which correspond to the label process $\underline{x}(t)$. That is, at any time $t\geq 0$,
$$\underline{\eta}(t)\defeq\underline{\xi}(t)-\sum_{j=1}^d\underline{\delta}_{X_j(t)} .$$
For example see Figure \ref{fig: coupling}
below.
\newpage
\begin{figure}[H]
    \centering
    \begin{tikzpicture}[scale=0.6]
      \node (a) at (-8, 2) [label=Suppose $\underline{\xi}(t)$ and $\underline{x}(t)$ are as follows\comma]{};
    \draw[thick, ->] (-11,0)--(0,0);
    \foreach \x in {-11,-10,-9,-8,-7,-6,-5,-4,-3,-2,-1} \draw[thick, -](\x cm, -0.03)--(\x cm, 0.1);
    \node (a) at (-11,-1) [label= 0]{};
    \node (a) at (-10,-1) [label= 1]{};
    \node (a) at (-9,-1) [label= 2]{};
    \node (a) at (-8,-1) [label= 3]{};
    \node (a) at (-7,-1) [label= 4]{};
    \node (a) at (-6,-1) [label= 5]{};
    \node (a) at (-5,-1) [label= 6]{};
    \node (a) at (-4,-1) [label= 7]{};
    \node (a) at (-3,-1) [label= 8]{};
    \node (a) at (-2,-1) [label= 9]{};
    \node (a) at (-1,-1) [label= 10]{};
    \filldraw[black] (-11,0.5) circle (4pt);
    \filldraw[black] (-10,0.5) circle (4pt);
    \filldraw[black] (-8,0.5) circle (4pt);
    \filldraw[black] (-6,0.5) circle (4pt);
    \filldraw[black] (-5,0.5) circle (4pt);
    \node (a) at (1, -0.5) [label=$\underline{\gamma}(t)$]{};
    \node (a) at (7, -0.5) [label=and so $\underline{x}(t)\equal (0\comma 1\comma 3\comma 5\comma 6)$]{};

    \draw[thick, <->] (-11,-3)--(11,-3);
    \foreach \x in {-10,-9,-8,-7,-6,-5,-4,-3,-2,-1,0,1,2,3,4,5,6,7,8,9,10} \draw[thick, -](\x cm, -3.03)--(\x cm, -2.9)node[anchor=north]{$\x$};
    \filldraw[black] (-7,-2.5) circle (4pt);
    \node(a) at (-7.3,-3) [label=\textcolor{black}{\tiny{2}}]{};
    \filldraw[red] (-8,-2.5) circle (4pt);
     \node(a) at (-8.3,-3) [label=\textcolor{red}{\tiny{1}}]{};
    \filldraw[red] (-9,-2.5) circle (4pt);
     \node(a) at (-9.3,-3) [label=\textcolor{red}{\tiny{0}}]{};
    \filldraw[red] (-6,-2.5) circle (4pt);
    \node(a) at (-6.3,-3) [label=\textcolor{red}{\tiny{3}}]{};
    \filldraw[black] (8,-2.5) circle (4pt);
    \node(a) at (7.7,-3) [label=\textcolor{black}{\tiny{8}}]{};
    \filldraw[black] (9,-2.5) circle (4pt);
    \node(a) at (8.7,-3) [label=\textcolor{black}{\tiny{9}}]{};
    \filldraw[black] (10,-2.5) circle (4pt);
    \node(a) at (9.7,-3) [label=\textcolor{black}{\tiny{10}}]{};
    \filldraw[black] (10.25,-2.75) circle (1pt);
    \filldraw[black] (10.5,-2.75) circle (1pt);
    \filldraw[black] (10.75,-2.75) circle (1pt);
    \filldraw[black] (-4,-2.5) circle (4pt);
    \node(a) at (-4.3,-3) [label=\textcolor{black}{\tiny{4}}]{};
    \filldraw[red] (-1,-2.5) circle (4pt);
    \node(a) at (-1.3,-3) [label=\textcolor{red}{\tiny{5}}]{};
    \filldraw[red] (3,-2.5) circle (4pt);
    \node(a) at (2.7,-3) [label=\textcolor{red}{\tiny{6}}]{};
    \filldraw[black] (5,-2.5) circle (4pt);
    \node(a) at (4.7,-3) [label=\textcolor{black}{\tiny{7}}]{};
    \filldraw[black] (-10.25,-2.75) circle (1pt);
    \filldraw[black] (-10.5,-2.75) circle (1pt);
    \filldraw[black] (-10.75,-2.75) circle (1pt);
    \node (a) at (11.5, -3.5) [label=$\underline{\xi}(t)$]{};
    \node (a) at (-11, -5.5) [label=then $\underline{\eta}(t)$ is\comma ]{};
    \draw[thick, <->] (-11,-7)--(11,-7);
    \foreach \x in {-10,-9,-8,-7,-6,-5,-4,-3,-2,-1,0,1,2,3,4,5,6,7,8,9,10} \draw[thick, -](\x cm, -7.03)--(\x cm, -6.9)node[anchor=north]{$\x$};
    \filldraw[black] (-7,-6.5) circle (4pt);
    \filldraw[black] (8,-6.5) circle (4pt);
    \filldraw[black] (9,-6.5) circle (4pt);
    \filldraw[black] (10,-6.5) circle (4pt);
    \filldraw[black] (10.25,-6.75) circle (1pt);
    \filldraw[black] (10.5,-6.75) circle (1pt);
    \filldraw[black] (10.75,-6.75) circle (1pt);
    \filldraw[black] (-4,-6.5) circle (4pt);
    \filldraw[black] (5,-6.5) circle (4pt);
    \filldraw[black] (-10.25,-6.75) circle (1pt);
    \filldraw[black] (-10.5,-6.75) circle (1pt);
    \filldraw[black] (-10.75,-6.75) circle (1pt);
    \node (a) at (11.5, -7.5) [label=$\underline{\eta}(t)$]{};
    \end{tikzpicture}
    \caption{An example of constructing $\underline{\eta}(t)$ from the pair $(\underline{\xi}(t),\underline{x}(t))$ (when $d=5$).}
    \label{fig: coupling}
\end{figure}
 
\begin{lem}\label{lem: basic coupling}
The process $\{\underline{\eta}(t)\}_{t\in\mathbb{R}_{\geq 0}}$ is an ASEP and the pair $\{\underline{\eta}(t), \underline{\xi}(t)\}_{t\in\mathbb{R}_{\geq 0}}$ is a basic coupling of two asymmetric simple exclusion processes which only differ at $d$ many sites at any time (i.e. the processes evolve together with the same Poison clocks describing the particle jumps). 
\end{lem} 
\begin{proof}
First we confirm that $\{\underline{\eta}(t)\}_{t\in\mathbb{R}_{\geq 0}}$ is an asymmetric simple exclusion process. Since $\underline{\xi}(t)$ takes values in the state space $\Omega=\{\underline{z}\in\{0,1\}^\mathbb{Z}:N^{\overset{\leftarrow}{p}}(\underline{\eta}), N^{\overset{\rightarrow}{h}}(\underline{\eta})<\infty\}$ and by definition $\underline{\eta}(t)=\underline{\xi}(t)-\sum_{j=1}^d\underline{\delta}_{X_j(t)}$, it is clear that the process $\{\underline{\eta}(t)\}_{t\in\mathbb{R}_{\geq 0}}$ also has state space $\Omega$. Now we will show that the process $\{\eta(t)\}_{t\in\mathbb{R}_{\geq 0}}$ evolves according to ASEP dynamics; that is each particle jumps left with rate $q$ and right with rate $1$ unless the jump is blocked. Let us consider the following possible cases for particle jumps in $\{\underline{\xi}(t)\}_{t\in\mathbb{R}_{\geq 0}}$:
\begin{itemize}
    \item A particle associated to the label process (red particles in Figure \ref{fig: coupling}) makes a particle jump. When this happens the particles in $\underline{\xi}(t)$ that are not associated with the label process $\underline{x}(t)$ are not affected and so there is no change to $\underline{\eta}(t)$.
    \item A particle not associated to the label process makes a particle jump. When this jump happens in $\underline{\xi}(t)$ it will happen in $\underline{\eta}(t)$ with the same jump rate ($q$ to the left and $1$ to the right). 
    \item Lastly, the label process may evolve resulting in the swapping of neighbouring particles in $\underline{\xi}(t)$ (one associated with the label process and one that is not). 
    \begin{itemize}
    \item Suppose a particle at site $i\geq 0$ in $\underline{\gamma}(t)$ jumps right (meaning particle $i$ and $i+1$ are neighbours in $\underline{\xi}(t)$ say at sites $j$ and $j+1$). Then in $\underline{\xi}(t)$ this results in particle $i$ no longer being associated with the label process and particle $i+1$ becomes associated with the label process. That is in $\underline{\eta}(t)$ the particle at site $j+1$ jumps to the left with rate equal to a right jump in $\{\underline{\gamma}(t)\}_{t\in\mathbb{R}_{\geq 0}}$ so with rate $q$.
    \item Similarly suppose that the particle at site $i>0$ in $\underline{\gamma}(t)$ jumps left (meaning particle $i-1$ and $i$ are neighbours in $\underline{\xi}(t)$ say at sites $j$ and $j+1$). Then in $\underline{\xi}(t)$ this results in particle $i$ no longer being associated with the label process and particle $i-1$ becomes associated with the label process. That is in $\underline{\eta}(t)$ the particle at site $j$ jumps to the right with rate equal to a left jump in $\{\underline{\gamma}(t)\}_{t\in\mathbb{R}_{\geq 0}}$ so with rate $1$.
    \end{itemize}
\end{itemize}
Thus $\{\underline{\eta}(t)\}_{t \in\mathbb{R}_{\geq 0}}$ evolves according to particle jumps at rates, 

\small\begin{align*}
p(\eta_i(t),\eta_{i+1}(t))=\mathbb{I}\{\xi_i(t) \neq 0\}&\mathbb{I}\{\xi_{i+1}(t)\neq 1\}\\
&+\mathbb{I}\{\gamma_{\sum\limits_{j=-\infty}^{i}\xi(t)}(t)\neq 0\}\mathbb{I}\{\gamma_{\sum\limits_{j=-\infty}^{i}\xi(t)-1}(t)\neq 1 \}\mathbb{I}\{\xi_i(t)=1\}\mathbb{I}\{\xi_{i+1}(t)=1\}
\end{align*}
\normalsize and, 
\small \begin{align*}
q(\eta_i(t),\eta_{i+1}(t))=q\Big(\mathbb{I}\{\xi_{i+1}(t) \neq 0\}&\mathbb{I}\{\xi_{i}(t)\neq 1\}\\
&+\mathbb{I}\{\gamma_{\sum\limits_{j=-\infty}^{i-1}\xi(t)}(t)\neq 0 \}\mathbb{I}\{\gamma_{\sum\limits_{j=-\infty}^{i-1}\xi(t)+1}(t)\neq 1\}\mathbb{I}\{\xi_i(t)=1\}\mathbb{I}\{\xi_{i+1}(t)=1\}\Big)
\end{align*} 

\normalsize \noindent or equivalently written only in terms of $\underline{\eta}(t)$,
$$p(\eta_i(t),\eta_{i+1}(t))=\mathbb{I}\{\eta_i(t) \neq 0\}\mathbb{I}\{\eta_{i+1}(t)\neq 1\} \hspace{3mm}\text{ and }\hspace{3mm} q(\eta_i(t),\eta_{i+1}(t))=q\mathbb{I}\{\eta_{i+1}(t) \neq 0\}\mathbb{I}\{\eta_{i}(t)\neq 1\}.$$
\par By construction it is clear that the two processes $\{\underline{\eta}(t)\}_{t\in\mathbb{R}_{\geq 0}}$ and $\{\underline{\xi}(t)\}_{t\in\mathbb{R}_{\geq 0}}$ evolve together according to ASEP dynamics and thus the pair gives a basic coupling.
\end{proof}
\par This coupling defines the asymmetric exclusion process with $d$ second class particles. Here, $\underline{x}(t)$ is the label process for the second class particles and $\underline{X}(t)$ gives the positions of the second class particles at time $t$.

\subsection{Distributional results for the basic coupling}\label{subsect: coupling dist results}~
\par Let us now look at some distributional results for the process with second class particles given we know the distribution of the ASEP with only first class particles.
\begin{lem}\label{lem: nu^n to nu^n+d}
With the relation $\underline{\eta}=\underline{\xi}-\sum\limits_{j=1}^d\underline{\delta}_{X_j}$, when we reach stationarity for the coupling, we have that, 
$$\underline{\xi} \sim \underline{\nu}^n \hspace{1mm}\Longleftrightarrow\hspace{1mm} \underline{\eta}\sim\underline{\nu}^{n+d}.$$
\end{lem}
\vspace{2mm}\begin{proof}
This is clear by definition of $\underline{\nu}^n$ and since,
\begin{align*}
    N(\underline{\xi})&=\sum\limits_{i=1}^\infty (1-\xi_i)-\sum\limits_{i=-\infty}^0\xi_i\\
    &=\sum\limits_{i=1}^\infty (1-\{\underline{\eta}+\sum\limits_{j=1}^d\underline{\delta}_{X_j}\}_i)-\sum\limits_{i=-\infty}^0\{\underline{\eta}+\sum\limits_{j=1}^d\underline{\delta}_{X_j}\}_i\\
    &=N(\underline{\eta})-\sum\limits_{j=1}^d\sum\limits_{i=-\infty}^\infty(\underline{\delta}_{X_j})_i\\
    &=N(\underline{\eta})-d.
\end{align*}
Since $\underline{\nu}^n$ is the unique stationary distribution for the marginal process conditioned on having conserved quantity $n$ it holds that $\underline{\xi} \sim \underline{\nu}^n \hspace{1mm}\Longleftrightarrow\hspace{1mm} \underline{\eta}\sim\underline{\nu}^{n+d}$.
\end{proof}
\begin{lem} \label{lem: mu^c to mu^c-d}
With the relation $\underline{\eta}=\underline{\xi}-\sum\limits_{j=1}^d\underline{\delta}_{X_j}$, when we reach stationarity for the coupling we have that, 
$$\underline{\xi} \sim \underline{\mu}^{c-d} \hspace{1mm}\Longleftrightarrow\hspace{1mm} \underline{\eta}\sim\underline{\mu}^c.$$
\end{lem}
\vspace{2mm}\begin{proof}
First let us show that $\underline{\xi}\sim \underline{\mu}^{c-d}\Rightarrow \underline{\eta}\sim\underline{\mu}^c$. Assume $\underline{\xi}\sim\underline{\mu}^{c-d}$ and consider,
$$\mathbb{P}(\underline{\eta}=z)=\mathbb{P}\left(\underline{\eta}=\underline{z}, N(\underline{\eta})=N(\underline{z})\right)=\underline{\nu}^{N(\underline{z})}(\underline{z}) \cdot \mathbb{P}\left(N(\underline{\eta})=N(\underline{z})\right).$$
Let $\tau$ denote the left shift $(\tau\underline{y})_i=y_{i+1}$, by Lemma \ref{lem: shift by tau is a shift in c} we have that $\underline{\mu}^c(\tau^j\underline{y})=\underline{\mu}^{c+j}(\underline{y})$, and so
\begin{align*}
    \mathbb{P}\left(N(\underline{\eta})=N(\underline{z})\right)&=\mathbb{P}\left(N(\underline{\xi})=N(\underline{z})-d\right)\\
    &=\sum\limits_{\underline{y}:N(\underline{y})=N(\underline{z})-d} \underline{\mu}^{c-d}(\underline{y})\\
    &=\sum\limits_{\underline{u}:N(\underline{u})=N(\underline{z})} \underline{\mu}^{c-d}(\tau^d\underline{u})\\
    &=\sum\limits_{\underline{u}:N(\underline{u})=N(\underline{z})} \underline{\mu}^{c}(\underline{u})\\
    &=\underline{\mu}^c\left(\{N(\underline{\eta})=N(\underline{z})\}\right).
\end{align*}
Hence, $\mathbb{P}(\underline{\eta}=\underline{z})=\underline{\nu}^{N(\underline{z})}(\underline{z}) \cdot \underline{\mu}^c\left(\{N(\underline{\eta})=N(\underline{z})\}\right)= \underline{\mu}^c(\underline{z})$.
By a similar argument we can show that $\underline{\eta}\sim\underline{\mu}^c \Rightarrow \underline{\xi} \sim \underline{\mu}^{c-d}$.
\end{proof}
\vspace{3mm}\subsection{Stationary Distribution for the Label Process}\label{subsect: asep label process}~
\par Let us now look at stationarity for the label process for the second class particles, $\underline{x}(t)$. By construction, $\underline{x}(t)$ gives the positions of the particles in a half infinite ASEP with exactly $d$ particles, $\underline{\gamma}(t)$,  at time $t$. The evolution of this half infinite ASEP depends on the state of the ASEP on $\mathbb{Z}$, $\underline{\xi}(t)$, as previously described. 
\par The stationary distribution for the label process $\{\underline{x}(t)\}_{t\in\mathbb{R}_{\geq 0}}$ is as given below, where we use the notation, $Z^d_{+}\defeq\{\underline{x} \in \mathbb{Z}^d: 0\leq x_1<x_2<...<x_d\}$.
\begin{thm} \label{thrm: label process stationary dist} The process $\underline{x}(t)$ has reversible stationary distribution $\underline{\pi}$ where,
$$\underline{\pi}(\underline{x})=\prod\limits_{i=1}^d\left(1-q^i\right)\cdot q^{\sum\limits_{i=1}^d x_i-\frac{d(d-1)}{2}} \hspace{10mm} \text{for any } \underline{x} \in Z_+^d.$$
\end{thm}

\begin{proof}
  (We adapt the proofs of Bal\'azs and Sepp\"al\"ainen \cite{bal_sepp_coupling}, Lemma 4.1 and also \cite{bal_sepp_annals}, Proposition 3.1).
\par \vspace{1mm} \noindent First we show that, $\sum\limits_{\underline{x}\in Z_+^d}\underline{\pi}(\underline{x})=1$.
\begin{align*}
        \sum\limits_{\underline{x}\in Z_+^d}\underline{\pi}(\underline{x})&=\prod\limits_{i=1}^d(1-q^i)q^{-\frac{d(d-1)}{2}}\cdot\sum\limits_{\underline{x}\in Z_+^d}q^{\sum\limits_{i=1}^d x_i} \\
        &= \prod\limits_{i=1}^d(1-q^i)q^{-\frac{d(d-1)}{2}}\cdot\sum\limits_{x_1=0}^\infty\sum\limits_{x_2=x_1+1}^\infty \cdots \sum\limits_{x_d=x_{d-1}+1}^\infty q^{\sum\limits_{i=1}^d x_i} \\
    &= \prod\limits_{i=1}^d(1-q^i)q^{-\frac{d(d-1)}{2}}\cdot\sum\limits_{x_1=0}^\infty \cdots \sum\limits_{x_{d-1}=x_{d-2}+1}^\infty\sum\limits_{n=1}^\infty q^{\sum\limits_{i=1}^{d-2} x_i+2x_{d-1}+n}\\
    &\vdots \\
    &= \prod\limits_{i=1}^d(1-q^i)q^{-\frac{d(d-1)}{2}}\cdot\sum\limits_{x_1=0}^\infty q^{dx_1}\sum\limits_{n_{d-1}=1}^\infty q^{(d-1)n_{d-1}}\cdots \sum\limits_{n_1=1}^\infty q^{n_1}\\
    &=\prod\limits_{i=1}^d(1-q^i)q^{-\frac{d(d-1)}{2}}\cdot \frac{1}{1-q^d}\prod\limits_{i=1}^{d-1}\frac{q^i}{1-q^i}\\
    &=q^{-\frac{d(d-1)}{2}}\cdot q^{\sum\limits_{i=1}^{d-1}i}=1.
\end{align*}
\par Let's consider the label process without the restrictions on the evolution given by the underlying asymmetric exclusion process on $\mathbb{Z}$, $\{\underline{\xi}(t)\}_{t\in\mathbb{R}_{\geq 0}}$. That is the process $\underline{y}(t)$, which gives the positions of the particles in an ASEP on $\mathbb{Z}_{\geq 0}$ with exactly $d$ particles and closed left boundary, $\{\underline{\gamma}(t)\}_{t \in \mathbb{R}_{\geq 0}}$. The jump rates for this process are (for $i\geq 0$), 
$$\hat{p}(z_i,z_{i+1})=q\mathbb{I}\{z_i \neq 0\}\mathbb{I}\{z_{i+1}\neq 1\} \hspace{5mm}\text{ and }\hspace{5mm} \hat{q}(z_i,z_{i+1})=\mathbb{I}\{z_{i+1} \neq 0\}\mathbb{I}\{z_{i}\neq 1\}.$$
For $j\in\{1,..,d\}$, let $\underline{e}_j$ be the $d$ long vector such that $(\underline{e}_j)_k=\mathbb{I}\{k=j\}$. When the $j^\text{th}$ particle jumps left  then its position is decreased by 1, $\underline{y}(t)\mapsto \underline{y}(t)-\underline{e}_j$. Similarly, if the $j^\text{th}$ particle jumps right then its position increases by 1, $\underline{y}(t)\mapsto \underline{y}(t)+\underline{e}_j$. 
\par We now show that,
$$\underline{\pi}(\underline{y})=\prod\limits_{i=1}^d\left(1-q^i\right)\cdot q^{\sum\limits_{i=1}^d y_i-\frac{d(d-1)}{2}} \hspace{10mm} \text{for any } \underline{y} \in Z_+^d,$$
is stationary for $\{\underline{y}(t)\}_{t\in\mathbb{R}_{t\geq 0}}$. To do this we check that detailed balance holds; that is for any $\underline{y}\in Z_+^d$ and $j\in\{1,...,d\}$, we check that,
$$\underline{\pi}(\underline{y})q\mathbb{I}\{z_i\neq 0\}\mathbb{I}\{z_{i+1}\neq 1\}=\underline{\pi}(\underline{y}+\underline{e}_j)\mathbb{I}\{z^{(i,i+1)}_{i+1}\neq 0\}\mathbb{I}\{z^{(i,i+1)}_{i}\neq 1\},$$
where $\underline{z}\in\hat{\Omega}$ is the current state of $\{\underline{\gamma}(t)\}_{t\in\mathbb{R}_{\geq 0}}$ with the $j^\text{th}$ particle at site $i\geq 0$. This trivially holds when a jump is not possible over the edge $(i,i+1)$. Now consider the case when $z_i=1$ and $z_{i+1}=0$, 
\begin{align*}
\underline{\pi}(\underline{y})q&=\prod\limits_{i=1}^d\left(1-q^i\right)\cdot q^{\sum\limits_{i=1}^d y_i-\frac{d(d-1)}{2}}\cdot q\\
&= \prod\limits_{i=1}^d\left(1-q^i\right)\cdot q^{\sum\limits_{i=1}^d y_i+1-\frac{d(d-1)}{2}}\\
&=\prod\limits_{i=1}^d\left(1-q^i\right)\cdot q^{\sum\limits_{i=1}^d (\underline{y}+\underline{e}_j)_i-\frac{d(d-1)}{2}}\\
&= \underline{\pi}(\underline{y}+\underline{e}_j).  
\end{align*}

\par Now we re-introduce the restriction coming from the process $\{\underline{\xi}(t)\}_{t\in\mathbb{R}_{\geq 0}}$. Consider an approximating process $\underline{x}^m(t)$ for some $m\in\mathbb{Z}_{\geq0}$ which has the same initial value, $\underline{x}^m(0)=\underline{x}(0)$. The process $\{\underline{x}^m(t)\}_{t\in\mathbb{R}_{\geq 0}}$ evolves so that the underlying process $\{\underline{\xi}(t)\}_{t\in\mathbb{R}_{\geq 0}}$ only restricts its motion between states in the range $\{\underline{x}\in Z_+^d:x_d<m\}$. We couple the processes together so that $\underline{x}(t)=\underline{x}^m(t)$ until the first time that one of them exits the range $\{\underline{x}\in Z_+^d:x_d<m\}$.
\par Suppose we fix $m$ for now. Let $0=t_0<t_1<t_2<...$ be a partition of time such that $t_j\nearrow \infty$ and the state of the process $\underline{\xi}(t)$ on the sites $(-\infty,j_m(t_i)]$ (where $j_m(t_i)$ is the site of the $m^\text{th}$ particle at time $t_i$) is constant on each time interval $t\in[t_i,t_{i+1})$. Then on each time interval $[t_i,t_{i+1})$ the process $\underline{x}^m(t)$ is a continuous time Markov chain with jump rates, (for $k\in\{1,...,d\}$) when $\underline{x}+\underline{e}_k\in Z^+_d$,
\small $$\rho(\underline{x},\underline{x}+\underline{e}_k)=\begin{cases}
q\mathbb{I}\{\gamma_{x_k}(t)=1\}\mathbb{I}\{\gamma_{x_k+1}=0\}\left(\sum\limits_{j\in\mathbb{Z}}\mathbb{I}\left\{\sum\limits_{\ell=-\infty}^{j-1}\xi_\ell(t)=k\right\}\mathbb{I}\{\xi_j(t)=1\}\mathbb{I}\{\xi_{j+1}=1\}\right) &\text{ if } x_k\leq m\\
q\mathbb{I}\{\gamma_{x_k}(t)=1\}\mathbb{I}\{\gamma_{x_k+1}=0\} &\text{ if } x_k> m,
\end{cases}$$
\normalsize with this rate being $0$ if $\underline{x}+\underline{e}_k\notin Z^+_d$; and when $\underline{x}-\underline{e}_k\in Z^+_d$,
\small$$\rho(\underline{x},\underline{x}-\underline{e}_k)=\begin{cases}
\mathbb{I}\{\gamma_{x_k}(t)=1\}\mathbb{I}\{\gamma_{x_k-1}=0\}\left(\sum\limits_{j\in\mathbb{Z}}\mathbb{I}\left\{\sum\limits_{\ell=-\infty}^{j-1}\xi_\ell(t)=k-1\right\}\mathbb{I}\{\xi_j(t)=1\}\mathbb{I}\{\xi_{j+1}=1\}\right) &\text{ if } x_k\leq m+1\\
\mathbb{I}\{\gamma_{x_k}(t)=1\}\mathbb{I}\{\gamma_{x_k-1}=0\} &\text{ if } x_k> m+1,
\end{cases}$$
\normalsize with this rate being $0$ when $\underline{x}-\underline{e}_k\notin Z^+_d$.
\par Since in the time intervals $[t_i,t_{i+1})$ the background giving the restriction of $\{\underline{y}(t)\}_{t\in\mathbb{R}_{\geq 0}}$ to $\{\underline{x}^m(t)\}_{t\in\mathbb{R}_{\geq 0}}$ is fixed we are in the setting of Proposition 5.10 of Liggett [II.,\cite{liggett}] and so on each interval $t\in[t_i,t_{i+1})$ the measure $\underline{\pi}$ is also reversible stationary for $\underline{x}^m(t)$. Thus $\underline{\pi}$ is reversible stationary for $\{\underline{x}^m(t)\}_{t\in\mathbb{R}_{\geq 0}}$.
\par The coupling between $\{\underline{x}^m(t)\}_{t\in\mathbb{R}_{\geq 0}}$ and $\{\underline{x}(t)\}_{t\in\mathbb{R}_{\geq 0}}$ gives us that $\underline{x}^m(t)\xrightarrow[m\rightarrow\infty]{}\underline{x}(t)$ a.s.\ , thus the measure $\underline{\pi}$ is also reversible stationary for the process $\{\underline{x}(t)\}_{t\in\mathbb{R}_{\geq 0}}$.
\end{proof}

\vspace{2mm}\begin{remark}
From the proof of the above theorem we see that the distribution of the labels of the second class particles is independent of the distribution of the process $\{\underline{\xi}(t)\}_{t\in\mathbb{R}_{\geq 0}}$.
\end{remark}
Using the coupling we give an alternative proof of Theorem \ref{thm: half infinite asep particles} (the distribution of the number of particles in a half infinite volume).

\vspace{2mm}\begin{proof}[Alternative proof of Theorem \ref{thm: half infinite asep particles}]
We prove this by considering the coupling for the case of a single second class particle, $\underline{\eta}=\underline{\xi}-\underline{\delta}_X$. Under the coupling we see that if there are $k$ particles to the left of site $m+1$ in $\underline{\eta}$ this is the same as one of two possible events in $\underline{\xi}$:
\begin{itemize}
    \item There are $k$ particles to the left of the site $m+1$ and the second class particle is to the right of site $m$. In this case the label of the second class particle is at least $k$.
    \item There are $k+1$ particles to the left of site $m+1$ and the second class particle is one of them. In this case the label of the second class particle is at most $k$.
\end{itemize}
Using Lemma \ref{lem: mu^c to mu^c-d} we then have that,
\begin{equation}\label{eq: Np recurrence}
\underline{\mu}^c(\{N^{\overset{\leftarrow}{p}}_{m+\frac{1}{2}}(\underline{\eta})=k\})=\underline{\mu}^{c-1}(\{N^{\overset{\leftarrow}{p}}_{m+\frac{1}{2}}(\underline{\xi})=k\})\mathbb{P}(x\geq k)+\underline{\mu}^{c-1}(\{N^{\overset{\leftarrow}{p}}_{m+\frac{1}{2}}(\underline{\xi})=k+1\})\mathbb{P}(x\leq k),
\end{equation}
where $x$ denotes the label of the second class particle. 
\par We want the measures above to all have the same parameter so we look to change $c-1$ to $c$; by Lemma \ref{lem: c shift in mu and N p,h} we have that,
$$ \underline{\mu}^{c-1}((\{N^{\overset{\leftarrow}{p}}_{m+\frac{1}{2}}(\underline{z})=k\})=\frac{q^{-k}}{1+q^{c-1-m}}\underline{\mu}^c(\{N^{\overset{\leftarrow}{p}}_{m+\frac{1}{2}}(\underline{z})=k\}).$$
Hence (\ref{eq: Np recurrence}) becomes, 
\begin{multline*}
\underline{\mu}^c(\{N^{\overset{\leftarrow}{p}}_{m+\frac{1}{2}}(\underline{z})=k\})= \frac{q^{-k}}{1+q^{c-1-m}}\underline{\mu}^c((\{N^{\overset{\leftarrow}{p}}_{m+\frac{1}{2}}(\underline{z})=k\})\cdot \mathbb{P}(x\geq k) \\ +\frac{q^{-(k+1)}}{1+q^{c-1-m}}\underline{\mu}^c((\{N^{\overset{\leftarrow}{p}}_{m+\frac{1}{2}}(\underline{z})=k+1\})\cdot \mathbb{P}(x\leq k).
\end{multline*}
By Theorem \ref{thrm: label process stationary dist} for $d=1$,
    $$\mathbb{P}(x\geq k)=\sum\limits_{i=k}^\infty (1-q)q^i=q^k \hspace{5mm} \text{and} \hspace{5mm} \mathbb{P}(x \leq k)=\sum\limits_{i=0}^k (1-q)q^i=1-q^{k+1}$$
thus we have that,
$$\underline{\mu}^c(\{N^{\overset{\leftarrow}{p}}_{m+\frac{1}{2}}(\underline{z})=k\})=\frac{\underline{\mu}^c((\{N^{\overset{\leftarrow}{p}}_{m+\frac{1}{2}}(\underline{z})=k\})}{1+q^{c-1-m}}+\frac{q^{-(k+1)}-1}{1+q^{c-1-m}}\underline{\mu}^c(\{N^{\overset{\leftarrow}{p}}_{m+\frac{1}{2}}(\underline{z})=k+1\}). $$
Rearranging gives the following recurrence, 
$$\underline{\mu}^c(\{N^{\overset{\leftarrow}{p}}_{m+\frac{1}{2}}(\underline{z})=k+1\})=\frac{q^{c-1-m}}{q^{-(k+1)}-1}\underline{\mu}^c(\{N^{\overset{\leftarrow}{p}}_{m+\frac{1}{2}}(\underline{z})=k\}).$$
We now solve this recurrence, 
\begin{align*}
    \underline{\mu}^c(\{N^{\overset{\leftarrow}{p}}_{m+\frac{1}{2}}(\underline{z})=k\})&=\frac{q^{c-1-m}}{q^{-k}-1}\underline{\mu}^c(\{N^{\overset{\leftarrow}{p}}_{m+\frac{1}{2}}(\underline{z})=k-1\})
    \\
    &=\frac{q^{2(c-1-m)}}{(q^{-k}-1)(q^{-(k-1)}-1)}\underline{\mu}^c(\{N^{\overset{\leftarrow}{p}}_{m+\frac{1}{2}}(\underline{z})=k-2\})
    \\
    \vdots \\
    &=\frac{q^{k(c-1-m)}}{\prod\limits_{j=0}^{k-1}(q^{-(k-j)}-1)}\underline{\mu}^c(\{N^{\overset{\leftarrow}{p}}_{m+\frac{1}{2}}(\underline{z})=0\})\\
    &= \frac{q^{k(c-1-m)}(-1)^k}{\prod\limits_{j=0}^{k-1}(1-q^{-(k-j)})}\underline{\mu}^c(\{N^{\overset{\leftarrow}{p}}_{m+\frac{1}{2}}(\underline{z})=0\}).
\end{align*}
We can calculate $\underline{\mu}^c(\{N^{\overset{\leftarrow}{p}}_{m+\frac{1}{2}}(\underline{z})=0\})$ explicitly to be $\prod\limits_{j\leq m}\frac{1}{1+q^{c-j}}$ and so we have that, 
    \begin{align*}
     \underline{\mu}^c(\{N^{\overset{\leftarrow}{p}}_{m+\frac{1}{2}}(\underline{z})=k\})&=   \frac{q^{k(c-1-m)}(-1)^k}{\prod\limits_{j=0}^{k-1}(1-q^{-(k-j)})}\prod\limits_{j\leq m}\frac{1}{1+q^{c-j}}\\ 
     &=\frac{q^{k(c-1-m)}(-1)^k}{(q^{-k};q)_k} \prod\limits_{i=0}^\infty\frac{1}{1+q^{i+c-m}}\\
     &=\frac{q^{k(c-1-m)}(-1)^k}{(q^{-k};q)_k(-q^{c-m};q)_\infty}.
    \end{align*}
    Finally we use the Pochhammer relation of Lemma \ref{lem: pochammer relation} giving,
    $$\underline{\mu}^c(\{N^{\overset{\leftarrow}{p}}_{m+\frac{1}{2}}(\underline{z})=k\})=\frac{q^{k(c-m-1)+\frac{k(k+1)}{2}}}{(q,q)_k(-q^{c-m},q)_\infty}=\frac{q^{k(c-m)+\frac{k(k-1)}{2}}}{(q;q)_k(-q^{c-m};q)_\infty}.$$
   
\end{proof}

\subsection{Second Class Particles Positions Distribution under $\underline{\mu}^c$} \label{subsect: asep positions}~
\par Now we look to describe the distribution of the positions of the second class particles within the system. We will consider two questions:
\begin{enumerate}
    \item For a given site what is the probability (under the blocking measure $\underline{\mu}^c$) that there is a second class particle there? 
    \item For a given $d$ distinct sites what is the probability (under the blocking measure $\underline{\mu}^c$) that the $d$ second class particles lie on these exact sites?
\end{enumerate}
The first question can be easily computed by looking at expectations under the blocking measures for $\underline{\eta}$ and $\underline{\xi}$. 
\vspace{3mm}\scpm*
 \begin{proof}
By the coupling we have that, $\eta_m=\xi_m-\sum\limits_{i=1}^d(\underline{\delta}_{X_i})_m$. We see that the event $\{\xi_m>\eta_m\}$ is exactly the event that $\{\sum\limits_{i=1}^d(\underline{\delta}_{X_i})_m>0\}$ or equivalently that $\{\sum\limits_{i=1}^d(\underline{\delta}_{X_i})_m=1\}$ since we are in the simple exclusion setting. Now if we take expectations over the coupling equation we have that, 
$$\mathbb{E}[\eta_m]=\mathbb{E}[\xi_m]-\mathbb{E}\left[\sum\limits_{i=1}^d(\underline{\delta}_{X_i})_m\right] \hspace{3mm} \Longleftrightarrow \hspace{3mm} \mathbb{P}(\eta_m=1)=\mathbb{P}(\xi_m=1)-\mathbb{P}\left(\sum\limits_{i=1}^d(\underline{\delta}_{X_i})_m=1\right).$$
Recall Lemma \ref{lem: mu^c to mu^c-d} and so by the assumption that $\underline{\xi}\sim\underline{\mu}^c$, it must be that $\underline{\eta}\sim\underline{\mu}^{c+d}$. Hence, 
$$\mathbb{P}(\xi_m>\eta_m)=\mathbb{P}(\xi_m=1)-\mathbb{P}(\eta_m=1)= \frac{q^{c-m}}{1+q^{c-m}}-\frac{q^{c+d-m}}{1+q^{c+d-m}}=\frac{(1-q^d)q^{c-m}}{(1+q^{c-m})(1+q^{c+d-m})}$$
\end{proof}
\begin{remark}
We see that this is equal to $\mathbb{P}(Y=-m)$ for some $Y$ which is distributed according to the discrete logistic distribution with parameters $(q,-c)$ (as defined by Chakraborty and Chakravarty in 2016 \cite{discrete_logistic}).
\end{remark}
\par To answer the second question we will make use of the label process for the second class particles. We then use the product structure of the blocking measure to split the state space up into the piece to the left of where we want the first second class particle and the pieces between the sites for the other second class particles. The following Theorem gives the stationary distribution for the positions of all the $d$ second class particles. 
\scpX*
\vspace{2mm}\begin{remark}
Of course when $d=1$ the two questions we considered are the same and so Theorems \ref{thrm: second class particle at a given site} and \ref{thrm: second class positions dist} are equal:
$$\mathbb{P}(X=m)=\frac{(1-q)q^{c-m}}{(1+q^{c-m})(1+q^{c+1-m})}.$$
When $d>1$ we see that this distribution is product in the $m_j$'s other than the exclusion phenomenon represented by the strict ordering of the $m_j$'s.
\end{remark}

\par \vspace{2mm} In order to prove Theorem \ref{thrm: second class positions dist} we first condition on the labels of the second class particles. By the independence of the label process and $\underline{\xi}$ we have that,
$$\mathbb{P}(\underline{X}=(m_1,..,m_d))=\sum\limits_{\underline{k}\in Z_+^d}\underline{\pi}(\underline{k})\cdot \underline{\mu}^c\left(\left\{\underline{\xi}: \xi_{m_i}=1, \sum\limits_{j=-\infty}^{m_i-1}\xi_j=k_i \text{ for } i \in \{1,..,d\}\right\}\right).$$
Let us focus on the probability under $\underline{\mu}^c$ in the above.
\begin{lem}\label{lem: mu^c(...)}
For $m_1<m_2<...<m_d \in \mathbb{Z}$ and $k_1<k_2<...<k_d \in \mathbb{Z}_{\geq0}$, 
\small $$\underline{\mu}^c\left(\left\{\underline{\xi}: \xi_{m_i}=1, \sum\limits_{j=-\infty}^{m_i-1}\xi_j=k_i \text{ for } i \in \{1,..,d\}\right\}\right) = \frac{q^{(k_1+1)(c-m_1)+\frac{k_1(k_1+1)}{2}}}{(q;q)_{k_1}(-q^{c-m_d};q)_\infty}\prod\limits_{j=2}^d\frac{q^{(\hat{k}_j+1)(c-m_j)+\frac{\hat{k}_j(\hat{k}_j+1)}{2}}(q;q)_{\hat{m}_j}}{(q;q)_{\hat{k}_j}(q;q)_{\hat{m}_j-\hat{k}_j}}$$
\normalsize where $\hat{k}_j:=k_j-k_{j-1}-1$ and $\hat{m}_j:=m_j-m_{j-1}-1$ for $j \in \{2,...,d\}$.
\end{lem}
\begin{proof}
Take some $m_1<m_2<...<m_d \in \mathbb{Z}$ and $k_1<k_2<...<k_d \in \mathbb{Z}_{\geq0}$ and define for $j \in \{2,...,d\}$ $\hat{k}_j:=k_j-k_{j-1}-1$, then 
\begin{align*}
\underline{\mu}^c\Bigg(\Bigg\{\underline{\xi}:& \xi_{m_i}=1, \sum\limits_{j=-\infty}^{m_i-1}\xi_j=k_i \text{ for } i \in \{1,..,d\}\Bigg\}\Bigg)  \\
&=\underline{\mu}^c\left(\left\{ \underline{z}: z_{m_i}=1,\sum\limits_{i=-\infty}^{m_1-1}z_i=k_1, \sum\limits_{i=m_{j-1}+1}^{m_j-1}z_i=\hat{k}_j \text{ for } j \in\{2,..,d\} \right\}\right).
\end{align*}
\normalsize So the event we are interested in (once we have fixed a value $\underline{k}$ for the label process) consists of:
\begin{itemize}
    \item Particles in $\underline{\xi}$ at sites $m_1,...,m_d$.
    \item $k_1$ particles to the left of site $m_1$.
    \item For each $j\in\{2,..d\}$, $\hat{k}_j$ particles on the sites in the range $[m_{j-1}+1,m_{j+1}-1]$.
\end{itemize}
Moreover, by the product structure of $\underline{\mu}^c$ these are independent and so the above is equal to,
\begin{equation}\label{eq: mu^(...) proof}
\prod\limits_{i=1}^d\mu^c_{m_i}(1)\cdot\underline{\mu}^c(\{N^{\overset{\leftarrow}{p}}_{m_1-1+\frac{1}{2}}(\underline{z})=k_1\})\cdot\prod\limits_{j=2}^d\underline{\mu}^c(\{N^{\overset{\leftarrow}{p}}_{m_j-1+\frac{1}{2}}(\underline{z})-N^{\overset{\leftarrow}{p}}_{m_{j-1}+\frac{1}{2}}(\underline{z})=\hat{k}_j\}).
\end{equation}
\par \noindent  By Theorem \ref{thm: half infinite asep particles} and Lemma \ref{lem: finite asep particles} we have,
$$   \text{Equation } (\ref{eq: mu^(...) proof})= \prod\limits_{i=1}^d \frac{q^{c-m_i}}{1+q^{c-m_i}}\cdot \frac{q^{k_1(c-m_1)+\frac{k_1(k_1+1)}{2}}}{(q;q)_{k_1}(-q^{c-m_1+1};q)_\infty}\prod\limits_{j=2}^d\frac{q^{\hat{k}_j(c-m_j)+\frac{\hat{k}_j(\hat{k}_j+1)}{2}}(q;q)_{\hat{m}_j}}{(-q^{c-m_j+1};q)_{\hat{m}_j}(q;q)_{\hat{k}_j}(q;q)_{\hat{m}_j-\hat{k}_j}}.$$
Let us consider, 
\begin{align*}
    (1+q^{c-m_1})&(-q^{c+m_1+1};q)_\infty\prod\limits_{j=2}^d(1+q^{c-m_j})(-q^{c-m_j+1};q)_{\hat{m}_j} \\
    &= (1+q^{c-m_1})\prod\limits_{i=0}^\infty(1+q^{c-m_1+1+i})\cdot \prod\limits_{j=2}^d\left\{(1+q^{c-m_j})\prod\limits_{i=0}^{\hat{m}_j-1}(1+q^{c-m_j+1+i})\right\}\\
    &= \prod\limits_{i=0}^\infty(1+q^{c-m_1+i})\cdot\prod\limits_{j=2}^d\left\{\prod\limits_{i=0}^{\hat{m}_j}(1+q^{c-m_j+i})\right\}=\prod\limits_{i=-m_1}^\infty(1+q^{c+i})\cdot\prod\limits_{j=2}^d\left\{\prod\limits_{i=-m_j}^{-m_{j-1}-1}(1+q^{c+i})\right\}\\&=\prod\limits_{i=-m_d}^\infty(1+q^{c+i})=\prod\limits_{i=0}^\infty(1+q^{c-m_d+i})
    =(-q^{c-m_d};q)_\infty.
\end{align*}
Hence, 
\small $$\underline{\mu}^c\left(\left\{\underline{\xi}: \xi_{m_i}=1, \sum\limits_{j=-\infty}^{m_i-1}\xi_j=k_i \text{ for } i \in \{1,..,d\}\right\}\right) = \frac{q^{(k_1+1)(c-m_1)+\frac{k_1(k_1+1)}{2}}}{(q;q)_{k_1}(-q^{c-m_d};q)_\infty}\prod\limits_{j=2}^d\frac{q^{(\hat{k}_j+1)(c-m_j)+\frac{\hat{k}_j(\hat{k}_j+1)}{2}}(q;q)_{\hat{m}_j}}{(q;q)_{\hat{k}_j}(q;q)_{\hat{m}_j-\hat{k}_j}}.$$
\end{proof}
\normalsize

\par \vspace{2mm} Using Lemma \ref{lem: mu^c(...)}, Euler's Identity and the $q$-Binomial Theorem we prove Theorem \ref{thrm: second class positions dist}, and find the stationary distribution of the positions of the second-class particles.
\begin{proof}[Proof of Theorem \ref{thrm: second class positions dist}]
 \begin{align*}
    &\mathbb{P}(\underline{X}=\underline{m})=\sum\limits_{\underline{k}\in Z_+^d}\underline{\pi}(\underline{k})\cdot \underline{\mu}^c\left(\left\{\underline{\xi}: \xi_{m_i}=1, \sum\limits_{j=-\infty}^{m_i-1}\xi_j=k_i \text{ for } i \in \{1,..,d\}\right\}\right)\\
    &=\sum\limits_{k_1=0}^\infty\sum\limits_{k_2=k_1+1}^{k_1+1+\hat{m}_2}\cdots \sum\limits_{k_d=k_{d-1}+1}^{k_{d-1}+1+\hat{m}_d}\prod\limits_{i=1}^d(1-q^i)q^{\sum\limits_{i=1}^d k_i-\frac{d(d-1)}{2}}\underline{\mu}^c\left(\left\{\underline{\xi}: \xi_{m_i}=1, \sum\limits_{j=-\infty}^{m_i-1}\xi_j=k_i \text{ for } i \in \{1,..,d\}\right\}\right)\\
    &=\sum\limits_{k_1=0}^\infty\sum\limits_{\hat{k}_2=0}^{\hat{m}_2}\cdots \sum\limits_{\hat{k}_d=0}^{\hat{m}_d}\prod\limits_{i=1}^d(1-q^i)q^{dk_1+\sum\limits_{i=2}^d (d+1-i)\hat{k}_i} \frac{q^{(k_1+1)(c-m_1)+\frac{k_1(k_1+1)}{2}}}{(q;q)_{k_1}(-q^{c-m_d};q)_\infty}\prod\limits_{j=2}^d\frac{q^{(\hat{k}_j+1)(c-m_j)+\frac{\hat{k}_j(\hat{k}_j+1)}{2}}(q;q)_{\hat{m}_j}}{(q;q)_{\hat{k}_j}(q;q)_{\hat{m}_j-\hat{k}_j}}\\
    &=\frac{\prod\limits_{i=1}^d(1-q^i)\cdot q^{dc-\sum\limits_{j=1}^dm_j}}{(-q^{c-m_d};q)_\infty}\sum\limits_{k_1=0}^\infty\sum\limits_{\hat{k}_2=0}^{\hat{m}_2}\cdots \sum\limits_{\hat{k}_d=0}^{\hat{m}_d}\frac{q^{k_1(c+d-m_1)+\frac{k_1(k_1+1)}{2}}}{(q;q)_{k_1}}\prod\limits_{j=2}^d\frac{q^{\hat{k}_j(c+d+1-j-m_j)+\frac{\hat{k}_j(\hat{k}_j+1)}{2}}(q;q)_{\hat{m}_j}}{(q;q)_{\hat{k}_j}(q;q)_{\hat{m}_j-\hat{k}_j}}
    \end{align*}
 Euler's identity gives us that, 
    $$\sum\limits_{k_1=0}^\infty \frac{q^{k_1(c+d-m_1)+\frac{k_1(k_1+1)}{2}}}{(q;q)_{k_1}}=(-q^{c+d+1-m_1};q)_\infty.$$
    For each $j \in \{2,...,d\}$ the $q$-Binomial Theorem gives, 
    $$\sum\limits_{\hat{k}_j=0}^{\hat{m}_j}\frac{q^{\hat{k}_j(c+d+1-j-m_j)+\frac{\hat{k}_j(\hat{k}_j+1)}{2}}(q;q)_{\hat{m}_j}}{(q;q)_{\hat{k}_j}(q;q)_{\hat{m}_j-\hat{k}_j}}=(-q^{c+d+2-j-m_j};q)_{\hat{m}_j}.$$
    Thus,
    
$$\mathbb{P}(\underline{X}=\underline{m})=\frac{\prod\limits_{i=1}^d(1-q^i)\cdot q^{dc-\sum\limits_{j=1}^dm_j}(-q^{c+d+1-m_1};q)_\infty\prod\limits_{j=2}^d(-q^{c+d+2-j-m_j};q)_{\hat{m}_j}}{(-q^{c-m_d};q)_\infty}.$$
 \vspace{5mm}
 By iteratively applying Lemma \ref{lem: m_j pochammers} we have, 
    \begin{align*}
    \mathbb{P}(\underline{X}=\underline{m})&=\frac{\prod\limits_{i=1}^d(1-q^i)\cdot q^{dc-\sum\limits_{j=1}^dm_j}(-q^{c+d+1-m_1};q)_\infty\prod\limits_{j=2}^d(-q^{c+d+2-j-m_j};q)_{\hat{m}_j}}{(-q^{c-m_d};q)_\infty}\\
    &=\frac{\prod\limits_{i=1}^d(1-q^i)\cdot q^{dc-\sum\limits_{j=1}^dm_j}(-q^{c+d+1-m_1};q)_\infty\prod\limits_{j=2}^{d-1}(-q^{c+d+2-j-m_j};q)_{\hat{m}_j}}{(1+q^{c-m_d})(1+q^{c+1-m_d})(-q^{c+1-m_{d-1}};q)_\infty}\\
    &=\frac{\prod\limits_{i=1}^d(1-q^i)\cdot q^{dc-\sum\limits_{j=1}^dm_j}(-q^{c+d+1-m_1};q)_\infty\prod\limits_{j=2}^{d-2}(-q^{c+d+2-j-m_j};q)_{\hat{m}_j}}{(1+q^{c-m_d})(1+q^{c+1-m_d})(1+q^{c+1-m_{d-1}})(1+q^{c+2-m_{d-1}})(-q^{c+2-m_{d-2}};q)_\infty}\\
    &\vdots\\
    &=\frac{\prod\limits_{i=1}^d(1-q^i)\cdot q^{dc-\sum\limits_{j=1}^dm_j}(-q^{c+d+1-m_1};q)_\infty}{\prod\limits_{j=2}^d(1+q^{c+d-j-m_j})(1+q^{c+d+1-j-m_j})(q^{c+d-1-m_1};q)_\infty}.
\end{align*}
Lastly we see that, 
$$\frac{(-q^{c+d+1-m_1};q)_\infty}{(q^{c+d-1-m_1};q)_\infty}=\frac{1}{(1+q^{c+d-1-m_1})(1+q^{c+d-m_1})}$$
and thus,
$$\mathbb{P}(\underline{X}=\underline{m})=\frac{\prod\limits_{i=1}^d(1-q^i)\cdot q^{dc-\sum\limits_{j=1}^dm_j}}{\prod\limits_{j=1}^d(1+q^{c+d-j-m_j})(1+q^{c+d+1-j-m_j})}.$$
\end{proof}
\par Thus, we have found the distribution under the blocking measure for the positions of any given $d$ second class particles in ASEP.

\subsection{Second Class Particles Positions Distribution under $\underline{\nu}^n$} \label{subsect: asep positions under nu}~
\par We will now consider the same questions under the ergodic measure. To recall the questions we consider are as follows.
\begin{enumerate}
    \item For a given site what is the probability (under the ergodic measure $\underline{\nu}^n$) that there is a second class particle there? 
    \item For a given $d$ distinct sites what is the probability (under the ergodic measure $\underline{\nu}^n$) that the $d$ second class particles lie on these exact sites?
\end{enumerate}
Just as in Section \ref{subsect: asep positions}, the first question can be easily computed by looking at expectations under the ergodic measures for $\underline{\eta}$ and $\underline{\xi}$. 

\nuscpm*

\vspace{3mm}\begin{proof}
By the coupling we have that, $\eta_m=\xi_m-\sum\limits_{i=1}^d(\underline{\delta}_{X_i})_m$. We see that the event $\{\xi_m>\eta_m\}$ is exactly the event that $\{\sum\limits_{i=1}^d(\underline{\delta}_{X_i})_m>0\}$ or equivalently that $\{\sum\limits_{i=1}^d(\underline{\delta}_{X_i})_m=1\}$ since we are in the simple exclusion setting. Recall Lemma \ref{lem: nu^n to nu^n+d} and so by the assumption that $\underline{\xi}\sim\underline{\nu}^n$, it must be that $\underline{\eta}\sim\underline{\nu}^{n+d}$. Now if we take expectations over the coupling equation we have that, 
$$\mathbb{E}[\eta_m]=\mathbb{E}[\xi_m]-\mathbb{E}\left[\sum\limits_{i=1}^d(\underline{\delta}_{X_i})_m\right] \hspace{3mm} \Longleftrightarrow \hspace{3mm} \mathbb{P}_{\underline{\nu}^{n+d}}(\eta_m=1)=\mathbb{P}_{\underline{\nu}^n}(\xi_m=1)-\mathbb{P}_{\underline{\nu}^n}\left(\sum\limits_{i=1}^d(\underline{\delta}_{X_i})_m=1\right).$$
Hence by Lemma \ref{lem: site under nu n} we have that, 
    \begin{align*}
        \mathbb{P}_{\underline{\nu}^n}(\xi_m>\eta_m)&=\mathbb{P}_{\underline{\nu}^n}(\xi_m=1)-\mathbb{P}_{\underline{\nu}^{n+d}}(\eta_m=1)\\
        &=(q;q)_\infty\left\{\sum\limits_{k=(m-n-1)^+}^\infty\frac{q^{(k+1)(n-m+k+1)}}{(q;q)_k(q;q)_{n-m+k+1}}-\sum\limits_{k=(m-n-d-1)^+}^\infty\frac{q^{(k+1)(n+d-m+k+1)}}{(q;q)_k(q;q)_{n+d-m+k+1}}\right\}.
    \end{align*}
\end{proof}
As before the second question we ask is what is the joint distribution of the positions of all $d$ second class particles. We will first study the case when $d=1$, of course this is equivalent to taking $d=1$ in Proposition \ref{prop: second class particle at a given site nu^n}. However by conditioning on the label of the second class particle we find an alternative form for this distribution. 

\nuscpmother*

\begin{proof}
    We begin by conditioning on the label of the second class particle. By the independence of the label process and $\underline{\xi}$ we have that, 
    \begin{align*}
        \mathbb{P}&_{\underline{\nu}^n}(X=m)=\sum\limits_{k= 0}^\infty\underline{\pi}(k)\cdot\underline{\nu}^n\left(\left\{\underline{\xi}:\xi_m=1, \sum\limits_{i=-\infty}^{m-1}\xi_i=k\right\}\right)\\
        &=\sum\limits_{k=(m-n-1)^+}^\infty \underline{\pi}(k)\cdot \frac{\underline{\mu}^c\left(\left\{\underline{\xi}:\xi_m=1, \sum\limits_{i=-\infty}^{m-1}\xi_i=k\right\}\cap\{N=n\}\right)}{\underline{\mu}^c(\{N=n\})}\\
        &=\sum\limits_{k=(m-n-1)^+}^\infty \underline{\pi}(k)\frac{\underline{\mu}^c(\{N^{\overset{\leftarrow}{p}}_{m-1+\frac{1}{2}}(\underline{z})=k\}\cap\{\xi_m=1\}\cap\{N^{\overset{\rightarrow}{h}}_{m+\frac{1}{2}}(\underline{z})=n-m+k+1\})}{\underline{\mu}^c(\{N=n\})}\\
        &=\sum\limits_{k=(m-n-1)^+}\underline{\pi}(k)\frac{\underline{\mu}^c(\{N^{\overset{\leftarrow}{p}}_{m-1+\frac{1}{2}}(\underline{z})=k\})\cdot\mu^c_m(1)\cdot\underline{\mu}^{2m+1-c}(\{N^{\overset{\leftarrow}{p}}_{m+\frac{1}{2}}(\underline{z})=n-m+k+1\})}{\underline{\mu}^c(\{N=n\})}\\
        &=\sum\limits_{k=(m-n-1)^+}^\infty \frac{(1-q)q^k\sum\limits_{\ell\in\mathbb{Z}}q^{\frac{\ell(\ell+1)}{2}-\ell c}}{q^{\frac{n(n+1)}{2}-nc}}\cdot \frac{q^{k(c+1-m)+\frac{k(k-1)}{2}}}{(q;q)_k(-q^{c+1-m};q)_\infty}\cdot \frac{q^{c-m}}{(1+q^{c-m})}\\
        &\hspace{30mm}\cdot\frac{q^{(n-m+k+1)(m+1-c)+\frac{(n-m+k+1)(n-m+k)}{2}}}{(q;q)_{n-m+k+1}(-q^{m+1-c};q)_\infty}.\\
    \end{align*}
    Then by Lemma \ref{lem: pochammer relation for nu^n calcs} and the Jacobi triple product identity, \eqref{eq: JTP}, we have that, 
    $$\mathbb{P}_{\underline{\nu}^n}(X=m)=(1-q)(q;q)_\infty \sum\limits_{k=(m-n-1)^+}^\infty \frac{q^{k(n-m+k+1)+k+(n-m+k+1)}}{(q;q)_k(q;q)_{n-m+k+1}}.$$
\end{proof}

Under the ergodic measure $\underline{\nu}^n$ (for any $n\in\mathbb{Z}$) there is a unique ground state, $\underline{\eta}^n\in\Omega^n$ given by, 
$$\eta^n_i=\begin{cases}
        1 &\text{ if }i>n\\
        0 &\text{ if }i\leq n
    \end{cases}.$$ 
If the ASEP with a single second class particle is distributed according to the ergodic measure (i.e.\ $\underline{\xi}\sim\underline{\nu}^n$) then in the ground state $\underline{\eta}^n$ the second class particle is at site $n+1$. It is therefore natural to consider the displacement of the second class particle from its (most probable) position in the ground state. We will call this displacement $D$ and we notice that, $D=X-(n+1)$. Using Proposition \ref{prop: position of single 2cp under nu} we give the distribution of the displacement of a single second class particle in ASEP. 

\begin{cor}\label{cor: 1 displacement}
   Consider the coupling $(\underline{\eta},\underline{\xi})$ when $d=1$ (i.e.\ a single second class particle in ASEP). Assume the coupled system is stationary with $\underline{\xi}\sim\underline{\nu}^n$. Then for any $\ell \in\mathbb{Z}$ the probability that the displacement of the second class particle is $\ell$ is, 
   $$\mathbb{P}_{\underline{\nu}^n}(D=\ell)=(1-q)(q;q)_\infty\sum\limits_{k=(\ell)^+}^\infty \frac{q^{k(k-\ell)+k+(k-\ell)}}{(q;q)_k(q;q)_{k-\ell}}.$$
\end{cor}

\begin{proof}
    By definition, $D=X-(n+1)$ and so, 
    $$\mathbb{P}_{\underline{\nu}^n}(D=\ell)=\mathbb{P}_{\underline{\nu}^n}(X=\ell+n+1)=(1-q)(q;q)_\infty\sum\limits_{k=(\ell)^+}^\infty \frac{q^{k(k-\ell)+k+(k-\ell)}}{(q;q)_k(q;q)_{k-\ell}}.$$
\end{proof}

\begin{remark}\label{rem: displacement}
    In 2012, Gnedin and Olshanski, \cite{gnedin_olshanski}, defined the two-sided infinite Mallows model for random permutations of the integers. We will denote the set of all permutations of $\mathbb{Z}$ (i.e. bijections $\sigma:\mathbb{Z}\rightarrow\mathbb{Z}$) by $\mathfrak{S}$. Here we will see how the above result links to these random permutations. First we give a short introduction to Mallows measures (for full details see \cite{gnedin_olshanski}).
    \par The two sided infinite Mallows measure is a probability distribution on permutations of the integers that concentrates on,
    $$\mathfrak{S}^{\text{bal}}\defeq\left\{\sigma\in\mathfrak{S}: |\{i\in\mathbb{Z}_{\leq 0}: \sigma(i)\in\mathbb{Z}_{>0}\}|<\infty \text{ and, } |\{j\in\mathbb{Z}_{>0}: \sigma(j)\in\mathbb{Z}_{\leq 0}\}|<\infty \right\}.$$
    That is, far to the left ($i\ll0$) in the permutation we only see negative numbers and far to right ($i\gg0$) we only see positive numbers. Note that this is similar to the blocking states in ASEP, where far to the left there are only holes and far to the right only particles. \par As described in [\cite{gnedin_olshanski}, Equation (3)], such a permutation, $\sigma$, of $\mathbb{Z}$ can be given by permutations, $\sigma^+$ and $\sigma^-$, of $\mathbb{Z}_{>0}$ and $\mathbb{Z}_{\leq 0}$ respectively, and an interlacing pattern, given by an integer partition $\lambda$, which encodes how to fit the two permutations together.  For a given $\sigma\in\mathfrak{S}^{\text{bal}}$ consider the associated permutation word,
    $$w=(\cdots w_{-1}w_0w_1w_2\cdots)=(\cdots \sigma(-1)\sigma(0)\sigma(1)\sigma(2)\cdots).$$
    \par \vspace{2mm}\noindent From this we derive a binary word, 
    $$\varepsilon=(\cdots \varepsilon_{-1}\varepsilon_0\varepsilon_1\varepsilon_2\cdots)\in\{-1,+1\}^{\mathbb{Z}},$$ 
    where,

    $$\varepsilon_i=\begin{cases} +1 &\text{if } \sigma(i)\in\mathbb{Z}_{>0}\\
    -1 &\text{if } \sigma(i)\in\mathbb{Z}_{\leq 0}.\end{cases} $$
    The binary word $\varepsilon$ can be encoded into an integer partition $\lambda$ by splitting $\mathbb{Z}$ into the disjoint union of the set of positions where $\varepsilon_i=-1$ and the set of positions where $\varepsilon=1$. That is, 
    $$\mathbb{Z}=\sigma^{-1}(\mathbb{Z}_{\leq 0})\cup\sigma^{-1}(\mathbb{Z}_{>0})=\{\cdots<j_{-2}<j_{-1}<j_0\}\cup\{i_1<i_2<i_3<\cdots\}.$$
    \par\noindent Then $\lambda$ is given by, 
    $$\lambda_\ell=\ell-i_\ell \quad \forall \ell\geq 0.$$
    \par\vspace{2mm}\noindent Or equivalently the conjugate partition, $\lambda'$, is given by,
    $$\lambda'_\ell=j_{-\ell+1}+\ell-1 \quad \forall \ell\geq 0.$$
    \par \vspace{2mm}\noindent With this we can define the permutations $\sigma^+$ and $\sigma^-$ of $\mathbb{Z}_{>0}$ and $\mathbb{Z}_{\leq 0}$ respectively. We define this via the associated permutation words, 
    $$w^+\defeq (\sigma^+(1)\sigma^+(2)\sigma^+(3)\cdots)=(w_{i_1}w_{i_2}w_{i_3}\cdots),$$
    and,
    $$w^-\defeq (\cdots\sigma^-(-2)\sigma^-(-1)\sigma^-(0))=(\cdots w_{j_{-2}}w_{j_{-1}}w_{j_0}).$$
    This gives the triple $(w^+,w^-,\varepsilon)$, or equivalently $(w^+,w^-,\lambda)$; $w$ can be recovered from the this by replacing each $+1$ entry of $\varepsilon$ from left-to-right by the entries of $w^+$ and similarly the $-1$ entries of $\varepsilon$ from right-to-left with the entries of $w^-$.
    \par Then, as given in [\cite{gnedin_olshanski}, Equation (4)], the two sided infinite Mallows measure of parameter $q\in(0,1)$, $\mathcal{Q}$, is defined as a product measure,
    $$\mathcal{Q}\defeq \mathcal{Q}^+\otimes \mathcal{Q}^-\otimes\mathcal{P},$$
    where $\mathcal{Q}^+$ is the Mallows measure on permutations of $\mathbb{Z}_{>0}$, similarly $\mathcal{Q}^-$ is the Mallows measure on permutations of $\mathbb{Z}_{\leq0}$ and $\mathcal{P}$ is the natural probability distribution on integer partitions given by, 
    $$\mathcal{P}(\lambda)=\frac{q^{|\lambda|}}{(q;q)_\infty} \quad \text{ for any integer partition } \lambda.$$
    For a detailed construction of the Mallows measure on $\mathbb{Z}_{>0}$ see Gnedin and Olshanski \cite{gnedin_olshanski_2010} (in particular $\mathcal{Q}^+$ is defined in Definition 4.4). First consider the Mallows measure on permutations of $\{1,\cdots,n\}$ for any $n\in\mathbb{Z}_{>0}$,
    $$\mathcal{Q}_n(\sigma)=\frac{q^{\text{inv}(\sigma)}(q;q)_n}{(1-q)^n}.$$
    Then as $n\rightarrow\infty$, $\mathcal{Q}_n$ weakly converge to $\mathcal{Q}^+$ [\cite{gnedin_olshanski_2010}, Proposition A.1.]. The Mallows measure, $\mathcal{Q}^{-}$, on $\mathbb{Z}_{\leq 0}$ can be similarly defined. 
    \par Consider a permutation, $\sigma\in\mathfrak{S}$. For any integer, $j$, its displacement is defined to be $D_j=\sigma(j)-j$. We see that Corollary \ref{cor: 1 displacement} gives exactly the distribution of displacement for any integer under the two-sided infinite Mallows measure with parameter $q$ as given in [\cite{gnedin_olshanski}, Theorem 5.1],
    $$\mathbb{P}_{\mathcal{Q}}(D=\ell)=(1-q)(q;q)_\infty\sum\limits_{\substack{r,s\geq 0:\\ r-s=\ell}}\frac{q^{rs+r+s}}{(q;q)_r(q;q)_s} \quad \forall \ell \in\mathbb{Z}.$$
    This is clear since the connection between ASEP under blocking measures and permutations of the integers distributed according to the Mallows measure has been well studied, for example by Bufetov and Nejjar \cite{bufetov_nejjar}. Indeed, if we let any $i\in\mathbb{Z}_{>n+1}$ be equivalent to a first class particle, any $j\in\mathbb{Z}_{<n+1}$ be equivalent to a hole and $n+1$ represent a single second class particle then, any Mallows distributed permutation of $\mathbb{Z}$ gives an ASEP state with a single second class particle distributed according to the ergodic measure $\underline{\nu}^n$.
\end{remark}

\par \vspace{3mm} Now we consider the joint distribution of any given $d>0$ second class particles in ASEP under the ergodic measure $\underline{\nu}^n$ for any $n\in\mathbb{Z}$.
\dundernu*

\begin{remark}
    The proof follows the same reasoning as used in the proofs of Theorem \ref{thrm: second class positions dist} and Proposition \ref{prop: position of single 2cp under nu}.
\end{remark}

\begin{proof}
    We first condition on the labels of the $d$ second class particles. Since the label process and the distribution of $\underline{\xi}$ are independent we have that, 
    \begin{align}
    \mathbb{P}_{\underline{\nu}^n}(\underline{X}=\underline{m})&=\sum\limits_{\underline{k}\in Z_+^d}\underline{\pi}(\underline{k})\cdot\underline{\nu}^n\Bigg(\bigg\{\underline{\xi}:\xi_{m_i}=1,\sum\limits_{j=-\infty}^{m_i-1}\xi_j=k_i \hspace{1mm} \forall i\in\{1,\cdots,d\}\bigg\}\Bigg)\notag\\
    &=\sum\limits_{\underline{k}\in Z_+^d}\underline{\pi}(\underline{k})\cdot\underline{\nu}^n\Bigg(\bigg\{\underline{\xi}:\xi_{m_i}=1 \hspace{1mm} \forall i\in\{1,\cdots,d\},\sum\limits_{i=-\infty}^{m_1-1}\xi_i=k_1,\sum\limits_{i=m_{j-1}+1}^{m_j-1}\xi_i=\hat{k}_j \hspace{1mm} \forall j\in\{2,\cdots,d\}\bigg\}\Bigg), \label{eq: second class under nu n}
\end{align}
where $\hat{k}_j\defeq k_j-k_{j-1}-1$. We note that for some values of $\underline{k}\in Z_+^d$ the probability of the quantity under $\underline{\nu}^n$ above will be 0 (we will make this more precise below).
\par Let's concentrate on the quantity under $\underline{\nu}^n$. For any $c\in\mathbb{R}$ we have that,
\begin{align}
    \underline{\nu}^n\Bigg(\bigg\{\underline{\xi}:&\xi_{m_i}=1 \hspace{1mm} \forall i\in\{1,\cdots,d\},\sum\limits_{i=-\infty}^{m_1-1}\xi_i=k_1,\sum\limits_{i=m_{j-1}+1}^{m_j-1}\xi_i=\hat{k}_j \hspace{1mm} \forall j\in\{2,\cdots,d\}\bigg\}\Bigg)\label{eq: nu stuff for d 2cp}\\
    &=\underline{\mu}^c\Bigg(\bigg\{\underline{\xi}:\xi_{m_i}=1 \hspace{1mm} \forall i\in\{1,\cdots,d\},\sum\limits_{i=-\infty}^{m_1-1}\xi_i=k_1,\sum\limits_{i=m_{j-1}+1}^{m_j-1}\xi_i=\hat{k}_j \hspace{1mm} \forall j\in\{2,\cdots,d\}\bigg\}\bigg|N(\underline{\xi})=n\Bigg)\notag\\
    \notag\\
    &=\frac{\underline{\mu}^c\Bigg(\bigg\{\underline{\xi}:\xi_{m_i}=1 \hspace{1mm} \forall i\in\{1,\cdots,d\},\sum\limits_{i=-\infty}^{m_1-1}\xi_i=k_1,\sum\limits_{i=m_{j-1}+1}^{m_j-1}\xi_i=\hat{k}_j \hspace{1mm} \forall j\in\{2,\cdots,d\}\bigg\}\bigcap\{N(\underline{\xi})=n\}\Bigg)}{\underline{\mu}^c(\{N(\underline{z})=n\})}.\notag
\end{align}
As we have already seen, for any $n,m\in\mathbb{Z}$,
$$\{N(\underline{z})=n\}\Longleftrightarrow\{N_{m+\frac{1}{2}}(\underline{z})=n-m\}.$$
We also note that, 
$$\bigg\{\underline{\xi}:\xi_{m_i}=1 \hspace{1mm} \forall i\in\{1,\cdots,d\},\sum\limits_{i=-\infty}^{m_1-1}\xi_i=k_1,\sum\limits_{i=m_{j-1}+1}^{m_j-1}\xi_i=\hat{k}_j \hspace{1mm} \forall j\in\{2,\cdots,d\}\bigg\}\Rightarrow\{N^{\overset{\leftarrow}{p}}_{m_d+\frac{1}{2}}(\underline{\xi})=k_d+1\}.$$
By definition, $$N_{m_d+\frac{1}{2}}(\underline{z})=N^{\overset{\rightarrow}{h}}_{m_d+\frac{1}{2}}(\underline{z})-N^{\overset{\leftarrow}{p}}_{m_d+\frac{1}{2}}(\underline{z}),$$
and so for any $n,m\in\mathbb{Z}$ and $k\in\mathbb{Z}_{\geq 0}$,
$$\{N^{\overset{\leftarrow}{p}}_{m+\frac{1}{2}}(\underline{z})=k\}\cap\{N_{m+\frac{1}{2}}(\underline{z})=n-m\}\Longleftrightarrow\{N^{\overset{\leftarrow}{p}}_{m+\frac{1}{2}}(\underline{z})=k\}\cap\{N^{\overset{\rightarrow}{h}}_{m+\frac{1}{2}}(\underline{z})=n-m+k\}.$$
Putting this together we see that, 
\begin{align*}
    &\bigg\{\underline{\xi}:\xi_{m_i}=1 \hspace{1mm} \forall i\in\{1,\cdots,d\},\sum\limits_{i=-\infty}^{m_1-1}\xi_i=k_1,\sum\limits_{i=m_{j-1}+1}^{m_j-1}\xi_i=\hat{k}_j \hspace{1mm} \forall j\in\{2,\cdots,d\}\bigg\}\bigcap\{N(\underline{z})=n\}\\
    &\Longleftrightarrow\quad \bigg\{\underline{\xi}:\xi_{m_i}=1 \hspace{1mm} \forall i\in\{1,\cdots,d\},\sum\limits_{i=-\infty}^{m_1-1}\xi_i=k_1,\sum\limits_{i=m_{j-1}+1}^{m_j-1}\xi_i=\hat{k}_j \hspace{1mm} \forall j\in\{2,\cdots,d\}\bigg\} \\
    &\quad\quad\quad\quad\quad\quad\quad\quad\bigcap\{N^{\overset{\rightarrow}{h}}_{m_d+\frac{1}{2}}(\underline{z})=n-m_d+k_d+1\}.
\end{align*}
Using the product structure of $\underline{\mu}^c$ and Corollary \ref{cor: c change between holes and particles in ASEP} we find that,
\begin{align*}
    \text{Equation} \eqref{eq: nu stuff for d 2cp}
    &=\frac{1}{\underline{\mu}^c(\{N(\underline{z})=n\})}\cdot \prod\limits_{i=1}^d\mu^c_{m_i}(1)\cdot\prod\limits_{j=2}^d\underline{\mu}^c(\{N^{\overset{\leftarrow}{p}}_{m_j-1+\frac{1}{2}}(\underline{z})-N^{\overset{\leftarrow}{p}}_{m_{j-1}+\frac{1}{2}}(\underline{z})=\hat{k}_j\}) \\
    &\hspace{20mm}\cdot\underline{\mu}^c(\{N^{\overset{\leftarrow}{p}}_{m_1-1+\frac{1}{2}}(\underline{z})=k_1\})\cdot\underline{\mu}^c(\{N^{\overset{\rightarrow}{h}}_{m_d+\frac{1}{2}}(\underline{z})=n-m_d+k_d+1\})\\
    &=\frac{1}{\underline{\mu}^c(\{N(\underline{z})=n\})}\cdot \prod\limits_{i=1}^d\mu^c_{m_i}(1)\cdot\prod\limits_{j=2}^d\underline{\mu}^c(\{N^{\overset{\leftarrow}{p}}_{m_j-1+\frac{1}{2}}(\underline{z})-N^{\overset{\leftarrow}{p}}_{m_{j-1}+\frac{1}{2}}(\underline{z})=\hat{k}_j\}) \\
    &\hspace{20mm}\cdot\underline{\mu}^c(\{N^{\overset{\leftarrow}{p}}_{m_1-1+\frac{1}{2}}(\underline{z})=k_1\})\cdot\underline{\mu}^{2m_d+1-c}(\{N^{\overset{\leftarrow}{p}}_{m_d+\frac{1}{2}}(\underline{z})=n-m_d+k_d+1\}).
\end{align*}
\par By Lemma \ref{lem: mu^c(...)} that is, for permissible configurations, 
\begin{align*}
     \text{Equation }\eqref{eq: nu stuff for d 2cp}&= \frac{1}{\underline{\mu}^c(\{N(\underline{z})=n\})}\cdot\underline{\mu}^{2m_d+1-c}(\{N^{\overset{\leftarrow}{p}}_{m_d+\frac{1}{2}}(\underline{z})=n-m_d+k_d+1\})\\
     &\hspace{20mm}\cdot \frac{q^{(k_1+1)(c-m_1)+\frac{k_1(k_1+1)}{2}}}{(q;q)_{k_1}(-q^{c-m_d};q)_\infty}\prod\limits_{j=2}^d\frac{q^{(\hat{k}_j+1)(c-m_j)+\frac{\hat{k}_j(\hat{k}_j+1)}{2}}(q;q)_{\hat{m}_j}}{(q;q)_{\hat{k}_j}(q;q)_{\hat{m}_j-\hat{k}_j}}\\
     &=\frac{\sum\limits_{\ell=-\infty}^\infty q^{\frac{\ell(\ell+1)}{2}-\ell c}}{q^{\frac{n(n+1)}{2}-nc}}\cdot \frac{q^{(n-m_d+k_d+1)(m_d+1-c)+\frac{(n-m_d+k_d)(n-m_d+k_d+1)}{2}}}{(q;q)_{n-m_d+k_d+1}(-q^{m_d+1-c};q)_\infty}\\
    &\hspace{20mm}\cdot \frac{q^{(k_1+1)(c-m_1)+\frac{k_1(k_1+1)}{2}}}{(q;q)_{k_1}(-q^{c-m_d};q)_\infty}\prod\limits_{j=2}^d\frac{q^{(\hat{k}_j+1)(c-m_j)+\frac{\hat{k}_j(\hat{k}_j+1)}{2}}(q;q)_{\hat{m}_j}}{(q;q)_{\hat{k}_j}(q;q)_{\hat{m}_j-\hat{k}_j}}\\
    &= \sum\limits_{\ell=-\infty}^\infty q^{\frac{\ell(\ell+1)}{2}-\ell c}\frac{q^{m_dc-\frac{m_d(m_d+1)}{2}+\frac{k_d(k_d+1)}{2}+(k_d+1)(n+1)}}{(q;q)_{n-m_d+k_d+1}(-q^{m_d+1-c};q)_\infty} \\
    &\hspace{20mm}\cdot \frac{q^{-m_1(k_1+1)+\frac{k_1(k_1+1)}{2}}}{(q;q)_{k_1}(-q^{c-m_d};q)_\infty}\prod\limits_{j=2}^d\frac{q^{-m_j(\hat{k}_j+1)+\frac{\hat{k}_j(\hat{k}_j+1)}{2}}(q;q)_{\hat{m}_j}}{(q;q)_{\hat{k}_j}(q;q)_{\hat{m}_j-\hat{k}_j}},
\end{align*}
where $\hat{m}_j\defeq m_j-m_{j-1}-1$ for each $j\in\{2,3,\cdots ,d\}$.
\par By Lemma \ref{lem: pochammer relation for nu^n calcs} and the Jacobi triple product identity, equation \eqref{eq: JTP}, we have that,
\begin{align*}
     \underline{\nu}^n\Bigg(\bigg\{\underline{\xi}&:\xi_{m_i}=1 \hspace{1mm} \forall i\in\{1,\cdots,d\},\sum\limits_{i=-\infty}^{m_1-1}\xi_i=k_1,\sum\limits_{i=m_{j-1}+1}^{m_j-1}\xi_i=\hat{k}_j \hspace{1mm} \forall j\in\{2,\cdots,d\}\bigg\}\Bigg)\\
     &=\frac{q^{\frac{k_d(k_d+1)}{2}+(k_d+1)(n+1)-m_1(k_1+1)+\frac{k_1(k_1+1)}{2}}(q;q)_\infty}{(q;q)_{n-m_d+k_d+1}(q;q)_{k_1}}\prod\limits_{j=2}^d\frac{q^{-m_j(\hat{k}_j+1)+\frac{\hat{k}_j(\hat{k}_j+1)}{2}}(q;q)_{\hat{m}_j}}{(q;q)_{\hat{k}_j}(q;q)_{\hat{m}_j-\hat{k}_j}}.
\end{align*}
\par Now let's return to the sum over possible label configurations of the second class particles as in \eqref{eq: second class under nu n}. Apriori we have that $\underline{k}\in Z_+^d$ that is, 
\begin{equation}\label{eq: ks are strict increasing}
0\leq k_1<k_2<\cdots <k_d \quad \Longrightarrow \quad k_1\geq 0,\hspace{1mm} k_2\geq1,\hspace{1mm}\cdots,\hspace{1mm} k_d\geq d-1.
\end{equation}

In order for the calculations above for the quantity under $\underline{\nu}^n$ to hold we have to consider extra conditions on $\underline{k}$. In particular, any state, $\underline{\xi}$, must be such that $N(\underline{\xi})=n$, and so for each $j\in\{1,2,\cdots,d\}$ $N_{m_j+\frac{1}{2}}(\underline{\xi})=n-m_j$. The states we are considering are such that $N^{\overset{\leftarrow}{p}}_{m_j+\frac{1}{2}}(\underline{\xi})=k_j+1$ for each $j\in\{1,2,\cdots,d\}$. Thus by the definition of $N_{m_j+\frac{1}{2}}$ these states are also such that,
$$N^{\overset{\rightarrow}{h}}_{m_j+\frac{1}{2}}(\underline{\xi})=n-m_j+k_j+1 \quad\quad \quad \forall j\in\{1,2,\cdots,d\}.$$
Since the number of holes to the right of a site must be non-negative this gives that $\underline{k}$ must satisfy,
\begin{equation}\label{eq: k bound from holes begin positive}
    k_j\geq m_j-n-1 \quad \quad \quad \forall j\in\{1,2,\cdots,d\}.
\end{equation}
Putting \eqref{eq: ks are strict increasing} and \eqref{eq: k bound from holes begin positive} together we find that any possible $\underline{k}$ is such that,
$$0\leq k_1<k_2<\cdots<k_d, \text{ and } k_j\geq m_j-n-1 \quad \quad\quad \forall j\in\{1,2,\cdots,d\},$$
or equivalently, 
$$k_j\geq \max\{m_j-n-1,j-1\}=(m_j-n-j)^++(j-1)\quad\quad\quad \forall j\in\{1,2,\cdots,d\}.$$
\par Thus we have that, 
   \small \begin{align*}
    &\mathbb{P}_{\underline{\nu}^n}(\underline{X}=\underline{m})\\
    &=\sum\limits_{\underline{k}\in Z_+^d}\underline{\pi}(\underline{k})\cdot\underline{\nu}^n\Bigg(\bigg\{\underline{\xi}:\xi_{m_i}=1,\sum\limits_{j=-\infty}^{m_i-1}\xi_j=k_i \hspace{1mm} \forall i\in\{1,\cdots,d\}\bigg\}\Bigg)\\
    &=\sum\limits_{\substack{\underline{k}\in Z_+^d:\\\forall j\in\{1,2,\cdots,d\},\\ k_j\geq m_j-n-1}} q^{\sum\limits_{i=1}^d k_i-\frac{d(d-1)}{2}}\prod\limits_{i=1}^d\left(1-q^i\right)\frac{q^{\frac{k_d(k_d+1)}{2}+(k_d+1)(n+1)-m_1(k_1+1)+\frac{k_1(k_1+1)}{2}}(q;q)_\infty}{(q;q)_{n-m_d+k_d+1}(q;q)_{k_1}}\prod\limits_{j=2}^d\frac{q^{-m_j(\hat{k}_j+1)+\frac{\hat{k}_j(\hat{k}_j+1)}{2}}(q;q)_{\hat{m}_j}}{(q;q)_{\hat{k}_j}(q;q)_{\hat{m}_j-\hat{k}_j}}\\
    &=q^{-\sum\limits_{i=1}^dm_i}(q;q)_\infty\prod\limits_{i=1}^d(1-q^i)\sum\limits_{\substack{\underline{k}\in Z_+^d:\\\forall j\in\{1,2,\cdots,d\},\\ k_j\geq m_j-n-1}} q^{dk_1+\sum\limits_{j=2}^d(d+1-j)\hat{k}_j}\frac{q^{\frac{k_d(k_d+1)}{2}+(k_d+1)(n+1)-m_1k_1+\frac{k_1(k_1+1)}{2}}}{(q;q)_{n-m_d+k_d+1}(q;q)_{k_1}}\prod\limits_{j=2}^d\frac{q^{-m_j\hat{k}_j+\frac{\hat{k}_j(\hat{k}_j+1)}{2}}(q;q)_{\hat{m}_j}}{(q;q)_{\hat{k}_j}(q;q)_{\hat{m}_j-\hat{k}_j}}\\
    &=q^{-\sum\limits_{i=1}^dm_i}(q;q)_\infty\prod\limits_{i=1}^d(1-q^i)\sum\limits_{\substack{\underline{k}\in Z_+^d:\\\forall j\in\{1,2,\cdots,d\},\\ k_j\geq m_j-n-1}} \frac{q^{\frac{k_1(k_1+1)}{2}+k_1(d-m_1)+\frac{k_d(k_d+1)}{2}+(k_d+1)(n+1)}}{(q;q)_{n-m_d+k_d+1}(q;q)_{k_1}}\prod\limits_{j=2}^dq^{\hat{k}_j(d+1-j-m_j)+\frac{\hat{k}_j(\hat{k}_j+1)}{2}}\cdot \begin{bmatrix}
        \hat{m}_j\\
        \hat{k}_j
    \end{bmatrix}_q.
\end{align*}
\normalsize
\end{proof}
\begin{remark}
    It seems we cannot simplify the sum in Proposition \ref{prop: d 2cps under nu} further. For example let's consider the case where $m_d-n-1<0$, and therefore $m_j-n-1<0$ for any $j\in\{1,2,\cdots,d\}$. In this case the condition that $k_j\geq m_j-n-1$ for all $j$ is automatically satisfied and the summation is simply over any $\underline{k}\in Z_+^d$. So as we did in the proof of Theorem \ref{thrm: second class positions dist} we can rewrite the summation over $\underline{k}$ as the summation, 
    $$\sum\limits_{{k_1}=0}^\infty\sum\limits_{\hat{k}_2=0}^{\hat{m}_2}\cdots \sum\limits_{\hat{k}_d=0}^{\hat{m}_d}.$$
    In the proof of Theorem \ref{thrm: second class positions dist} we could proceed to simplify the sum by using identities such as Euler and the $q$-Binomial theorem. 
    \par However for the distribution given in Proposition \ref{prop: d 2cps under nu} we cannot use such identities. This is due to the factor of, 
    $$q^{\frac{k_d(k_d+1)}{2}}$$
    in the summand. If we write $k_d$ as, 
    $$k_d=k_1+\sum\limits_{j=2}^d\hat{k}_j+d-1,$$
    we see that $\frac{k_d(k_d+1)}{2}$  gives cross terms such as $\hat{k}_i\hat{k}_j$ and $k_1\hat{k}_j$. So these identities cannot help us here.
\end{remark}
\par \vspace{3mm}In Corollary \ref{cor: 1 displacement} we gave the distribution of the displacement of the second class particle from its ground state position in the case of ASEP with a single second class particle. Now we can give the joint distribution of displacements of $d$ second class particles from their positions in the ground state. If the ASEP with $d$ second class particles is distributed according to the ergodic measure (i.e.\ $\underline{\xi}\sim\underline{\nu}^n$) then in the ground state, $\underline{\eta}^n\in\Omega^n$, the second class particles are at sites $n+1,n+2,\cdots, n+d$. We will denote the displacement of the $j$-th second particle from it's position in the ground state by $\underline{D}=(D_1,\cdots,D_d)$, and note that $D_j=X_j-(n+j)$ for each $j\in\{1,\cdots,d\}$. We note that the second class particles stay ordered which implies that the displacements are non-decreasing, $D_1\leq D_2\leq\cdots\leq D_d$.
\begin{cor}\label{cor: d displacements}
    Assume that the coupled system is stationary and $\underline{\xi}\sim\underline{\nu}^n$. For any $\underline{\ell}\in\mathbb{Z}^d$, such that $\ell_1\leq \ell_2\leq \cdots \leq \ell_d$, the probability that the displacements of the $d$ second class particles, $\underline{D}$, is $\underline{\ell}$ is, 
    \begin{align*}
        \mathbb{P}_{\underline{\nu}^n}(\underline{D}=\underline{\ell})=q^{-\sum\limits_{i=1}^d\ell_i-\frac{d(d+1)}{2}}(q;q)_\infty\prod\limits_{i=1}^d(1-q^i)\sum\limits_{\substack{\underline{k}\in Z_+^d:\\\forall j\in\{1,2,\cdots,d\},\\ k_j\geq \ell_j+j-1}} &\frac{q^{\frac{k_1(k_1+1)}{2}+k_1(d-1-\ell_1)+\frac{k_d(k_d+1)}{2}+k_d+1}}{(q;q)_{k_d-\ell_d-d+1}(q;q)_{k_1}}\\
        &\quad \cdot\prod\limits_{j=2}^d q^{\hat{k}_j(d+1-2j-\ell_j)+\frac{\hat{k}_j(\hat{k}_j+1)}{2}}\cdot\begin{bmatrix}
            \ell_j-\ell_{j-1}\\
            \hat{k}_j       
            \end{bmatrix}_q,
    \end{align*}
    where $\hat{k}_j\defeq k_j-k_{j-1}-1$ for each $j\in\{2,\cdots,d\}$.
\end{cor}

\begin{proof}
In the ground state $\underline{\nu}^n$ the position of the $j$-th second class particle is $n+j$, for each $j\in\{1,\cdots,d\}$. Thus we have that,
    $$\mathbb{P}_{\underline{\nu}^n}(\underline{D}=\underline{\ell})=\mathbb{P}_{\underline{\nu}^n}(\underline{X}=\underline{m}),$$
    where $m_j\defeq\ell_j+n+j$ for each $j\in\{1,\cdots,d\}$. The result then follows form Proposition \ref{prop: d 2cps under nu}.
\end{proof}
\begin{remark}
    From our previous discussion (Remark \ref{rem: displacement}) we see that Corollary \ref{cor: d displacements} should be linked to [\cite{gnedin_olshanski}, Theorem 6.1] the probability of the joint displacement of $d$ integers under the two-sided infinite Mallows measure on permutations of $\mathbb{Z}$ (in the case where $D_1\leq\cdots\leq D_d$). That is for any $d>0$ and integers $\ell_1\leq \ell_2\leq \cdots \leq\ell_d$,
    $$\mathbb{P}_{\mathcal{Q}}(\underline{D}=\underline{\ell})=(1-q)^dq^{-\frac{d(d+1)}{2}}(q;q)_\infty\prod\limits_{j=2}^d(q;q)_{\ell_j-\ell_{j-1}}\sum\limits_{\substack{0\leq a_1,\cdots,a_d,b_1\cdots,b_d<\infty:\\
    \forall j\in\{1,\cdots,d\},\\
    \sum\limits_{k=1}^jb_k-\sum\limits_{k=j}^da_j=\ell_j}}\frac{q^{\sum\limits_{1\leq i\leq j\leq d}(b_i+1)(a_j+1)}}{(q;q)_{b_1}\cdots(q;q)_{b_d}(q;q)_{a_1}\cdots(q;q)_{a_d}}.$$
    
    \par As before we can equate a Mallows distributed permutation of $\mathbb{Z}$ with an ASEP state with $d$ second class particles distributed as $\underline{\nu}^n$. This is done by setting the integers $i\leq n$ to be holes, the integers $i\geq n+d+1$ to be first class particles and the integers $i\in\{n+1,\cdots, n+d\}$ to be second class particles. 
    \par As we have defined the ASEP with $d$ second class particles, the second class particles remain ordered (i.e.\ $X_1<X_2<\cdots<X_d$) and so the displacements of them from their position in the ground state are also ordered ($D_1\leq D_2\leq \cdots \leq D_d$). However, with Mallows permutations of $\mathbb{Z}$ it is not true that the displacements of consecutive integers have to be ordered in this way. Thus, we see that Corollary \ref{cor: d displacements} and [\cite{gnedin_olshanski}, Theorem 6.1] will not match exactly. In fact under the Mallows measure we need to consider all possible permutations (e.g.\ if $\ell_i=\ell_{i+1}$ then the permutation that only swaps $i$ and $i+1$ need not be considered) of the displacements $\ell_1,\cdots,\ell_d$. That is, for any $\ell_1\leq\cdots\leq\ell_d$,
    $$\mathbb{P}_{\underline{\nu}^n}(\underline{D}=\underline{\ell})=\mathbb{P}_{\mathcal{Q}}(\underline{D}=\underline{\ell})+\sum\limits_{\substack{\sigma:\{1,\cdots,d\}\stackrel{1:1}{\longrightarrow}\{1,\cdots,d\}:\\
    \sigma(\underline{\ell})\neq \underline{\ell}}}\mathbb{P}_{\mathcal{Q}}(\underline{D}=\sigma(\underline{\ell})),$$
    where the formula of [\cite{gnedin_olshanski}, Theorem 6.1] gives $\mathbb{P}_{\mathcal{Q}}(\underline{D}=\underline{\ell})$. However as remarked in \cite{gnedin_olshanski} the general case can be handled by introducing an additional factor of $q^{\text{inv}(d_1+1,\cdots,d_d+d)}$, thus all terms in the above sum can be found. 
\end{remark}
\section{Combinatorial identities and their meaning} \label{sect: combinatorics}
In Section \ref{subsect: identities} we gave probabilistic proofs to three classical combinatorial identities. Here we discuss their combinatorial meaning.
\par First, by considering the particle-hole symmetry of ASEP (Corollary \ref{cor: c change between holes and particles in ASEP}) we proved the Durfee Rectangles Identity. 
\durfee*
\par Let us first consider the case when $n=0$, then the identity is, 
$$\frac{1}{(q;q)_\infty}=\sum\limits_{k=0}^\infty\frac{q^{k^2}}{(q;q)_k^2}.$$
This is a well known identity, namely the identity for Durfee squares of integer partitions (for example see Chapter 8 in Andrews and Eriksson \cite{integer_partitions}). For a given integer partition the Durfee square is the largest square (anchored in the upper left-hand corner) in its Ferrers diagram. Equivalently, a given partition with $\ell$ parts, $(\lambda_1\geq\lambda_2\geq...\geq\lambda_\ell)$, has a Durfee square of side length $k\leq \ell$ if $\lambda_k\geq k$ and $\lambda_{k+1}\leq k$. With this definition, any integer partition decomposes into 3 pieces: its unique Durfee square of side length $k$, a partition into up $k$ parts to its right, and a partition with parts at most size $k$ underneath (for example see Figure \ref{fig:durfee square} below). We denote this decomposition in the following way, 
$$\lambda=(\lambda_1,\lambda_2,...,\lambda_\ell)\equiv\{k^2,\lambda^{(r)},\lambda^{(d)}\}$$
where $\lambda^{(r)}=(\lambda_1-k,...,\lambda_k-k)$ is the partition to the right of the Durfee square and $\lambda^{(d)}=(\lambda_{k+1},...,\lambda_{\ell})$ the partition underneath the Durfee square.
\par From this decomposition we see where the RHS of the identity comes from. Consider one term in the sum, say $\frac{q^{k^2}}{(q;q)_k^2}$, for some $k$. This is the generating function for integer partitions whose Durfee square is of side length $k$:
\begin{itemize}
\item $q^{k^2}$ corresponds to the Durfee square.
\item $\frac{1}{(q;q)_k}$ is the generating function for partitions into up to $k$ parts (see Section 6.2 in Andrews and Eriksson \cite{integer_partitions}) and so corresponds to the partition to the right of the Durfee square.
\item By conjugation $\frac{1}{(q;q)_k}$ is also the generating function for partitions into parts of size at most $k$ (see Section 6.2 in Andrews and Eriksson \cite{integer_partitions}) so this corresponds to partition underneath the Durfee square. 
\end{itemize}
The sum over all $k$ considers Durfee squares of any size and so we recover all integer partitions, for which the generating function is $\frac{1}{(q;q)_\infty}$, hence the identity holds.
\begin{figure}[H]
\centering
\begin{tikzpicture}[scale=0.7]
\draw[red] (-0.25,3.25)--(1.25,3.25);
\draw[red] (-0.25,1.75)--(1.25,1.75);
\draw[red] (-0.25,1.75)--(-0.25,3.25);
\draw[red] (1.25,1.75)--(1.25,3.25);
\filldraw [black] (0,3) circle (2pt);
\filldraw [black] (0.5,3) circle (2pt);
\filldraw [black] (1,3) circle (2pt);
\filldraw [black] (1.5,3) circle (2pt);
\filldraw [black] (2,3) circle (2pt);
\filldraw [black] (2.5,3) circle (2pt);
\filldraw [black] (3,3) circle (2pt);
\filldraw [black] (3.5,3) circle (2pt);
\filldraw [black] (0,2.5) circle (2pt);
\filldraw [black] (0.5,2.5) circle (2pt);
\filldraw [black] (1,2.5) circle (2pt);
\filldraw [black] (1.5,2.5) circle (2pt);
\filldraw [black] (2,2.5) circle (2pt);
\filldraw [black] (2.5,2.5) circle (2pt);
\filldraw [black] (3,2.5) circle (2pt);
\filldraw [black] (3.5,2.5) circle (2pt);
\filldraw [black] (0,2) circle (2pt);
\filldraw [black] (0.5,2) circle (2pt);
\filldraw [black] (1,2) circle (2pt);
\filldraw [black] (1.5,2) circle (2pt);
\filldraw [black] (2,2) circle (2pt);
\filldraw [black] (2.5,2) circle (2pt);
\filldraw [black] (3,2) circle (2pt);
\filldraw [black] (0,1.5) circle (2pt);
\filldraw [black] (0.5,1.5) circle (2pt);
\filldraw [black] (1,1.5) circle (2pt);
\filldraw [black] (0,1) circle (2pt);
\filldraw [black] (0.5,1) circle (2pt);
\filldraw [black] (0,0.5) circle (2pt);
\filldraw [black] (0,0) circle (2pt);
\node(a) at (1.5,-1.5) [label=$\lambda$\equal (8\comma 8\comma 7\comma 3\comma 2\comma 1\comma 1)]{};
\node(a) at (5,1) [label=$\Longleftrightarrow$]{};
\node(a) at (9,0.8) [label=$\lambda\equiv \{3^2\comma (5\comma5\comma4)\comma(3\comma2\comma1\comma1)\}$]{};
\end{tikzpicture}
    \caption{An example of the Durfee square for an integer partition, $\lambda$, of $30$.}
    \label{fig:durfee square}
\end{figure}
\par Now let us consider what the identity means combinatorially for any fixed $n\in\mathbb{Z}$. The identity now states that, for $n$ fixed, every integer partition has a unique Durfee rectangle, a generalisation of the notion of Durfee square (see for example equation (4) in Gessel \cite{Ges_durfee_rectangle}). For some fixed $n$ we say that a given partition with $\ell$ parts, $(\lambda_1\geq\lambda_2\geq...\geq\lambda_\ell)$ has a Durfee rectangle of side lengths $n+k$ and $k$ if $\lambda_{k}\geq n+k$ and $\lambda_{k+1}\leq n+k$. From this definition we see that for fixed $n$ each integer partition has a unique Durfee rectangle (just as it had a unique Durfee square i.e.\ the case when $n=0$). Now with this definition, for each fixed $n$, any integer partition decomposes into 3 pieces: its unique Durfee rectangle of side lengths $k$ and $n+k$, a partition into up to $k$ parts to its right, and a partition with parts at most size $n+k$ underneath. We denote this decomposition in the following way,
$$\lambda=(\lambda_1,\lambda_2,...,\lambda_\ell)\equiv\{k*(n+k),\lambda^{(n,r)},\lambda^{(n,d)}\}$$
where $\lambda^{(n,r)}=(\lambda_1-(n+k),...,\lambda_k-(n+k)$ is the partition to the right of the Durfee rectangle and $\lambda^{(n,d)}=(\lambda_{k+1},...,\lambda_\ell)$ the partition underneath the Durfee rectangle. See Figure \ref{fig: durfee rectangle} below for an example of the Durfee rectangle inside a given integer partition for different values of $n$.
\par Now we can see where the RHS of the identity comes from. Consider one term in the sum, say $\frac{q^{k(n+k)}}{(q;q)_k(q;q)_{n+k}}$, for some $k$. This is the generating function for integer partitions with Durfee rectangle of side lengths $k$ and $n+k$:
\begin{itemize}
    \item $q^{k(n+k)}$ corresponds to the Durfee rectangle.
    \item $\frac{1}{(q;q)_k}$ is the generating function for partitions into up to $k$ parts and so corresponds to the partition to the right of the Durfee rectangle.
    \item By conjugation $\frac{1}{(q;q)_{n+k}}$ is the generating function for partitions with parts of size at most $n+k$ and so corresponds to the partition underneath the Durfee rectangle.
\end{itemize}
\par Summing these generating functions for all $k$ from $\max\{-n,0\}$ to $\infty$ (with $n\in\mathbb{Z}$ fixed) we get all integer partitions, and thus the identity holds. 
\begin{figure}[H]
\centering
\begin{tikzpicture}[scale=0.7]
\draw[red] (-0.25,3.25)--(2.25,3.25);
\draw[red] (-0.25,1.75)--(2.25,1.75);
\draw[red] (-0.25,1.75)--(-0.25,3.25);
\draw[red] (2.25,1.75)--(2.25,3.25);
\filldraw [black] (0,3) circle (2pt);
\filldraw [black] (0.5,3) circle (2pt);
\filldraw [black] (1,3) circle (2pt);
\filldraw [black] (1.5,3) circle (2pt);
\filldraw [black] (2,3) circle (2pt);
\filldraw [black] (2.5,3) circle (2pt);
\filldraw [black] (3,3) circle (2pt);
\filldraw [black] (3.5,3) circle (2pt);
\filldraw [black] (0,2.5) circle (2pt);
\filldraw [black] (0.5,2.5) circle (2pt);
\filldraw [black] (1,2.5) circle (2pt);
\filldraw [black] (1.5,2.5) circle (2pt);
\filldraw [black] (2,2.5) circle (2pt);
\filldraw [black] (2.5,2.5) circle (2pt);
\filldraw [black] (3,2.5) circle (2pt);
\filldraw [black] (3.5,2.5) circle (2pt);
\filldraw [black] (0,2) circle (2pt);
\filldraw [black] (0.5,2) circle (2pt);
\filldraw [black] (1,2) circle (2pt);
\filldraw [black] (1.5,2) circle (2pt);
\filldraw [black] (2,2) circle (2pt);
\filldraw [black] (2.5,2) circle (2pt);
\filldraw [black] (3,2) circle (2pt);
\filldraw [black] (0,1.5) circle (2pt);
\filldraw [black] (0.5,1.5) circle (2pt);
\filldraw [black] (1,1.5) circle (2pt);
\filldraw [black] (0,1) circle (2pt);
\filldraw [black] (0.5,1) circle (2pt);
\filldraw [black] (0,0.5) circle (2pt);
\filldraw [black] (0,0) circle (2pt);
\node(a) at (1.5,-1.5) [label=$\lambda\equal (8\comma 8\comma 7\comma 3\comma 2\comma 1\comma 1)$]{};
\node(a)at(2.7,-2.25)[label=$\equiv \{3*5 \comma (3\comma3\comma2)\comma(3\comma2\comma1\comma1)\}$]{};
\node(a) at (1.5, 3.5) [label=If $n\equal2$\comma \hspace{0.5mm} then here $k\equal 3$]{};
\draw[red] (9.75,3.25)--(10.75,3.25);
\draw[red] (9.75,0.75)--(10.75,0.75);
\draw[red] (9.75,0.75)--(9.75,3.25);
\draw[red] (10.75,0.75)--(10.75,3.25);
\filldraw [black] (10,3) circle (2pt);
\filldraw [black] (10.5,3) circle (2pt);
\filldraw [black] (11,3) circle (2pt);
\filldraw [black] (11.5,3) circle (2pt);
\filldraw [black] (12,3) circle (2pt);
\filldraw [black] (12.5,3) circle (2pt);
\filldraw [black] (13,3) circle (2pt);
\filldraw [black] (13.5,3) circle (2pt);
\filldraw [black] (10,2.5) circle (2pt);
\filldraw [black] (10.5,2.5) circle (2pt);
\filldraw [black] (11,2.5) circle (2pt);
\filldraw [black] (11.5,2.5) circle (2pt);
\filldraw [black] (12,2.5) circle (2pt);
\filldraw [black] (12.5,2.5) circle (2pt);
\filldraw [black] (13,2.5) circle (2pt);
\filldraw [black] (13.5,2.5) circle (2pt);
\filldraw [black] (10,2) circle (2pt);
\filldraw [black] (10.5,2) circle (2pt);
\filldraw [black] (11,2) circle (2pt);
\filldraw [black] (11.5,2) circle (2pt);
\filldraw [black] (12,2) circle (2pt);
\filldraw [black] (12.5,2) circle (2pt);
\filldraw [black] (13,2) circle (2pt);
\filldraw [black] (10,1.5) circle (2pt);
\filldraw [black] (10.5,1.5) circle (2pt);
\filldraw [black] (11,1.5) circle (2pt);
\filldraw [black] (10,1) circle (2pt);
\filldraw [black] (10.5,1) circle (2pt);
\filldraw [black] (10,0.5) circle (2pt);
\filldraw [black] (10,0) circle (2pt);
\node(a) at (11.5,-1.5) [label=$\lambda\equal (8\comma 8\comma 7\comma 3\comma 2\comma 1\comma 1)$]{};
\node(a)at(12.5,-2.25)[label=$\equiv \{5*2 \comma (6\comma6\comma5\comma1)\comma(1\comma1)\}$]{};
\node(a) at (11.5, 3.5) [label=If $n\equal -3$\comma \hspace{0.5mm} then here $k\equal 5$]{};
\end{tikzpicture}
    \caption{An example of the Durfee rectangle for an integer partition, $\lambda$, of $30$ when $n=2$ and $n=-3$.}
    \label{fig: durfee rectangle}
\end{figure}
\vspace{5mm}\par By considering the distribution of a site under $\underline{\nu}^n$ (Lemma \ref{lem: site under nu n}) we proved the following identity.
\crankiden*
Combinatorially this identity arises by considering all integer partitions depending on whether their crank is greater than or less than $n$. Crank (or Dyson's crank) is a quantity of any integer partition that was first discussed by Dyson in 1944 \cite{dyson} and then was formally defined by Andrews and Garvan in 1988, \cite{andrews_garvan}. Given an integer partition $\lambda$ let $l(\lambda)$ denote the largest part of $\lambda$, $w(\lambda)$ denote the number of parts of size 1 in $\lambda$, and $\mu(\lambda)$ denote the number of parts of $\lambda$ that are larger than $w(\lambda)$. The the crank, $c(\lambda)$ is defined to be
$$c(\lambda)=\begin{cases}
    l(\lambda) &\text{if } w(\lambda)=0,\\
    \mu(\lambda)-w(\lambda) &\text{if } w(\lambda)>0.
\end{cases}$$
\begin{example}\label{eg: crank}~
    \begin{enumerate}[(a)]
        \item Consider the partition $\lambda_1=(7,5,5,2)$.
        \vspace{1mm}\begin{center}
            \begin{tikzpicture}[scale=0.7]
                \filldraw [black] (0,3) circle (2pt);
                \filldraw [black] (0.5,3) circle (2pt);
                \filldraw [black] (1,3) circle (2pt);
                \filldraw [black] (1.5,3) circle (2pt);
                \filldraw [black] (2,3) circle (2pt);
                \filldraw [black] (2.5,3) circle (2pt);
                \filldraw [black] (3,3) circle (2pt);
                \filldraw [black] (0,2.5) circle (2pt);
                \filldraw [black] (0.5,2.5) circle (2pt);
                \filldraw [black] (1,2.5) circle (2pt);
                \filldraw [black] (1.5,2.5) circle (2pt);
                \filldraw [black] (2,2.5) circle (2pt);
                \filldraw [black] (0,2) circle (2pt);
                \filldraw [black] (0.5,2) circle (2pt);
                \filldraw [black] (1,2) circle (2pt);
                \filldraw [black] (1.5,2) circle (2pt);
                \filldraw [black] (2,2) circle (2pt);
                \filldraw [black] (0,1.5) circle (2pt);
                \filldraw [black] (0.5,1.5) circle (2pt);
            \end{tikzpicture}
        \end{center}
        We have that,
            $$l(\lambda_1)=7, \quad \quad w(\lambda_1)=0, \quad \quad \text{and}, \quad \quad  \mu(\lambda_1)=4.$$
            Thus the crank of $\lambda_1$ is $c(\lambda_1)=7$.
        
        \vspace{2mm}\item Consider the partition $\lambda_2=(6,4,4,2,1,1)$.
        \vspace{1mm}\begin{center}
            \begin{tikzpicture}[scale=0.7]
                \filldraw [black] (0,3) circle (2pt);
                \filldraw [black] (0.5,3) circle (2pt);
                \filldraw [black] (1,3) circle (2pt);
                \filldraw [black] (1.5,3) circle (2pt);
                \filldraw [black] (2,3) circle (2pt);
                \filldraw [black] (2.5,3) circle (2pt);
                \filldraw [black] (0,2.5) circle (2pt);
                \filldraw [black] (0.5,2.5) circle (2pt);
                \filldraw [black] (1,2.5) circle (2pt);
                \filldraw [black] (1.5,2.5) circle (2pt);
                \filldraw [black] (0,2) circle (2pt);
                \filldraw [black] (0.5,2) circle (2pt);
                \filldraw [black] (1,2) circle (2pt);
                \filldraw [black] (1.5,2) circle (2pt);
                \filldraw [black] (0,1.5) circle (2pt);
                \filldraw [black] (0.5,1.5) circle (2pt);
                \filldraw [black] (0,1) circle (2pt);
                \filldraw [black] (0,0.5) circle (2pt);
            \end{tikzpicture}
        \end{center}
        We have that,
            $$l(\lambda_2)=6, \quad \quad w(\lambda_2)=2, \quad \quad \text{and}, \quad \quad  \mu(\lambda_2)=3.$$
            Thus the crank of $\lambda_2$ is $c(\lambda_2)=1$.
        
        \vspace{2mm}\item Consider the partition $\lambda_3=(5,3,2,1,1,1)$.
        \vspace{1mm}\begin{center}
            \begin{tikzpicture}[scale=0.7]
                \filldraw [black] (0,3) circle (2pt);
                \filldraw [black] (0.5,3) circle (2pt);
                \filldraw [black] (1,3) circle (2pt);
                \filldraw [black] (1.5,3) circle (2pt);
                \filldraw [black] (2,3) circle (2pt);
                \filldraw [black] (0,2.5) circle (2pt);
                \filldraw [black] (0.5,2.5) circle (2pt);
                \filldraw [black] (1,2.5) circle (2pt);
                \filldraw [black] (0,2) circle (2pt);
                \filldraw [black] (0.5,2) circle (2pt);
                \filldraw [black] (0,1.5) circle (2pt);
                \filldraw [black] (0,1) circle (2pt);
                \filldraw [black] (0,0.5) circle (2pt);
            \end{tikzpicture}
        \end{center}
        We have that,
            $$l(\lambda_3)=5, \quad \quad w(\lambda_3)=3, \quad \quad \text{and}, \quad \quad  \mu(\lambda_3)=1.$$
            Thus the crank of $\lambda_3$ is $c(\lambda_3)=-2$.
    \end{enumerate}
\end{example}
In \cite {andrews_garvan}, Andrews and Garvan give the generating function for integer partitions with a given crank value say $m\in\mathbb{Z}$. In 2022, Hopkins, Sellers and Yee, [\cite{hopkins_sellers}, Theorem 2.1], gave the generating function for integer partitions with crank$\geq n\geq 0$, by considering Durfee rectangles. This generating function is given by, 
$$\sum\limits_{k\geq 0} \frac{q^{(k+1)(n+k)}}{(q;q)_k(q;q)_{n+k}}.$$
Also it is well-known that integer partitions with crank$\leq -n$ are in one to one correspondence with integer partitions of crank $\geq n$ (for any positive $n$) and so they have the same generating functions.
\par Consider the identity in Proposition \ref{prop: identitiy related to durfee} for $n=1$, that is,
$$\sum\limits_{k=0}^\infty\frac{q^{(k+1)(k+1)}}{(q;q)_k(q;q)_{k+1}}+\sum\limits_{k=0}^\infty\frac{q^{(k+1)k}}{(q;q)_k(q;q)_k}=\frac{1}{(q;q)_\infty}.$$
We know that the right hand side is the generating function for all integer partitions. The left hand side is splitting integer partitions depending on whether they have positive or negative crank. In particular the first sum is the generating function for partitions with crank$\geq 1$ or equivalently $\leq -1$ and the second sum is the generating function for partitions with crank$\geq 0$.
\par The general statement for any $n\in\mathbb{Z}$ similarly splits integer partitions in terms of whether the crank is greater than or less than or equal to $n$. 
\vspace{5mm}\par By considering the distribution of the number number of particles to the left of some given site in ASEP (Theorem \ref{thm: half infinite asep particles}) we proved Euler's identity.
\euler*
Combinatorially this identity gives two ways of writing the 2-variable generating function for integer partitions into distinct parts. A general term is of the form $d_{n,k}q^nz^k$ where $d_{n,k}$ gives the number of partitions of $n$ into exactly $k$ distinct parts. The product on the LHS of the identity clearly counts partitions into distinct parts (with a $z$ in front of $q$ to count the number of non-zero parts). For the sum on the RHS, $k$ gives the number of parts and it is well known that $\frac{q^{\frac{k(k+1)}{2}}}{(q;q)_k}$ is the generating function for partitions into exactly $k$ distinct parts (see for example Section 2 of Cimpoea\c{s}'s 2022 paper \cite{cimp_exactly_distinct}). In Appendix \ref{appendix: combinatorics} we give an alternate proof of Theorem \ref{thm: half infinite asep particles} by considering particles states with $k$ particles to the left of some site $m$ as partitions of some $s$ (left jumps away from a ground state) into up to $k$ parts. Instead we can think of a state that has $k$ particles to the left of $m$ as a partition of some $n$ where each particle denotes a part of size how far this particle is from the boundary site $m$. Since there can only be one particle per site these distances must be distinct and thus we have a partition into exactly $k$ distinct parts. This explains why Euler's identity arises a consequence of Theorem \ref{thm: half infinite asep particles}.
\par\vspace{5mm} Similarly, by considering the distribution of the number of particles in between two given sites in ASEP (Lemma \ref{lem: finite asep particles}) we proved the $q$-Binomial Theorem.
\qbin*
Combinatorially this identity gives two ways of writing the 2-variable generating function for integer partitions into distinct parts of size at most $m$. It is known that the one-variable generating function for partitions into distinct parts is given by $\prod\limits_{i=1}^\infty(1+q^i)$, clearly if we truncate this product at $i=m$ we have distinct parts of size a most $m$ (see for example Equation 5.4 of Andrews and Eriksson \cite{integer_partitions}). If we then write $\prod\limits_{i=1}^m(1+zq^i)$ this is the two-variable generating function for these partitions with the power of $z$ counting the exact number of parts used. Now for the sum on the RHS, as before $k$ gives the number of parts. It is known that $\begin{bmatrix}
m\\
k
\end{bmatrix}_{q}$ is the generating function for partitions into up to $k$ part of size up to $m-k$ (for example see Section 7.2 of Andrews and Eriksson \cite{integer_partitions}). If we add a triangle of side length $k$ to the Young diagram of such a partition we now have a partition with distinct parts of size up to $m$, as desired. Note that adding this triangle is the same as multiplying the generating function by a factor of $q^{\frac{k(k+1)}{2}}$. We note that if we take $m$ to infinity in the $q$-Binomial Theorem we get Euler's Identity. In Appendix \ref{appendix: combinatorics} we give an alternative proof of Lemma \ref{lem: finite asep particles} by considering particle states with $0\leq k\leq m_2-m_1-1$ particles inside $(m_1,m_2)$ as partitions of some $s$ (left jumps away from a ground state) into up $k$ parts of size up to $m_2-m_1-1-k$. Instead we can think of the states as partitions into exactly $k$ parts with size up to $m\defeq m_2-m_1-1$ by letting each particle represent a part of size its distance away from the right boundary site $m_2$. Again we see that these distances must be distinct due to the exclusion rule and so the parts in the corresponding partition are distinct. This explains why the $q$-Binomial Theorem arises as a consequence of Lemma \ref{lem: finite asep particles}.

\section{Future directions and open questions}\label{open questions}
\par As we saw in Section \ref{subsect: identities}, natural questions for the particle system can lead to probabilistic proofs of well-known combinatorial identities. The authors are exploring further directions.
\par In the paper of Bal\'azs, Fretwell and Jay \cite{MDJ}, it is shown that states of blocking particle systems with $0$-$1$-\dots-$k$ (for $k\geq 2$) particles per site correspond to generalised Frobenius partitions (these are a generalisation of the notion of integer partitions). With this in mind, perhaps by considering the number of particles to the left of some site or within a finite range or symmetry between holes and particles for these more general blocking systems we would see GFP versions of the Euler, $q$-Binomial and Durfee rectangles identities.
\par Recently Amir, Bahadoran, Busani and Saada, \cite{amir_bahadoran_busani_saada}, characterised the blocking measures for asymmetric simple exclusion on multiple lanes. There they comment that it would be interesting to see what combinatorial identities can be found by studying equivalences between the multilane exclusion and other particle systems like in the work of Bal\'azs and Bowen, \cite{blocking}, and that of Bal\'azs, Fretwell and Jay, \cite{MDJ}. It is natural to ask what is the distribution of particles in half infinite and finite ranges in the multilane exclusion under the natural blocking measure. These distributions should also lead to interesting combinatorial identities as we have shown here for the single lane ASEP (Section \ref{subsect: identities}).
\par Both directions are subjects of upcoming papers.
\newpage\bibliographystyle{plain}
\bibliography{refsJJ}
\newpage \appendix
\section{Simple $q$-Pochhammer manipulations}
\begin{lem}\label{lem: pochammer relation for nu^n calcs}
    For any $m\in\mathbb{Z}$ and $c\in\mathbb{R}$,
    $$\frac{q^{\frac{m(m+1)}{2}-mc}(-q^{1+m-c};q)_\infty(-q^{c-m};q)_\infty}{(-q^{1-c};q)_\infty(-q^{c};q)_\infty}=1.$$ 
\end{lem}

\begin{proof}
   Clearly this holds for $m=0$, so let's consider the cases when $m>0$ and $m<0$ separately. 
   \par\noindent Firstly when $m>0$ we have that,
    \begin{align*}
        \frac{q^{\frac{m(m+1)}{2}-mc}(-q^{1+m-c};q)_\infty(-q^{c-m};q)_\infty}{(-q^{1-c};q)_\infty(-q^{c};q)_\infty}&= q^{\frac{m(m+1)}{2}-mc}\cdot\frac{\prod\limits_{i=-m}^{-1}(1+q^{i+c})}{\prod\limits_{i=1}^m(1+q^{i-c})}\\
        &=q^{\frac{m(m+1)}{2}-mc}\cdot\frac{\prod\limits_{i=1}^{m}q^{c-i}(1+q^{i-c})}{\prod\limits_{i=1}^m(1+q^{i-c})}=1.
    \end{align*}
    Similarly when $m<0$. We suppose $m=-l$ for some $l\in\mathbb{Z}_{>0}$ and so,
    \begin{align*}
        \frac{q^{\frac{-l(-l+1)}{2}+lc}(-q^{1-l-c};q)_\infty(-q^{c+l};q)_\infty}{(-q^{1-c};q)_\infty(-q^{c};q)_\infty}&=q^{\frac{-l(-l+1)}{2}+lc}\cdot \frac{\prod\limits_{i=-l+1}^0(1+q^{i-c})}{\prod\limits_{i=0}^{l-1}(1+q^{i+c})}\\
        &=q^{\frac{(l-1)l}{2}+lc}\cdot\frac{\prod\limits_{i=0}^{l-1}q^{-i-c}(1+q^{i+c})}{\prod\limits_{i=0}^{l-1}(1+q^{i+c})}=1.
    \end{align*} 
\end{proof}
\begin{lem}\emph{\textbf{$q$-Binomial analogue of Pascal's identity} (e.g.\ (7.1) in Andrews and Erikkson \cite{integer_partitions})}\label{lem: q-pascal}
\par For any $m\in\mathbb Z_{>0}$ and $k\in\{0,\dots,m\}$,
$$\begin{bmatrix}
    m\\
    k\\
\end{bmatrix}_q=q^k\begin{bmatrix}
    m-1\\k
\end{bmatrix}_q+\begin{bmatrix}
    m-1\\
    k-1
\end{bmatrix}_q.$$    
\end{lem}
\begin{proof}
 Let us consider,
 \begin{align*}
    q^k\begin{bmatrix}
    m-1\\k
\end{bmatrix}_q+\begin{bmatrix}
    m-1\\
    k-1
\end{bmatrix}_q&= q^k\frac{(q;q)_{m-1}}{(q;q)_k(q;q)_{m-1-k}}+\frac{(q;q)_{m-1}}{(q;q)_{k-1}(q;q)_{m-k}} \\
&=q^k\frac{(q;q)_{m-1}(1-q^{m-k})}{(q;q)_k(q;q)_{m-k}}+\frac{(q;q)_{m-1}(1-q^k)}{(q;q)_k(q;q)_{m-k}}\\
&=\frac{(q;q)_{m-1}(1-q^m)}{(q;q)_k(q;q)_{m-k}}=\frac{(q;q)_m}{(q;q)_k(q;q)_{m-k}}=\begin{bmatrix}
    m\\
    k
\end{bmatrix}_q.
 \end{align*}
\end{proof}
\vspace{-3mm}\begin{lem}\label{lem: pochammer relation}
For any $k\in\mathbb Z_{\ge0}$,
     $$(q^{-k};q)_k=(q^{-1};q^{-1})_k=\frac{(q;q)_k}{(-1)^kq^{\frac{k(k+1)}{2}}}. $$  
\end{lem}
\begin{proof}
  We have that, 
    $$(q^{-k};q)_k=\prod\limits_{j=0}^{k-1}(1-q^{-(k-j)})=\prod\limits_{i=1}^k(1-q^{-i})=(q^{-1};q^{-1})_k.$$
    $$(q^{-1};q^{-1})_k=\prod\limits_{i=1}^k(1-q^{-i})=(1-q^{-1})(1-q^{-2})\cdots(1-q^{-k})=q^{-\sum\limits_{i=1}^k i}\prod\limits_{j=1}^k(q^i-1)= \frac{(q;q)_k}{(-1)^kq^{\frac{k(k+1)}{2}}}.$$   
\end{proof}
\begin{lem}\label{lem: m_j pochammers}
For any $d\in\mathbb Z_{>0}$, any $m_1<m_2<...<m_d \in \mathbb Z$, any $c\in\mathbb{R}$, and for each $j \in \{2,..,d\}$,
$$\frac{(-q^{c+d+2-j-m_j};q)_{\hat{m}_j}}{(-q^{c+d-j-m_j};q)_\infty}=\frac{1}{(1+q^{c+d-j-m_j})(1+q^{c+d+1-j-m_j})(-q^{c+d-(j-1)-m_{j-1}};q)_\infty}$$
where $\hat{m}_j\defeq m_j-m_{j-1}-1$
\end{lem}
\begin{proof}
\vspace{-1mm}\begin{align*}
 \frac{(-q^{c+d+2-j-m_j};q)_{\hat{m}_j}}{(-q^{c+d-j-m_j};q)_\infty}&=\frac{\prod\limits_{i=1}^{\hat{m}_j}(1+q^{c+d+1-j-m_j+i})}{\prod\limits_{i=0}^\infty(1+q^{c+d-j-m_j+i})}\\
 &=\frac{\prod\limits_{i=1}^{\hat{m}_j}(1+q^{c+d+1-j-m_j+i})}{(1+q^{c+d-j-m_j})(1+q^{c+d+1-j-m_j})\prod\limits_{i=1}^\infty(1+q^{c+d+1-j-m_j+i})}\\
 &=\frac{1}{(1+q^{c+d-j-m_j})(1+q^{c+d+1-j-m_j})\prod\limits_{i=\hat{m}_j+1}^\infty(1+q^{c+d+1-j-m_j+i})}\\
 &=\frac{1}{(1+q^{c+d-j-m_j})(1+q^{c+d+1-j-m_j})\prod\limits_{i=0}^\infty(1+q^{c+d+1-j-m_{j-1}+i})}\\
 &=\frac{1}{(1+q^{c+d-j-m_j})(1+q^{c+d+1-j-m_j})(-q^{c+d-(j-1)-m_{j-1}};q)_\infty}
\end{align*}\vspace{-2mm}
\end{proof}
\section{Combinatorial Arguments}\label{appendix: combinatorics}
This section is included for completeness. Notice that no other parts of the paper rely on these arguments. 
\begin{proof}[Alternative proof of Theorem \ref{thm: half infinite asep particles}]~
\par
We notice that, any state $\underline{z}$ such that $N_{m+\frac{1}{2}}^{\overset{\leftarrow}{p}}(\underline{z})=k$ is some $s$ number of left jumps away from the state $\tilde{\underline{\omega}}$, which is defined on the sites $(-\infty,m]$ as, 
$$\tilde{\omega}_i=\begin{cases}
 1 &\text{ if } i \in \{m-k+1,...,m\}\\
 0 &\text{ o/w. }
\end{cases}$$
 We see that a left particle jump changes the probability of a state, under $\underline{\mu}^c$, by a factor of $q$,
$$\underline{\mu}^c(\underline{z}^{(j,j-1)})=\prod\limits_{i\in\mathbb{Z}}\frac{q^{-(i-c)z_i^{(j,j-1)}}}{1+q^{-(i-c)}}=q^{j-c}\cdot q^{-(j-1-c)}\cdot\underline{\mu}^c(\underline{z})=q\cdot\underline{\mu}^c(\underline{z}).$$
 \noindent Thus we have that, 
$$\underline{\mu}^c(\{N_{m+\frac{1}{2}}^{\overset{\leftarrow}{p}}(\underline{z})=k\})=\underline{\mu}^c|_{(-\infty,m]}(\tilde{\underline{\omega}})\sum\limits_{s=0}^\infty q^s \hat{g}(k,s),$$
where $\hat{g}(k,s)$ counts the number of states on $(-\infty,m]$ that are $s$ left jumps away from $\underline{
\tilde{\omega}}$. A state that is $s$ left jumps away from $\underline{\tilde{\omega}}$ can be viewed as a partition of $s$ where the parts are the number of left jumps each particle is away from its position in $\underline{\tilde{\omega}}$. So the partitions that these states correspond to are partitions of $s$ with up to $k$ parts. Then $\sum\limits_{s=0}^\infty q^s \hat{g}(k,s)$ is the generating function for such partitions, by conjugation of partitions we know that $\sum\limits_{s=0}^\infty q^s \hat{g}(k,s)=\frac{1}{(q;q)_k}$ (for example see Section 6.2 of Andrews and Eriksson \cite{integer_partitions}). Now we find, 
$$\underline{\mu}^c|_{(-\infty,m]}(\tilde{\underline{\omega}}) = \prod\limits_{i=-\infty}^m \mu^c_i(\tilde{\omega}_i)=\frac{\prod\limits_{i=m-k+1}^mq^{c-i}}{\prod\limits_{i=-\infty}^m(1+q^{c-i})}=\frac{q^{k(c-m)+\frac{k(k-1)}{2}}}{\prod\limits_{i=0}^\infty(1+q^{c-m+i})}=\frac{q^{k(c-m)+\frac{k(k-1)}{2}}}{(-q^{c-m};q)_{\infty}}.$$

Putting this together we have that, 
$$\underline{\mu}^c(\{N_{m+\frac{1}{2}}^{\overset{\leftarrow}{p}}(\underline{z})=k\})=\frac{q^{k(c-m)+\frac{k(k-1)}{2}}}{(q;q)_k(-q^{c-m};q)_\infty}.$$
\end{proof}
\begin{proof}[Alternative proof of Lemma \ref{lem: finite asep particles}]~
\par
\par \noindent For some $k \in \{0,...,\hat{m}_2\}$ define a state $\underline{\omega}$ on $[m_1+1,m_2-1]$ such that,
$$\omega_i=\begin{cases}
1 &\text{for } i \in \{m_2-k,...,m_2-1\}\\
0 &\text{for } i \in \{m_1+1,...,m_2-k-1\}.
\end{cases}$$
We see that for any state $\underline{z}$ defined on $[m_1+1,m_2-1]$ there is some minimal number of right particle jumps needed to get to $\underline{\omega}$; that is $\underline{z}$ is some $s$ left jumps away from $\underline{\omega}$.
Thus, since a left jump changes the measure by a factor of $q$,
$$\underline{\mu}^c(\{N^{\overset{\leftarrow}{p}}_{m_2-1+\frac{1}{2}}(\underline{z})-N^{\overset{\leftarrow}{p}}_{m_1+\frac{1}{2}}(\underline{z})=k\})=\underline{\mu}^c|_{[m_1+1,m_2-1]}(\underline{\omega})\sum\limits_{s=0}^{k(\hat{m}_2-k)}q^sg(k,s)$$
where $g(k,s)$ counts the number of states on $[m_1+1,m_2-1]$ that are $s$ jumps away from $\underline{\omega}$. The probability under $\underline{\mu}^c|_{[m_1+1,m_2-1]}$ of the state $\underline{\omega}$ is, 
$$
\underline{\mu}^c|_{[m_1+1,m_2-1]}(\underline{\omega})=\prod\limits_{i=m_1+1}^{m_2-1}\mu^c_i(\omega_i)=\frac{\prod\limits_{i=m_2-k}^{m_2-1}q^{c-i}}{\prod\limits_{i=m_1+1}^{m_2-1}(1+q^{c-i})}=\frac{q^{k(c+1-m_2)+\sum\limits_{i=1}^{k-1} i}}{\prod\limits_{i=1}^{\hat{m}_2}(1+q^{c-m_2+i})}=\frac{q^{k(c+1-m_2)+\frac{k(k-1)}{2}}}{(-q^{c-m_2+1};q)_{\hat{m}_2}}.
$$
 A state that is $s$ left jumps away from $\underline{\omega}$ can be viewed as a partition of $s$ where the parts are the number of left jumps each particle is away from its position in $\underline{\omega}$. So the partitions which these states correspond to are partitions of $s$ with up to $k$ parts of size up to $\hat{m}_2-k$. Hence $\sum\limits_{s=0}^{k(\hat{m}_2-k)}q^sg(k,s)$ is the generating function for these such partitions which can be written in the following way (for example see Section 7.2 of Andrews and Eriksson \cite{integer_partitions}),
$$\sum\limits_{s=0}^{k(\hat{m}_2-k)}q^sg(k,s)=\begin{bmatrix}
    \hat{m}_2 \\
    k
    \end{bmatrix}_{q} = \frac{(q;q)_{\hat{m}_2}}{(q;q)_k(q;q)_{\hat{m}_2-k}}
$$
Putting all this together we find that, $$\underline{\mu}^c(\{N^{\overset{\leftarrow}{p}}_{m_2-1+\frac{1}{2}}(\underline{z})-N^{\overset{\leftarrow}{p}}_{m_1+\frac{1}{2}}(\underline{z})=k\})=\frac{q^{k(c+1-m_2)+\frac{k(k-1)}{2}}(q;q)_{\hat{m}_2}}{(-q^{c-m_2+1};q)_{\hat{m}_2}(q;q)_k(q;q)_{\hat{m}_2-k}}$$
\end{proof}
\end{document}